\documentclass[a4paper, 11pt, reqno]{amsart}
\usepackage{hyperref}
\usepackage[active]{srcltx}
\usepackage[curve]{xypic}
\usepackage{graphicx}
\usepackage{mathrsfs}
\usepackage{longtable}
\usepackage{rotating}
\usepackage{multirow}
\usepackage{titletoc} %设置目录格式
\usepackage{tikz}
\usepackage[active]{srcltx}
\usepackage{epsfig,amsmath}
\usepackage[english]{babel}
\usepackage{amsmath,amsfonts,amssymb,amscd,color,epsfig,amsthm}%eufrak

\newtheorem{theorem}{Theorem}[section]
\newtheorem{lemma}{Lemma}[section]
\newtheorem{proposition}[lemma]{Proposition}
\newtheorem{corollary}[lemma]{Corollary}
\newtheorem{definition}{Definition}[section]
\newtheorem{remark}{Remark}[section]

\newtheorem{example}{\bf Example}[section]
\newtheorem{assumption}{\bf Assumption}[section]

\newtheorem{facts}[theorem]{\bf Fact}

\numberwithin{equation}{section}

\newcommand{\R}{\mathbb{R}}
\newcommand{\Z}{\mathbb{Z}}

\newcommand{\wt}{\widetilde}
\newcommand{\mb}{\mathbb}
\newcommand{\mc}{\mathcal}
\newcommand{\sm}{\setminus}
\newcommand{\tu}{\textup}
\newcommand{\ol}{\overline}
\newcommand{\es}{\emptyset}

\newcommand{\tb}{\textbf}

\newcommand{\wh}{\widehat}
\newcommand{\olC}{\widehat{\mathbb{C}}}

\def\EEE{{\cal E}}

\def\g{\gamma}
\def\G{\Gamma}

\def\de{\delta}

\def\sm{\setminus}
\def\R{\mbox{$\mathbb R$}}

\def\Z{\mbox{$\mathbb Z$}}

\def\lv{ \left(\begin{matrix} }
 \def\rv{\end{matrix}\right)}

\def\g{\gamma}

\def\cal{\mathcal}

\def\dw{{\dw}}

\newcommand{\mylabel}[1]{\label{#1}}

\newcommand{\REFEQN}[1] { \begin{equation}\mylabel{#1} }
\newcommand{\ENDEQN}{\end{equation}}
\newcommand{\REFTHM}[1] { \begin{theorem}\mylabel{#1} }
\newcommand{\ENDTHM}{\end{theorem}}
\newcommand{\REFNTH}[1] { \begin{newthm}\mylabel{#1} }
\newcommand{\ENDNTH}{\end{newthm}}
\newcommand{\REFPROP}[1]{\begin{proposition}\mylabel{#1} }
\newcommand{\ENDPROP}{\end{proposition} }
\newcommand{\REFLEM}[1]{\begin{lemma}\mylabel{#1} }
\newcommand{\ENDLEM}{\end{lemma} }
\newcommand{\REFCOR}[1]{\begin{corollary}\mylabel{#1} }
\newcommand{\ENDCOR}{\end{corollary} }

\def\ov{\overline}

\title{Invariant graphs of rational maps}
\author{Guizhen Cui}
\address{Guizhen Cui, College of Mathematics and Statistics, Shenzhen University and Academy of Mathematics and Systems Science, Chinese Academy of Sciences}
\email{gzcui@math.ac.cn}
\author{Yan Gao}
\address{Yan Gao, College of Mathematics and Statistics, Shenzhen University, Shenzhen, 518060, P.R.China.}
\email{gyan@scu.edu.cn}
\author{Jinsong Zeng}
\address{Jinsong Zeng, School of Mathematics and Information Science, Guangzhou University, Guangzhou, 510006, P.R. China.}
\email{jinsongzeng@163.com}
\begin{document}

\maketitle
\begin{abstract}
We prove that every postcritically finite rational map $f:\olC\to\olC$ admits an $f^n$-invariant finite and connected graph containing the postcritical set of $f$ as $n$ is large enough.
\end{abstract}
\begin{quote}
\end{quote}

\tableofcontents

%\tableofcontents

\section{Introduction}
Let $f$ be a rational map of the Riemann sphere $\wh{\mb{C}}$ with $\deg f>1$. We denote by $\tu{Crit}(f)$ the set of the critical points of $f$. The \emph{postcritical set} of $f$ is defined by 
$$
\tu{Post}(f)=\ol{\bigcup_{n\geq 1}f^n(\tu{Crit}(f))}.
$$
The map $f$ is called \emph{postcritically finite} if $\tu{Post}(f)$ is a finite set. 

For postcritically finite polynomials, Douady-Hubbard have introduced so-called Hubbard trees to capture their dynamical features \cite{DH2}. A long-standing problem is to develop analogous combinatorial invariants for general rational maps \cite[Problem 5.5]{Mc1}.

Here we recall some results on this topic. For quadratic rational maps, this problem was intensively studied in \cite{Ber, Ree1,ST, Wit}. There has also been some progress for Newton's methods; see \cite{DMRS,LMS1,LMS2,WYZ}. In \cite{De,QWY}, the authors constructed invariant cut rays for McMullen maps. Applying invariant graphs, the authors provided a classification of critically fixed rational maps \cite{CGNPP,Hlu}.

%All works mentioned above study particular classes of rational maps.
The icebreaking general result on this problem is due to Bonk-Meyer \cite{BM} and Cannon-Floyd-Parry \cite{CFP}, where they proved that, any postcritically finite rational map $f$ with empty Fatou set admits an $f^n$-invariant Jordan curve containing $\tu{Post}(f)$ for sufficiently large integer $n>0$. Actually, this result was established in a broader setting of \emph{expanding Thurston maps} \cite[Definition 2.2]{BM}. An open problem is whether a postcritically finite rational map with non-empty Fatou set admits $f^n$-invariant graphs containing its postcritical set (\cite[Problem 3]{BM} and \cite[Question 4.2]{CFP}).  %Here a set $S$ is said to be \emph{$f^n$-invariant} if $f^n(S)\subseteq S$.

When $f$ is a \emph{carpet rational map}, i.e., a postcritically finite rational map whose Julia set is a Sierpi\'{n}ski carpet, it was proved in \cite{GHMZ} that for sufficiently large integer $n$, there is an $f^n$-invariant Jordan curve containing $\tu{Post}(f)$. M. Rees provided a similar result but with a different approach in \cite{Ree2}.

It is not hard to check that there are postcritically finite rational maps $f$ that admit no $f^n$-invariant Jordan curves for any integer $n>0$; see Remark \ref{rmk:example}\,(2) for a counter example. Instead of searching for invariant Jordan curves for a general postcritically finite rational map $f$, the main purpose of this article is to construct $f^n$-invariant graphs containing $\tu{Post}(f)$. % More precisely,
%One may wonder whether there exist $f^n$-invariant Jordan curves for general postcritically finite rational maps. This is not true (see Remark \ref{rmk:example}\,(2)). In this work, we intend to construct invariant graphs for general postcritically finite rational maps.
%, which are called \emph{PCF} for simplicity.
Our main result is as follows.

\begin{theorem}\label{thm:main}
Let $f:\olC\to \olC$ be a postcritically finite rational map with non-empty Fatou set. Let $P\subseteq \wh{\mb{C}}$ be a finite set containing $\tu{Post}(f)$ such that $f(P)\subseteq P$. Then for sufficiently large integer $n>0$, there exists a finite and connected graph $G\subseteq \olC$ containing $P$ such that $f^n(G)\subseteq G$.
\end{theorem}

In the setting of Thurston maps, we believe that Theorem \ref{thm:main} should be true for \emph{B\"{o}ttcher expanding maps}. One may refer to \cite{BD,BD2,FPP} for the definition and related theories on B\"{o}ttcher expanding maps.

Let $G$ be an $f^n$-invariant graph in Theorem \ref{thm:main}. Then the set $S:=G\cup f^{-1}(G)\cup\cdots\cup f^{-n+1}(G)$ is $f$-invariant. But $S$ may not be a finite graph, since it possibly has infinitely many complementary components. The problem of whether there exists an $f$-invariant graph containing $\tu{Post}(f)$ is still open, even for rational expanding Thurston maps and carpet rational maps.% See also \cite[Problem 2]{BM} for expanding Thurston maps.

The framework of the proof of Theorem \ref{thm:main} is an analogue of \cite[Theorem 15.1]{BM}. Beginning with an initial graph $G_0$ under certain conditions, we will find a graph $G_1$ in $f^{-n}(G_0)$ for sufficiently large integer $n>0$ such that $G_1$ is close to $G_0$ and isotopic to $G_0$ rel $P$. Such a graph $G_0$ is called \emph{isotopically invariant}. Then by isotopic lifting, we obtain a sequence of graphs $G_k$. Finally, we show that the sequence of graphs $G_k$ converges to a graph $G_\infty$ in the sense of Hausdorff topology. The limit $G:=G_\infty$ is indeed as required.

There are two major differences between the proof of Theorem \ref{thm:main} and that of \cite[Theorem 15.1]{BM}. Firstly, in order to find an isotopically invariant graph, we propose some conditions on the initial graph $G_0$, and then manage to construct graphs satisfying these conditions (see Definition \ref{def:admissible} and Theorem \ref{thm:admissible}]). In contrast, for expanding Thurston maps, it was shown that the initial Jordan curve can be taken arbitrarily.  On the other hand, since the map in our case is not globally expanding, we give a new and direct argument to show that the sequence of the graphs $G_k$ converges to an invariant graph in the sense of Hausdorff topology. 

%The argument in \cite[Theorem 5.1]{BM} strongly depends on the expansionary of the map on $\wh{\mb{C}}$. Since the map in our case is not expanding on $\wh{\mb{C}}$, some essential difficulties appear. The main differences are as follows. In the proof of \cite[Theorem 15.1]{BM}, starting from an arbitrary Jordan curve with vertex set $P$, its isotopically invariant Jordan curve always exists. But this point possibly fails for a general map $f$ and an initial graph. Therefore, the first part of this paper is to find suitable conditions for the initial graph $G_0$ such that the first step can be established. The key properties about $G_0$ are proposed as  the joint condition for graphs; see Definitions . This restricts the behavior of an edge of a graph approaching boundary points of Fatou domains,  namely \emph{accessible points}, and limits how distinct edges of a graph link to a vertex in $\mc{J}_f$.

%Another significant difference between our argument and that in \cite[Theorem 15.1]{BM} is the proof of the final step as stated above, i.e., the convergence of the graphs $G_k$. In \cite{BM} they use the ``combinatorial expanding'' property of expanding Thurston maps (with respect to any Jordan curve containing $P$), while this property does not hold in our case. Instead, we will give a new and direct argument to show that the limit exists and is actually a graph.

The paper is organized as follows.
In \S \ref{sec_pre}, we introduce some basic notation and results to be used later on in this paper. In \S \ref{sec_admissible}, we first give the definition of admissible graph and state some results related to admissible graphs. After that, a proof of Theorem \ref{thm:main} is given by assuming Theorems \ref{thm:admissible}, \ref{thm:multi} and \ref{thm:invariant}. In \S \ref{sec_arcs}, we define and investigate two kinds of arcs satisfying the {\it joint condition}: clean arcs and $\Omega$-clean arcs. \S \ref{sec_existence} is devoted to proving Theorem \ref{thm:admissible}, that is, the existence of admissible graphs.
In \S \ref{sec_tiles} we develop the tiles for admissible graphs and then use them to complete the proof of Theorem \ref{thm:multi}. In \S \ref{sec_homotopy}, we show that non-trivial edges of an admissible graph $G$ can be arbitrarily approximated by arcs in $f^{-n}(G)$ for large enough integer $n>0$; see Proposition \ref{prop:epsilon_neighbor}. In \S \ref{sec_invariant}, we first show that the union of trivial edges of an admissible graph is invariant after a small perturbation; see Proposition \ref{prop:new}. Then combining Proposition \ref{prop:epsilon_neighbor} with Propostion \ref{prop:new}, we prove Theorem \ref{thm:invariant}.

\noindent\tb{Notation}: We will use the following notation frequently.

		$\bullet$ $\wh{\mb{C}}$, ${\mb{C}}$, $\mb{R}$ and $\mb{D}$ are the Riemann sphere, the complex plane, the real axis and the unit disk, respectively. We write $\mb{D}(z,r):=\{\xi\in\mb{C}:|\xi-z|<r\}$ and $\mb{D}_{r}:=\mb{D}(0,r)$.% The set of rational numbers is denoted by $\mb{N}$.
		
		$\bullet$ Let $A$ be a set in $\wh{\mb{C}}$. The closure, the boundary and the interior of $A$ are denoted by $\ol{A}$, $\partial A$ and $\tu{int}(A)$ respectively. We denote by $\tu{Comp}(A)$ the collection of all connected components of $A$. The cardinal number of $A$ is $\#A$.
		
		$\bullet$ Two sets satisfying $A\Subset B$ means that $\ol{A}\subseteq B$.
		
		$\bullet$ By an \emph{arc} we mean a continuous injection $\gamma$ from the closed interval $[0,1]$ into $\wh{\mb{C}}$. We write $\tu{end}(\gamma):=\{\gamma(0),\gamma(1)\}$ and $\tu{int}(\gamma):=\gamma\setminus\tu{end}(\gamma)$. For $z\neq w\in\gamma$, we denote by $]z,w[_\gamma$, $[z,w]_\gamma$ and $[z,w[_\gamma$ (or $]z,w]_\g$) the open, closed and semi-open segments in $\gamma$ between $z$ and $w$, respectively.
		
		$\bullet$ Let $A$ be a non-empty subset of $\wh{\mb{C}}$. The diameter of $A$ is $\tu{diam}\,A:=\tu{sup\,}$$\{\tu{dist\,}(x,y),x\neq y\in A\}$, where $\tu{dist}\,(\cdot,\cdot)$ is the standard spherical metric. Given $\epsilon>0$, the set $\mc{N}(A, \epsilon):=\{z\in\wh{\mb{C}},\tu{dist}\,(z,A)<\epsilon\}$ is an $\epsilon$-neighborhood of $A$.
		%\item[$\bullet$] Denote $d_H(A_1,A_2)$ by the Haussdorff distance between two subsets $A_1,A_2$ of $\wh{\mb{C}}$.
		
	%	$\bullet$ The Julia set and Fatou set of a rational map $f$ are denoted by $\mc{J}_f$ and $\mc{F}_f$ respectively.\vskip 1cm
%The collection of periodic Fatou domains is denoted by $\mc{U}_f$.
%	\end{itemize}
%\end{notation}
\vskip 0.3cm
\noindent\emph{Acknowledgment.} We sincerely thank the anonymous referees for careful reading of this paper and many valuable suggestions, including providing us a brief proof of Lemma \ref{lemma:intersection_three_component}. The first author was supported by the NSFC under grant No.11688101 and no.12071303 and Key Research Program of Frontier Sciences, CAS, under grant No. QYZDJ-SSW-SYS 005. The second author was supported by the NSFC under grant No.11871354. The third author was supported by the NSFC under grant No.11801106. Jinsong Zeng is the corresponding author.

\section{Preliminaries}\label{sec_pre}
Let $f:\olC\to\olC$ be a rational map with $\deg f>1$. Its \emph{Fatou set} $\mc{F}_f$ is defined to be the set of points such that the sequence of forward iterates $\{f^n\}$ forms a normal family in a neighborhood of the point. The \emph{Julia set} $\mc{J}_f$ is the complement of the Fatou set. Both of them are completely invariant under $f$.

For simplicity, a postcritically finite rational map with non-empty Fatou set is written as PCF in this paper. Here we summarize some results on Fatou and Julia sets which will be used in this paper; refer to \cite{DH2}, \cite{Mc2} and \cite{Mi1} for the proofs.

\begin{lemma}\label{lem:new}%\label{lemma:Dynamics_Components}
	Let $f$ be a PCF. Then the following statements hold.
	\begin{itemize}
		\item[(1)] The Julia set $\mc{J}_f$ is a locally connected non-trivial continuum with empty interior. Consequently, the boundary of every Fatou domain is locally connected and the number of Fatou domains whose diameters are greater than a given positive number is finite; moreover, for a given Fatou domain $U$, the components of $\olC\setminus\ol{U}$ are Jordan domains with their boundaries in $\partial U$ and diameters tending to zero.
		 \item[(2)] (Sullivan's no wandering domains theorem) Every Fatou domain is eventually periodic under the iterations of $f$.
		 \item[(3)]	There exists a family $\mc{B}_f:=\{(U,\phi_U)\}_{U\in\tu{Comp}(\mc{F}_f)}$, called a \emph{system of B\"{o}ttcher coordinates} for $f$, such that each mapping $\phi_U:U\to \mb{D}$ is conformal and $\phi_U(z)^{\tu{deg}(f|_U)}=\phi_{f(U)}(f(z))\ \forall \,z\in U$. By Carath\'{e}odory's theorem, $\phi^{-1}_U$ can be continuously extended to a surjection $\phi^{-1}_U:\ol{\mb{D}}\to \ol{U}$. We call $c(U):=\phi_U^{-1}(0)$ the \emph{center} of $U$.
		 \item[(4)] The center of every periodic Fatou domain is contained in a superattracting cycle and is therefore postcritical. So there are only finitely many periodic Fatou domains.
	\end{itemize}
\end{lemma}
In general, the system of B\"{o}ttcher coordinates for a PCF is not unique. Once fixed, the pullback of a radial ray, i.e., $$R_U(\theta):=\phi^{-1}_U(\{re^{2\pi\tb{i}\theta}:0\leq r<1\})$$ is called the \emph{(internal) ray} of angle $\theta$ in $U$. According to Lemma \ref{lem:new}\,(3), $R_U(\theta)$ always \emph{lands}, i.e., $\lim\limits_{r\to 1^{-}}\phi^{-1}_U(re^{2\pi \tb{i}\theta})$ exists. By a \emph{closed internal ray} we mean the union of an internal ray and its landing point. %Moreover, it holds that
%$$\theta\emph{ is rational}\Leftrightarrow\emph{$R_U(\theta)$ is eventually periodic}\Leftrightarrow\emph{its landing point is eventually periodic.}$$
The circle $\phi_U^{-1}(\{re^{2 \pi\tb{i}\theta},0\leq \theta\leq 1\})$ for an $r\in(0,1)$ is called an \emph{equipotential curve} of $U$.

\begin{lemma}\label{lemma:intersection_three_component}
	Let $U_i, i=1,2,3$, be three distinct Fatou domains. Then the intersection $\bigcap_{i=1,2,3}\partial U_i$ contains at most two points.
\end{lemma}
\begin{proof}
	Suppose that $\partial U_1\cap\partial U_2$ contains three distinct points $z_1, z_2$ and $z_3$. Let $R_1, R_2$ and $R_3$ (resp. $R_1', R_2'$ and $R_3'$) be internal rays in $U_1$ (resp. $U_2$), which land at $z_1, z_2$ and $z_3$. The complement of $R_1\cup R_2\cup R_3\cup R_1'\cup R_2'\cup R_3'\cup\{z_1,z_2,z_3\}$ consists of three Jordan domains. One of these Jordan domains must contain $U_3$. So $\partial U_3$ can contain at most two of the points $z_1, z_2$ and $z_3$. The proof is complete.
\end{proof}
\subsection{Accessible points}
Let $f$ be a PCF and $z$ be a point in $\mc{J}_f$. If the point $z$ is disjoint from the boundary of every Fatou domain, then $z$ is called \emph{buried}. Otherwise, it lies in the boundary of a Fatou domain and thus it is the landing point of an internal ray by Lemma \ref{lem:new}\,(3). In this case, the point $z$ is called \emph{accessible}.
%A point in $\olC$ is called a \emph{Julia point} or a \emph{Julia-type point} (resp. a \emph{Fatou point} or a \emph{Fatou-type point}) with respect to $f$ if it is contained in $\mc{J}_f$ (resp. $\mc{F}_f$).
%A point in $\mc{J}_f$ is called \emph{buried} if it is disjoint from the boundary of any Fatou domain. According to Lemma \ref{lem:new}, a point in the boundary of a Fatou domain is the landing point of an internal ray in $U$; then we call it \emph{accessible}. When $f$ is a polynomial, all points in $\mc{J}_f$ are accessible. %Let $z$ be a Julia point.
% A point in $\olC$ is called a \emph{Julia point} or a \emph{Julia-type point} (resp. a \emph{Fatou point} or a \emph{Fatou-type point}) with respect to $f$ if it is contained in $\mc{J}_f$ (resp. $\mc{F}_f$). A Julia point is called \emph{accessible} if it belongs to the boundary of a Fatou domain of $f$, and called \emph{buried} otherwise. It is clear that the set of accessible (resp. buried) points are completely invariant under $f$. Accessible points are available when going along internal rays. When $f$ is a polynomial, all the points in $\mc{J}_f$ are accessible. Let $z$ be a Julia point.
We introduce the following notation.

\begin{itemize}
	\item[ ] $\tu{Dom}(z):$ the family of all Fatou domains whose boundaries contain the point $z$;
	\item[ ]$\tu{Acc}(z):$ the family of all internal rays landing at the point $z$;
	\item[ ]  $\tu{Acc}(z,U):$ the family of all internal rays within a Fatou domain $U$ landing at $z$.
\end{itemize}

In particular, the point $z$ is called \emph{multi-accessible} if $\#\tu{Acc}(z)\geq 3$.
%Clearly $\tu{Dom}(z)=\tu{Acc}(z)=\emptyset$ if and only if $z$ is buried.
For simplicity, we write $f(\mc{S}):=\{f(S): S\in\mc{S}\}$ when $\mc{S}$ is a family of subsets of $\olC$.

\begin{lemma}\label{lema:facts}
	Let $f$ be a PCF and $z$ be an accessible point. Then
	\begin{itemize}
		\item[(1)] $f(\tu{Dom}(z))=\tu{Dom}(f(z))$ and so $\#\tu{Dom}(z)\geq \#\tu{Dom}(f(z))$;
		\item [(2)] $f(\tu{Acc}(z,U))\subseteq\tu{Acc}(f(z),f(U))$;
		\item[(3)]$f(\tu{Acc}(z))=\tu{Acc}(f(z))$; %If $\#\tu{Acc}(z)<\infty$, then $\#\tu{Acc}(z)=\delta_z\cdot\#\tu{Acc}(f(z))$, where $\delta_z$ is the local degree of $f$ at $z$;
	%	\item[(4)] If $z$ is periodic, then $\#\tu{Dom}(z)<\infty$ and $\#\tu{Acc}(z)<\infty$; moreover, everyone in $\tu{Dom}(z)$ and $\#\tu{Acc}(z)$ is periodic;
		%\item[(4)] $\#\tu{Dom}(z)<\infty$. Moreover, when $z$ is periodic, then $\#\tu{Acc}(z)<\infty$ and the elements in $\tu{Acc}(z)$ has the same period; the elements in $\tu{Dom}(z)$ has the same period. But the two periods need not be equal;\color{black}
	\item[(4)]Let $U$ be a Fatou domain and $K_U:=\wh{\mb{C}}\setminus U$. If $z\in\partial U$, then we have $$\#\tu{Comp}(K_U\setminus\{z\})=\#\tu{Acc}(z,U);$$
	\item[(5)] $\#\tu{Dom}(z)\leq \#\tu{Acc}(z)<\infty$;
	\item[(6)] If the point $z$ is periodic, then every domain in $\tu{Dom}(z)$ is periodic.
	\end{itemize}
\end{lemma}

\begin{proof}
	Statement (1) follows from the fact that the Fatou set is completely invariant. Statement (2) holds directly by definition.
For each internal ray $R$ landing at $f(z)$, there are exactly $\delta_z$ internal rays in $f^{-1}(R)$ terminating at $z$, where $\delta_z$ is the local degree of $f$ at $z$. Thus statement (3) follows. For the proof of statement (4) we refer to \cite[Theorem 6.6 p.85]{Mc3}.

For statement (5), we argue by contradiction and assume $\#\tu{Acc}(z)=\infty$. Note that each ray in $\tu{Acc}(z)$ eventually falls into a periodic Fatou domain by Lemma \ref{lem:new}\,(2). Since periodic Fatou domains are finitely many, then for an arbitrarily large integer $N$ there exists a Fatou domain $U_0$ of period $p\geq 1$ and an integer $n_0\geq 1$ such that $f^{n_0}(z)\in \partial U_0$ and $\#\tu{Acc}(f^{n_0}(z), U_0)\geq N$. Since $f$ is postcritically finite, the point $f^n(z)$ is not critical for each $n$.

If the point $z$ is eventually periodic, according to \cite[Lemma 2.3]{FPP}, we obtain a finite number $$N_0:=\tu{max}\,\{\#\tu{Acc}(f^i(z), U): i\geq 0\tu{ and } U\tu{ is a periodic Fatou domain.}\}$$
Thus $N\leq N_0$. It is a contradiction.

Otherwise, the point $z$ is wandering, i.e., $f^i(z)\neq f^j(z)$ for all $i\neq j\geq 0$. By Lemma \ref{lem:new}\,(3) the circle map $P_d: z\mapsto z^d$ on $\partial \mb{D}$ provides a combinatorial model for the map $f^p:\ol{U_0}\to\ol{U_0}$, where $d:=\tu{deg}(f^p|_{U_0})$.
We choose $N$ rays in $\tu{Acc}(f^{n_0}(z), U_0)$. Their union is denoted by $S_0$. 

Then $S_i:=f^{ip}(S_0)$ is a union of $N$ rays in $\tu{Acc}(f^{ip+n_0}(z), U_0)$ satisfying that $$\ol{S_i}\cap \ol{S_j}=\{c(U_0)\}$$ for every $i\neq j$. For each $i$, the closure of the image $\phi_{U_0}(S_i)$ meets $\partial \mb{D}$ in $N$ points. Let $T_i$ be the union of these points. Then we have $T_i=P_d^{i}(T_0)$, $\# T_i=N$ and for every $i\neq j$ the two sets $T_i$ and $T_j$ are unlinked. Thus $T_0$ forms a \emph{wandering polygon} for $P_d$; see \cite[p.28]{C}. According to \cite[Theorem 1.3]{C}, it follows that $N\leq d-1$. Again, this is a contradiction. 

Statement (6) is a consequence of statements (1) and (5). The proof of the lemma is complete.
\end{proof}

\subsection{Shrinking lemmas}

Let $f$ be a PCF. For any such map $f$, there exists an \emph{orbifold metric} $d_{\mc{O}}$ on $\mc{O}:=\wh{\mb{C}}\sm(\tu{Post}(f)\cap\mc{F}_f)$,
%, which is the complement of the postcritical points of Fatou-type,
such that for any $w\in\mc{O}$ and $z\in f^{-1}(w)$, it satisfies $||f'(z)||_\mc{O}>1$; see \cite[Section 19]{Mi1} as well as \cite[Appendix A.10]{BM}.
	
Now we fix a sufficiently small closed topological disk $B_z$ containing $z$ such that $f(B_z)\Subset B_{f(z)}$, for any postcritical point $z$ of \emph{Fatou-type}, i.e., $z\in\mc{F}_f$. Let $V$ be the complement of the union of these disks. Thus, we have $f^{-1}(V)\Subset V\Subset\mc{O}$, and there exists a constant $\lambda>1$ such that $||f'(z)||_\mc{O}\geq \lambda$ for all $z\in f^{-1}(V)$. Moreover, we have the following locally shrinking lemma; see \cite[Lemma 2.3]{GHMZ}.

\begin{lemma}[Locally shrinking]\label{lemma:shrinking_lemma_locally}
	Let $f$ be a PCF, and $V$ be a domain as defined above. Then there are constants $\delta>0$ and $C>0$ such that for any $n\geq 1$, any compact set $E\subseteq V$ with $\tu{diam}(E)\leq \delta$, and any connected component $E_n$ of $f^{-n}(E)$, we have $$\tu{diam}_{\mc{O}}E_n\leq C\tu{diam}_{\mc{O}}E/\lambda^n.$$
\end{lemma}

We state here the well-known lemma; one may refer to \cite[Section 12.1]{LM} for the proof.
\begin{lemma}[Globally shrinking]\label{lemma:shrinking_lemma}
	Let $f$ be a PCF. Let $W\subseteq\wh{\mb{C}}$ be a domain such that for each $n\geq 1$ and $W_n\in\tu{Comp}(f^{-n}(W))$ the degree of the mapping $f^n:W_n\to W$
	is bounded above by a constant. Then
	$$\tu{max}\,\left\{\tu{diam}\,E_n, E_n\in\tu{Comp}\,(f^{-n}(E))\right\}\to 0\tu{ as }n\to\infty$$
	for any set $E$ compactly contained in $W$.
\end{lemma}

\subsection{Graphs on the sphere}\label{def:graph}

	In this paper, by a \emph{graph} we mean a connected and finite graph. Precisely, a graph $G$ is a continuum in $\wh{\mb{C}}$, which can be written as a union of finitely many arcs, called \emph{edges}, with their interiors mutually disjoint. Endpoints of edges are called \emph{vertices} of $G$. We sometimes denote a graph by $G=(\mc{V},\mc{E})$, where $\mc{V}$ and $\mc{E}$ are the collections of vertices and edges, respectively.% For a vertex $v$ of $G$, we denote by $\mc{E}(v)$ the collections of edges of $G$ containing $v$.
		
	Let $G=(\mc{V}, \mc{E})$ be a graph. The \emph{valence} of a vertex $v$ in $\mc{V}$, written as $\tu{valence}(v,G)$, is the number of edges in $\mc{E}$ containing $v$. The \emph{valence} of a point $x\in G\setminus\mc{V}$ is defined as $\tu{valence}(x, G):=2$. A point in $G$ with valence equal to $1$ (resp. greater than 2) is called an \emph{endpoint} (resp. \emph{branch point}) of $G$. The graph $G$ is a \emph{tree} if its complement is connected.
	
	A graph has many different choices for the set of vertices. When not stated obviously, we always take $\mc{V}$ to be the minimal vertex set, the set consisting of endpoints and branch points of $G$. Moreover, we require the edges considered in this paper having distinct endpoints.
	
	A point $x\in G$ is called a \emph{cut} (resp. \emph{non-cut}) point of a graph $G$, provided that the set $G\sm\{x\}$ is disconnected (resp. connected). Obviously,
	$G$ has no cut points $\Leftrightarrow$ $\tu{Comp}(\wh{\mb{C}}\setminus G)$ consists of Jordan domains.

\section{Admissible graphs and proof of Theorem \ref{thm:main}}\label{sec_admissible}
In this section we introduce the notion of admissible graphs,  and prove Theorem \ref{thm:main} by assuming Theorem \ref{thm:admissible}, Theorem \ref{thm:multi} and Theorem \ref{thm:invariant}. Their proofs will be established in \S \ref{sec_existence}, \S \ref{sec_tiles} and \S \ref{sec_invariant}, respectively.

Let $f$ be a PCF and $z$ be an accessible point in $\mc{J}_f$.
Let $B_z$ be a small open disk around $z$ such that the intersection of $\ol{B}_z$ with each ray in $\tu{Acc}(z)$ is an arc terminating at $z$. The rays in $\tu{Acc}(z)$ divide $B_z\sm\{z\}$ into $\#\tu{Acc}(z)$ components, each of which is called an \emph{entrance} of the point $z$ associated to $B_z$. %
By definition, an entrance $E$ is an open set bounded by three arcs:
two are segments from the rays in $\tu{Acc}(z)$, say $R_1$ and $R_2$;
the remaining one is from $\partial B_z$.
The rays $R_1$ and $R_2$ are said to be \emph{adjacent} associated to $E$. The trivial case is that
$$R_1=R_2\Leftrightarrow \tu{Acc}(z)=\{R_1\}\Leftrightarrow\#\tu{Acc}(z)=1.$$
An entrance is called \emph{special} if its associated adjacent rays $R_1$ and $R_2$ (allowing $R_1=R_2$) are from a common Fatou domain; otherwise, it is called \emph{ordinary}.

	\tu{A graph $G=(\mc{V}, \mc{E})$ is called \emph{regulated} if the intersection of $G$ with the Fatou set is a union of internal rays. Obviously, each endpoint or branch point of a regulated graph is either a Fatou center or a point in $\mc{J}_f$. For simplicity, we always require that $\mc{V}\cap \mc{F}_f$ consists of Fatou centers for regulated graphs.}

\begin{definition}[Joint conditions]\label{def:joint}
	\label{def:obs}	\tu{Let $G=(\mc{V}, \mc{E})$ be a regulated graph and $z$ be an accessible point in $G$. Let $B_z$ be a small open disk around $z$ such that the components of $B_z\cap R$ with $R\in\tu{Acc}(z)$ and $B_z\cap G\setminus\{z\}$ are disjoint open arcs incident to the point $z$. 	
	Then $G$ is said to satisfy the \emph{joint condition} at the point $z$, provided that every entrance $E$ associated to $B_z$ intersects at most one component of $B_z\cap G\setminus\{z\}$ and the intersection is non-empty only if $E$ is special; see Figure \ref{fig:link-condition}.
	We say that $G$ satisfies the \emph{joint condition} if
\begin{itemize}
	\item[(1)] the joint condition holds at each accessible point in $G$, and
	\item[(2)] every buried point in $\mc{V}\cap \mc{J}_f$ is an endpoint.
\end{itemize}}
	%	\tu{(1) the joint condition holds at each accessible point in $G$, and}
	%	\tu{(2) the valence of any buried Julia-type vertex $v$ in $\mc{V}$, i.e., $v\in\mc{J}_f\cap \mc{V}$ buried, is equal to one.}
%	Let $\gamma$ be a regulated arc and $z$ be an accessible point in $\gamma$ (possibly $z\in\tu{end}(\gamma))$. We say $\gamma$ satisfies the so-called \emph{unique visible condition}, provided that, for every accessible point $z$ in $\gamma$ (possibly $z\in\tu{end}(\gamma)$),  when a person stands in the arc $\gamma$ close enough to the point $z$, he can see just one Fatou domain in $\tu{Dom}(z)$; precisely, every component of $\gamma\setminus\{z\}$ approaches $z$ along either an internal ray in $\tu{Acc}(z)$ or a special entrance of the point $z$.
\end{definition}
\begin{figure}[h]
	\centering
	\begin{tikzpicture}
	\node at (0,0){\includegraphics[width=0.83\linewidth]{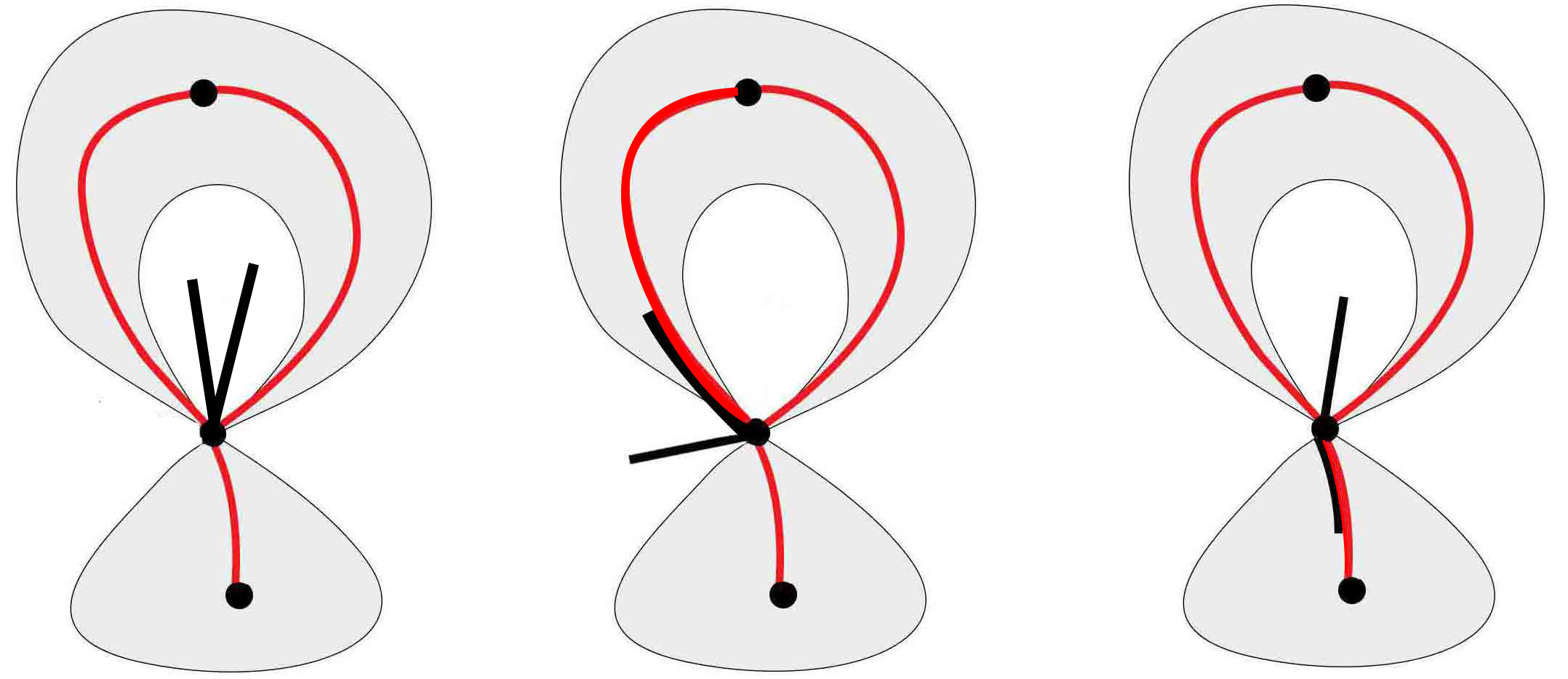}};
	\node at (-4.4,-0.83) {$z_1$};
	\node at (0.25,-0.83) {$z_2$};
	\node at (5.1,-0.83) {$z_3$};
	%\ruler{7}{2};
	\end{tikzpicture}
	\caption{The joint condition at points. The three points $z_i, i\in\{1,2,3\}$ with $\#\tu{Acc}(z_i)=3$ lie in a regulated graph $G$. The red curves are $\tu{Acc}(z_i)$. The three black arcs are neighborhoods of $z_i$ in $G$. The figure shows that $G$ fails the joint condition at $z_1$ and $z_2$.}
	\label{fig:link-condition}
\end{figure}

\begin{remark}
		\tu{(1) The joint condition at an accessible point $z\in G$ is equivalent to saying that each component of $B_z\cap G\setminus\{z\}$ tends to the point $z$ along either an internal ray or a special entrance, and a special entrance receives at most one edge; see Figure \ref{fig:link-condition}. Since there are at most $\#\tu{Acc}(z)$ special entrances at the point $z$, it follows that $\tu{valence}(z, G)\leq 2\,\#\tu{Acc}(z)$.}
	
	\tu{(2) Let $\gamma$ be a Jordan curve containing $\tu{Post}(f)$ such that $\gamma\cap\mc{F}_f$ is a union of internal rays. Then $\gamma=(\tu{Post}(f), \mc{E})$ is a regulated graph. But it fails the joint condition if $\tu{Post}(f)$ contains a buried point in $\mc{J}_f$.}
\end{remark}

\begin{definition}[Admissible graphs]\label{def:admissible}
	\tu{Let $P$ be a finite set in the union of Fatou centers and $\mc{J}_f$ such that $\tu{Post}(f)\subseteq P$. A graph $G=(\mc{V}, \mc{E})$ is called \emph{admissible} with respect to $(f, P)$ if the following conditions are satisfied.
	\begin{itemize}
		\item[(C.1)]  $G$ is regulated.
		 \item[(C.2)] All periodic Fatou centers are non-cut points of $G$.
		 \item[(C.3)] $G$ satisfies the joint condition.
		 \item[(C.4)]  $P\subseteq \mc{V}$ and the \emph{Julia-type} vertices, i.e., vertices in $\mc{J}_f$, of $\mc{V}\setminus P$ are multi-accessible.
		 \item[(C.5)]  All rays in $\tu{Acc}(v)$ with $v\in \mc{V}\cap \mc{J}_f$ periodic are contained in $G$.
	\end{itemize}}

\end{definition}
\begin{remark}
	\tu{Condition (C.2) is used crucially in \S \ref{sec_tiles} to construct tiles of graphs. For the existence of isotopically invariant graphs we need condition (C.3).
	The necessity of conditions (C.4) and (C.5) is due to Theorems \ref{thm:multi} and \ref{thm:invariant}.}%, which say that an admissible graph can be turned into an invariant graph having the same vertex set.}
\end{remark}
%The existence of admissible graphs is stated in the following theorem, which will be proved in \S \ref{sec_existence}.
\begin{theorem}[Existence of admissible graphs]\label{thm:admissible}
	Let $P$ be a finite set in the union of Fatou centers and $\mc{J}_f$ such that $\tu{Post}(f)\subseteq P$. Then there exists a graph $G=(\mc{V},\mc{E})$ that is admissible with respect to $(f, P)$.
\end{theorem}

%All multi-accessible points eventually fall into our invariant graphs obtained in Theorem \ref{thm:invariant}. This is a consequence of the following theorem. We will prove it in \S \ref{sec_tiles}.
\begin{theorem}\label{thm:multi}Every multi-accessible point is eventually periodic and there are at most finitely many cycles of multi-accessible points for every PCF.
\end{theorem}
A subset $S$ of $\olC$ is said to be \emph{invariant} or precisely \emph{$f^n$-invariant} if $f^n(S)\subseteq S$ for an integer $n>0$. Let $P_0$ be the union of
\begin{itemize}
	\item $\tu{Post}(f)$,
	\item the landing points of $R_U(0)$ for all periodic Fatou domains $U$, and
	\item the periodic multi-accessible points, which are finitely many by Theorem \ref{thm:multi}.
	%the periodic points $z$ with $\#\tu{Dom}(z)\geq 3$, which are finite in number by Lemma \ref{lemma:intersection_three_component}.
\end{itemize}
Obvisouly, $P_0$ is a finite and $f$-invariant set.
\begin{theorem}\label{thm:invariant}
	Let $P$ be a finite and $f$-invariant set containing $P_0$. Let $G=(\mc{V}, \mc{E})$ be an admissible graph with respect to $(f, P)$. Then for any $\epsilon>0$ and for each sufficiently large integer $n$, there exists an $f^n$-invariant graph $G_\infty=(\mc{V}_\infty,\mc{E}_\infty)$ in the $\epsilon$-neighborhood of $G$, such that $\mc{V}_\infty=\mc{V}$ and $G_\infty$ is isotopic to $G$ rel. $\mc{V}$.
\end{theorem}
Two graphs $G_i=(\mc{V}_i, \mc{E}_i), i=0,1$ with $\mc{V}_0=\mc{V}_1=\mc{V}$ are said to be \emph{isotopic} rel. $\mc{V}$ if there is an isotopy $H:[0,1]\times \olC\to \olC$ rel. $\mc{V}$ such that $H(0, \cdot)=\tu{id}$ and $H(1, G_0)=G_1$.
\begin{proof}[Proof of Theorem \ref{thm:main}]
	By adding finitely many points into the set $P$ if necessary, we may assume $P_0\subseteq P$. By Theorem \ref{thm:admissible}, there is an admissible graph $G_0=(\mc{V}_0, \mc{E}_0)$ with respect to $(f, P)$. According to Theorem \ref{thm:invariant}, given $\epsilon>0$, one can turn $G_0$ into an $f^n$-invariant graph $G$ in the $\epsilon$-neighborhood of $G_0$ for each sufficiently large integer $n$. The graph $G$ has the same vertex set as $G_0$. The proof of Theorem \ref{thm:main} is complete.
\end{proof}
\begin{remark}\label{rmk:example}
	\tu{
	(1) If the set $P$ contains a buried point $z$ in $\mc{J}_f$, then $z$ is an endpoint of the graph $G$ in Theorem \ref{thm:main} by the joint condition for admissible graphs. In contrast, the invariant graphs obtained in \cite[Theorem 1.2]{GHMZ} are always Jordan curves for carpet rational maps.}
	%, if  cannot be a Jordan curve, comparing the existence of invariant Jordan curves for carpet Julia sets \cite{GHMZ}.} %Because the joint condition fails for any regulated Jordan curve whose vertex set contains a buried point.}
	
	\tu{(2) For some PCFs, we cannot find an $f^n$-invariant Jordan curve containing $\tu{Post}(f)$ for any integer $n>0$. To see this, consider the example: $f(z)=z^2+\bf{i}$. Note that $\tu{Post}(f)=\{\tb{i}, -1+\tb{i}, -\tb{i}, \infty\}$ and $\mc{J}_f$ is a dendrite.
	Suppose $\gamma\supset \tu{Post}(f)$ is an $f^n$-invariant Jordan curve. Let $\alpha$ be the component of $\gamma\cap B$ containing $\infty$, where $B$ is the basin of infinity of $f$. Since $\infty$ is superattracting, we have $f^n(\alpha)\subseteq\alpha$ and $\gamma\cap B=\alpha$. Thus the arc $\wt{\alpha}:=\gamma\setminus\alpha$ is contained in $\mc{J}_f$ and it joins the other three postcritical points. This is impossible, as the three points belong to distinct components of $\mc{J}_f\setminus\{\xi\}$ with $\xi$ an $f$-fixed point.}
\end{remark}

\section{Arcs with the joint condition}\label{sec_arcs}
In the polynomial case,  Douady-Hubbard \cite{DH2} introduced a kind of canonical arcs in the filled Julia set, called \emph{DH-regulated arcs}, with the strong property that two points in the filled Julia set can be joined by a unique DH-regulated arc. % It is easy to check that DH-regulated arcs satisfies the joint condition given in Definition \ref{def:joint}.
 As a generalization of DH-regulated arcs, in this section we shall introduce two kinds of regulated arcs for PCF: clean arcs and $\Omega$-clean arcs. Both of them satisfy the joint condition given in Definition \ref{def:joint}. % Such arcs will play an important role in our construction of invariant graphs.

\subsection{The classification of accessible points}
Let $f$ be a PCF. Let $U$ be a Fatou domain and $z\in\partial U$ be an accessible point.
Recall that $\tu{Dom}(z)$ denotes the family of all Fatou domains whose boundaries contain the point $z$, and $\tu{Acc}(z,U)$ represents the family of internal rays within $U$ terminating at the point $z$.

The point $z$ is called \emph{semi-buried} if it is disjoint from the boundary of any component of $\tu{int}(K_U)$ with $K_U:=\olC\sm U$.
In this case, $\tu{Dom}(z)=\{U\}$. In the following form, we denote by $\Omega$ (resp. $\Omega'$) the component of $\tu{int}(K_U)$ (resp. $\tu{int}(K_{U'})$) containing $U'$ (resp. $U$).
The classification of accessible points is stated as follows.
 %where $U,U'$ are distinct Fatou domains such that $z\in\partial U\cap\partial U'$, and Jordan domains $\Omega$ resp. $\Omega'$ are the components of $\wh{\mb{C}}\sm \ol{U}$ resp. $\wh{\mb{C}}\sm \ol{U'}$, such that $U'\subseteq \Omega$ resp. $U\subseteq \Omega'$. \footnote{perhaps a picture.}
\begin{center}
	\begin{tabular}{|l|l|p{5cm}|}
		\hline
		\multirow{2}{*}{type 1: $\#\tu{Dom}(z)=\#\tu{Acc}(z)=1$.} & {type 1a: $z$ is not semi-buried.}
		\\
		\cline{2-2}  &{type 1b: $z$ is semi-buried.}
		\\
		\hline		
		\multirow{2}{6cm}{type 2: $\#\tu{Dom}(z)=\#\tu{Acc}(z)=2$,
			$\textcolor{white}{aaaaaaaa}\tu{Dom}(z):=\{U,U'\}$.}    & {type 2a: $\{z\}=\bigcap_{n}\,]x_n,y_n[\,_{\partial \Omega}$ with $x_n,y_n\in\partial\Omega\cap\partial\Omega'$.}
		\\
		\cline{2-2}    &   {type 2b: $z$ is not of type 2a.}
		\\
		\hline
		\multicolumn{2}{|l|}  {type 3: $\#\tu{Dom}(z)=\#\tu{Acc}(z)\geq3$.}
		\\
		\hline
		\multirow{2}{6cm}{type 4: $\#\tu{Dom}(z)<\#\tu{Acc}(z)$.} &{type 4a: $z$ is not semi-buried. %in $\partial U$ \color{red}with $\tu{Acc}(z,U)\geq 2$.\color{black}
		}
		\\
		\cline{2-2}   &{type 4b: $z$ is semi-buried, thus $\tu{Dom}(z)=\{U\}$.}
		\\
		\hline
		
	\end{tabular}
\end{center}
\begin{figure}[h]
	\centering
	\begin{tikzpicture}
	\node at (0,0){\includegraphics[width=0.7\linewidth]{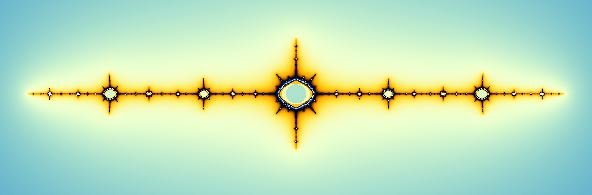}};
	%\ruler{3}{1};
	\end{tikzpicture}
	\caption{Airplane: the Julia set of $z\mapsto z^2-1.75488$.}
	\label{fig:airplane}
\end{figure}
\begin{figure}[h]
	\centering
	\begin{tikzpicture}
	\node at (0,0){\includegraphics[width=0.7\linewidth]{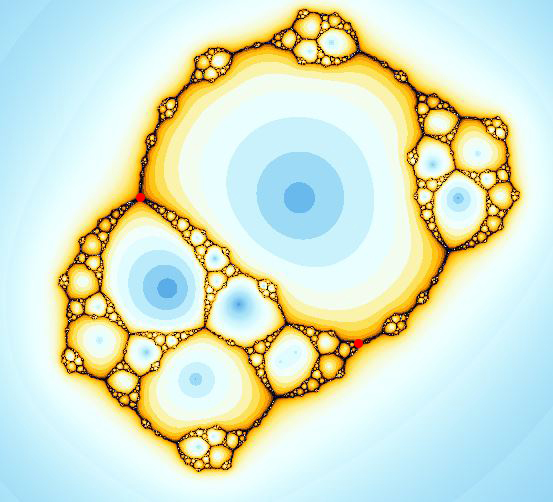}};
	\node at (-1.4,2) {$U_1$};
	\node at (-3.5,2.5) {$U_0$};
	\node at (-2.75,0.25) {$U_2$};
	\node at (-3.1,1.3){$z_0$};
	\node at (2,-2.25) {$w_0$};
	%\ruler{4}{3};
	\end{tikzpicture}
	\caption{The Julia set of a cubic Newton map after conjugacy. The three fixed Fatou domains are $U_0,U_1$ and $U_2$.}
	\label{fig:cubic-Newton}
\end{figure}

\begin{example}
	\tu{(1) If $\mc{J}_f$ is a Sierpi\'{n}ski carpet, then points in $\mc{J}_f$ are either buried or type 1a.}
	
	\tu{(2) In Figure \ref{fig:airplane}, the tips of the ``airplane" Julia set are semi-buried and thus are of type 1b. A point in the boundary of a bounded Fatou domain is of type 4a if it receives at least two external rays; otherwise, it is of type 2a. There are uncountably many type 4b points in the real axis.}
	
	\tu{(3) In Figure \ref{fig:cubic-Newton}, the intersection $\bigcap_{i=0,1,2}\partial U_i=\{z_0\}$ is a singleton. The point $z_0$ is of type 3, as well as its preimages $\bigcup_{i\geq 0}f^{-i}(z_0)$. The intersection $\partial U_0\cap \partial U_1$ is a cantor set, where the points of $\bigcup_{i\geq 0}f^{-i}(w_0)$ (resp. $(\partial U_0\cap\partial U_1)\setminus\bigcup_{i\geq 0}f^{-i}\{z_0, w_0\}$) are of type 2a (resp. type 2b).}
\end{example}

\begin{lemma}\label{lemma:countable}
	The set of types 2b, 3 and 4a points is countable.
\end{lemma}
\begin{proof}
	By definition, an accessible point of type $3$ is contained in the common boundaries of at least three distinct Fatou domains. Then, according to Lemma \ref{lemma:intersection_three_component}, such points are at most countable.
	
	We write
	$S(\Omega,U):=\{z\in\partial\Omega, \#\tu{Acc}(z,U)\geq 2\}$ for a component $\Omega$ of $\olC\setminus \ol{U}$.
	The union of points of type 4a can be represented as
$$\bigcup_{U}\bigcup_{\Omega\in\tu{Comp}(\wh{\mb{C}}\setminus\ol{U})}S(\Omega,U),$$
	where $U$ runs over all Fatou domains. Thus it suffices to show that $S(\Omega,U)$ is at most countable.
	For every $z\in S(\Omega,U)$, there exists an open arc $I_z=(\theta_z,\theta_z')\subseteq \mb{R}/\mb{Z}$ with $R_U(\theta_z), R_U(\theta_z')\in\tu{Acc}(z,U)$ satisfying that the landing point of every $R_U(\theta), \theta\in I_z$ avoids $\partial \Omega$. By the fact that $I_z\cap I_w=\emptyset$ whenever $z\neq w\in S(\Omega,U)$, we see that the family $\{I_z\}_{z\in S(\Omega, U)}$ is at most countable, and so is  $S(\Omega,U)$.

Finally, we prove that the points of type 2b are at most countable. For two distinct Fatou domains $U_1$ and $U_2$, let $\Omega_i$ be the unique component of $\wh{\mb{C}}\setminus\ol{U_i}, i\in\{1,2\}$, such that
$$U_2\subseteq \Omega_1\tu{ and }U_1\subseteq\Omega_2.$$ Then $\Omega_1$ and $\Omega_2$ are Jordan domains and $\partial U_1\cap\partial U_2=\partial \Omega_1\cap\partial \Omega_2$.
Thus the components of $\partial \Omega_i\setminus(\partial\Omega_1\cap\partial\Omega_2)$ are open intervals. The collection of these intervals is denoted by $\mc{I}(U_1,U_2)$. Then the set
$$\bigcup_{U_1\neq U_2\in\tu{Comp}(\mc{F}_f)}\bigcup_{I\in\mc{I}(U_1,U_2)}\tu{end}\,(I),$$
which contains all type 2b points, is countable.
The proof is complete.
\end{proof}

\subsection{First-in and Last-out rule}
In this subsection, we show that any two points in $\mc{J}_f$ can be joined by a regulated arc. An arc $\gamma$ is called \emph{regulated} if it is a regulated graph, that is, if the intersection of $\gamma$ with $\mc{F}_f$ is a union of internal rays. Obviously, endpoints of regulated arcs are either Fatou centers or points in $\mc{J}_f$.% Precisely, if the intersection of $\gamma$ with $\mc{F}_f$ is a union of internal rays.% This point is realized by applying the so-called \emph{First-in and Last-out} rule.

 Let $\gamma:[0,1]\to\wh{\mb{C}}$ be an arc and $S$ be a closed set in $\wh{\mb{C}}$ such that $\gamma\cap S\neq\es$. Let
$$t_1:=\tu{inf}\,\{t\in[0,1]:\gamma(t)\in S\}\tu{ and }t_2:=\tu{sup}\,\{t\in[0,1]:\gamma(t)\in S\}.$$
Then the \emph{first-in time} that $\gamma$ meets $S$ is $t_1$ and the \emph{first-in place} is $\gamma(t_1)$; similarly, the \emph{last-out time} is $t_2$, and the \emph{last-out place} is $\gamma(t_2)$.

Recall that the open, closed and semi-open segments in $\gamma$ between $z$ and $w$ for $z\neq w\in \gamma$ are denoted by $]z,w[_\gamma$, $[z,w]_\gamma$, and $[z,w[_\gamma$ (or $]z,w]_\g$) accordingly.
\begin{lemma}[First-in and Last-out rule]\label{first_in_last_out}
	Let $\gamma:[0,1]\to\wh{\mb{C}}$ be an arc with two Julia-type endpoints $z_1$ and $z_2$. Then there exists a regulated arc ${\gamma_\infty}$ connecting $z_1$ and $z_2$ such that  $${\gamma_\infty}\cap\mc{J}_f\subseteq\gamma\cap\mc{J}_f.$$ Moreover, the intersection of $\g_\infty$ with the closure of a Fatou domain $U$, if non-empty, is either
	\begin{itemize}
		\item  one point in $\partial U$, or
		\item  the union of two closed internal rays, or
		\item  two points in $\partial U$.
	\end{itemize}
The last case happens only if there exists another Fatou domain $U'$ with $\tu{diam}\,U'\leq\tu{diam}\,U$ such that $\g_\infty\cap\partial U=:\{x,y\}\subseteq \partial U\cap\partial U'$ and $\g_\infty\cap\ol{U'}$ consists of two closed internal rays landing at $x$ and $y$, respectively.
\end{lemma}

\begin{proof}
	Enumerate all Fatou domains of $f$ as $U_1, U_2,\dotsc$ such that
	$\tu{diam}\,U_n\geq \tu{diam}\,U_{n+1}$. Let $\gamma_0:=\gamma$. We will inductively construct a sequence of arcs $\gamma_n$ such that the limit $\gamma_\infty$ is as required.
	
	Let $n$ be a positive integer. If $\#\gamma_{n-1}\cap\ol{U}_n\leq 1$, we set $\gamma_n:=\gamma_{n-1}$ and $I_n:=\emptyset$. Otherwise, we obtain the distinct first-in time $a_n$ and last-out time $b_n$ of $\gamma_{n-1}$ meeting $\ol{U}_n$ at the corresponding places $x_n:=\gamma_{n-1}(a_n)$ and $y_n:=\gamma_{n-1}(b_n)$. Then we remove the open segment $]x_n,y_n[_{\gamma_{n-1}}$ and replace it by the union of the two internal rays of $U_n$ that land at $x_n$ and $y_n$. The new arc $\gamma_n$ obtained in this way differs from $\gamma_{n-1}$ only on the interval $I_n:=[a_n,b_n]$.

	Let $\mc{I}$ denote the collection of all non-empty intervals  $I_n$. The construction above implies that either $I_m\subseteq I_n$ or $ \tu{int}\,(I_n)\cap\tu{int}\,(I_m)=\es$ for $1\leq m<n$. % Moreover, $\g_0(t)=\g_n(t)\,\,\forall\, t\in[0,1]\setminus\bigcup_{i\geq1}I_i$.
	We define a partial relation ``$\prec$'' on $\mc{I}$ in the sense that $$I_m\prec I_n\Leftrightarrow I_m\subseteq I_n.$$ Observe that $\g(a_n),\g(b_n)\in \ol{U}_n$ and $\lim\limits_{n\to\infty}\tu{diam}\,U_n=0$.
Hence, any totally ordered subsets of $\mc{I}$ under this partial relation are finitely many, and thus have a maximal element. Let $\{I_{n_k}\}$ be the collection of maximal elements in $\mc{I}$ with $n_1<n_2<\cdots$. It follows that these $I_{n_k}$'s have mutually disjoint interiors, and their union equals $\bigcup_{n\geq1}I_n$.

	We now check the uniform convergence of $\gamma_n$. Since the diameter of $I_{n_k}$ tends to zero as $k$ goes to $\infty$,
	 the uniform continuity of $\gamma_0$ on $[0,1]$ implies that there exists an integer $k_0>0$ such that $\tu{diam}\,\gamma_0(I_{n_k})<\epsilon$ when $k\geq  k_0$. We assume that $k_0$ is so large that $\tu{diam}\,U_n\leq \epsilon$ when $n\geq n_{k_0}$. Note that $\g_n(t)=\g_m(t)$ if $t\in I_{n_1}\cup\cdots\cup I_{n_{k_0}}$ and $m,n\geq n_{k_0}$, and $\g_0(t)=\g_n(t)$ if $t\in[0,1]\setminus\cup_{k\geq1}I_{n_k}$, $n\geq1$.
Hence,	for $m,n\geq n_{k_0}$ and $t\in[0,1]$, we have the estimate:	
	\begin{equation}\notag
	\begin{split}
	\tu{dist}\,(\gamma_m(t),\gamma_n(t))
	\leq&~\mathop{\tu{sup}}^{}_{t\in I_{n_1}\cup\cdots\cup I_{n_{k_0}}}\tu{dist}\,(\gamma_m(t),\gamma_n(t))+ \mathop{\tu{sup}}^{}_{t\in [0,1]\setminus I_{n_1}\cup\cdots\cup I_{n_{k_0}}}\tu{dist}\,(\gamma_m(t),\gamma_n(t))\\
	\leq &~ 0+\mathop{\tu{sup}}^{}_{t\in [0,1]\setminus \cup_{k\geq 1}I_{n_k}}\tu{dist}\,(\gamma_m(t),\gamma_n(t))+\mathop{\tu{sup}}^{}_{t\in \cup_{k>{k_0}} I_{n_k}}\tu{dist}\,(\gamma_m(t),\gamma_n(t))\\
	\leq &~0+0+\mathop{\tu{sup}}^{}_{t\in \cup_{k> {k_0}}I_{n_k}}\tu{dist}\,(\gamma_m(t),\gamma_0(t))+\mathop{\tu{sup}}^{}_{t\in \cup_{k> {k_0}}I_{n_k}}\tu{dist}\,(\gamma_0(t),\gamma_n(t)) \\
	\leq &~2\mathop{\tu{sup}}_{k>{k_0}}\tu{diam}\,\gamma_0|_{I_{n_k}}+2\mathop{\tu{sup}}_{n> n_{k_0}}\tu{diam}\,U_n\\
	\leq &~4\epsilon.
	\end{split}
	\end{equation}	
 It follows that $\gamma_n$ uniformly converges to a limit, say $\gamma_{\infty}$. From the inductive construction of $\g_n$, we see that $\gamma_\infty$ is regulated.
\end{proof}

\begin{remark}\label{rem:first-last}\tu{
Lemma \ref{first_in_last_out} will be  frequently used in the sequel. For simplicity, we denote the resulting arc after the first-in and last-out processes upon an arc $\gamma$ by $\tu{First-Last}\,(\gamma)$.}
\end{remark}

\begin{lemma}\label{lem:first-last}
	The following statements hold for the arc $\gamma_\infty$ obtained in Lemma \ref{first_in_last_out}.
	\begin{itemize}
		\item[(1)] $\tu{int}(\gamma_\infty)$ is disjoint from the type 1b and 4b points;
		\item[(2)] Let $z$ be a type 2a point in $\tu{int}(\gamma_\infty)$. Then the rays in $\tu{Acc}(z)$ are contained in $\gamma_\infty$.
	\end{itemize}
\end{lemma}
\begin{proof}
	(1) If it is not true, let $z\in\partial U$ be a type 1b or 4b point in $\tu{int}(\gamma_\infty)$. Then at least one component of $\gamma_\infty\setminus\{z\}$ is contained in $K_U$ and converging to the point $z$. Since $\#\gamma_\infty\cap\partial U\leq 2$, there is an open subarc of $\gamma_\infty$ approaches $z$ within $\tu{int}(K_U)$. This is impossible as the point $z$ is assumed to be semi-buried in $\partial U$.
	
	(2) The argument is very similar to the proof of statement (1). If it is not true, a component of $\gamma_\infty\setminus\{z\}$ tends to $z$ in $K_{U_1}\cap K_{U_2}$, where $\tu{Dom}(z)=\{U_1,U_2\}$. Since $\gamma_\infty$ meets $\partial U_i, i=1,2,$ in at most two points, an open subarc of $\gamma_\infty$ approaches the point $z$ in $\tu{int}(K_{U_1})\cap\tu{int}(K_{U_2})$. This is a contradiction, since the point $z$ is accumulated by points from $\partial U_1\cap\partial U_2$ in both directions.	
\end{proof}

\subsection{Clean arcs}
We introduce here the first kind of regulated arcs, namely clean arcs. They are the ``units'' to build up the other specific regulated arcs and graphs.
\begin{definition}[Clean arcs]\label{def:clean}
\tu{A regulated arc $\g$ in $\wh{\mathbb{C}}$ is called \emph{clean}, provided that
\begin{itemize}
	\item every accessible point $z\in\tu{int}(\g)$ is of either type 1a or 2a, and all rays in $\tu{Acc}(z)$ belong to $\g$;
	\item $\gamma$ satisfies the joint condition at the accessible Julia-type endpoints if any, i.e., $\gamma$ tends to each accessible Julia-type endpoint along either an internal ray or a special entrance.
\end{itemize}}

%\begin{itemize}
%\item the arc $\gamma$ meets the Fatou set in a union of internal rays;
%\item every accessible point $z\in\tu{int}(\g)$ is of either type 1a or 2a, and the rays in $\tu{Acc}(z)$ belong to $\g$;
%\item the arc $\g$ has no obstructions at its Julia-type endpoints.
%\end{itemize}
\end{definition}
Obviously, clean arcs satisfy the joint condition. Moreover, subarcs with endpoints in the union of $\mc{J}_f$ and Fatou centers of clean arcs are clean as well.
\begin{remark}
	\tu{A clean arc $\gamma$ disjoint from a Fatou domain $U$ implies that $\gamma$ is contained in $\ol{\Omega}$ for some $\Omega\in\tu{Comp}(\olC\setminus\ol{U})$.}
	%(1) The non-empty intersection of a clean arc with the closure of a Fatou domain $U$ is either a singleton, or two points, or one closed internal ray, or the union of two closed internal ray, or the union of two closed internal rays and one point in $\gamma$.
\end{remark}

\begin{proposition}[Existence of clean arcs I]\label{prop:exist_well_regulated_arc}
	Let $\gamma$ be an arc with two Julia-type endpoints $z_1$ and $z_2$. Then for any $\epsilon>0$ there exists a clean arc $\gamma_\infty$ joining $z_1$ and $z_2$ such that
	\begin{equation}\label{eq:neighh}
	\gamma_\infty\cap\mc{J}_f\subseteq\mc{N}(\gamma\cap \mc{J}_f,\epsilon).
	\end{equation}
	Moreover, $\gamma_\infty$ approaches $z_i$ along an internal ray whenever $\tu{Dom}(z_i)\neq \es$ for $i=1,2$.
\end{proposition}

We need the following lemma to prove Proposition \ref{prop:exist_well_regulated_arc}.%, the notion of pre-clean arcs and Lemma \ref{lem:pre-clean-property}.

\begin{lemma}[Avoiding countably many points]\label{lem:avoiding}
	Let $D\subseteq\wh{\mb{C}}$ be a Jordan domain and $S$ be a countable set in $\wh{\mb{C}}$. Let $\gamma$ be an arc with
	\begin{equation}\label{eq:avoiding}
		\tu{int}(\gamma)\subseteq D.
	\end{equation} Then for any $\epsilon>0$ there exists an arc $\wt{\gamma}$ with $\tu{end}(\wt{\gamma})=\tu{end}(\gamma)=\wt{\gamma}\cap\gamma$ satisfying \eqref{eq:avoiding} such that
	$$\tu{int}(\wt{\gamma})\cap S=\es\tu{ and }\wt{\gamma}\subseteq\mc{N}(\gamma,\epsilon).$$
\end{lemma}
\begin{proof}
Choose an isotopy $\Psi:\g\times[0,1]\to \ol{D}$ rel. $\tu{end}(\gamma)$ such that
 $\Psi(\g\times [0,1])\subseteq \mc{N}(\gamma,\epsilon)$, $\Psi(\tu{int}(\gamma)\times[0,1])\subseteq D$ and the intersection of the two arcs $\Psi(\g\times\{t_1\})$ and $\Psi(\g\times\{t_2\})$ is $\tu{end}(\gamma)$  for $t_1\not=t_2\in[0,1]$. Since $S$ is countable, there is a $t_0\in[0,1]$ such that the interior of $\wt{\g}:=\Psi(\gamma\times\{t_0\})$ avoids $S$. Then $\wt{\g}$ is as required.
\end{proof}

\begin{proof}[Proof of Proposition \ref{prop:exist_well_regulated_arc}]
In the trivial case that $\tu{Dom}(z_1)\cap \tu{Dom}(z_2)$ contains a Fatou domain $U$, the union of the two closed internal rays of $U$ landing at $z_1$ and $z_2$ is as required.

Otherwise, by Lemmas \ref{lemma:countable} and \ref{lem:avoiding} one may assume that $\gamma$ is disjoint from the types 2b, 3 and 4a points. One may assume further that $\gamma$ tends to each $z_i$ along a Fatou domain whenever $\tu{Dom}(z_i)\neq\es$. Consider the new arc $\wt{\gamma}:=\tu{First-Last}(\gamma)$. By Lemmas \ref{first_in_last_out} and \ref{lem:first-last}, it satisfies that
\begin{itemize}
	\item[(i)] $\tu{Dom}(z_1)\cap\tu{Dom}(z_2)=\es$;
	% and $\wt{\gamma}$ tends to $z_i$ along an internal ray whenever $\tu{Dom}(z_i)\neq\es$;
	\item[(ii)] each accessible point $z\in\tu{int}(\wt{\gamma})$ is of either type 1a or 2a; the latter happens only if all rays in $\tu{Acc}(z)$ are contained in $\wt{\gamma}$; and
	\item[(iii)] the non-empty intersection of $\wt{\gamma}$ with the closure of a Fatou domain is either one point or the union of two closed internal rays.
\end{itemize}
To see (iii), if it is not true, there is a Fatou domain $U$ such that ${\ol{U}}\cap \wt{\gamma}=\{\xi_1,\xi_2\}$ with $\#\tu{Dom}(\xi_i)\geq 2$ by Lemma \ref{first_in_last_out}. By (ii), we have $\xi_1, \xi_2 \in \tu{end}(\wt{\g})=\{z_1, z_2\}$. This contradicts (i).

For simplicity, an arc with two Julia-type endpoints satisfying (i)(ii) and (iii) is called \emph{pre-clean}. Let $\beta$ be a pre-clean arc. A point $z\in\tu{int}(\beta)$ is called \emph{dusty} if it is of type 1a and the unique ray in $\tu{Acc}(z)$ does not belong to $\beta$.

It is useful to see that a point $z\in \partial U\cap \tu{int}(\beta)$ is dusty for $
\beta$ $\Leftrightarrow$ $\ol{U}\cap \beta=\{z\}$. By definition, the points that prevent $\beta$ from being clean are just the dusty points if $\beta$ satisfies the joint condition at endpoints.

The strategy for constructing $\gamma_\infty$ is as follows. First enumerate all Fatou domains such that
$$\tu{diam}\,U_1\geq \cdots\geq \tu{diam}\,U_n\geq \tu{diam}\,U_{n+1}\geq \cdots.$$
Then beginning with $\gamma_1:=\wt{\gamma}$, we construct a sequence of pre-clean arcs $\gamma_{n+1}$ by induction such that $\gamma_{n+1}$ contains no dusty point from the boundaries $\partial U_1,\ldots,\partial U_n$. Finally, we prove that the sequence $\gamma_{n+1}$ uniformly converges to a required arc $\gamma_\infty$. To do this, we need the lemma.

\begin{lemma}\label{lem:pre-clean-property}
	Let $\beta: [0, 1]\to \olC$ be a pre-clean arc joining $z_1$ and $z_2$. Let $z\in \tu{int}(\beta)$ be a dusty point. Let $Dom(z)=\{U_0\}$ and $\Omega$ be the component of $\olC\setminus\ol{U}_0$ containing $\beta\setminus\{z\}$. Then the following statements hold.
	\begin{itemize}
		\item[(1)] There exist two sequences $\{x_n\}$ and $\{y_n\}$ in the union of type 2a and buried points within $\beta$ such that $z\in\,]x_n,y_n[_{\beta}$, $x_n\to z$ and $y_n\to z$ as $n\to\infty$.
		\item[(2)] Let $\alpha_n\subseteq \Omega$ be an arc joining $x_n$ and $y_n$  with its interior disjoint from both $\beta$ and the types 2b, 3 and 4a points (see Lemmas \ref{lemma:countable} and \ref{lem:avoiding} for the existence of such $\alpha_n$). Then $${\beta}_n:=\tu{First-Last}\left((\beta\setminus[x_n,y_n]_{\beta})\cup\alpha_n\right)$$ is pre-clean.
		\item[(3)] The arcs $\alpha_n$ can be chosen so that $\tu{diam}\,\alpha_n\to 0$ as $n\to \infty$. In this case, let 
		$$t_n^-:=\tu{sup}\,\{t: \beta([0, t])\subset \beta_n\}\tu{ and } t_n^+:=\tu{inf}\,\{t:\beta([t, 1])\subset \beta_n\}.$$
		Then we have $$\tu{diam}\,[a^-_n,a^+_n]_\beta\to 0\tu{ and }\tu{diam}\,[a^-_n, a^+_n]_{\beta_n}\to 0\tu{ as }n\to\infty,$$
		where $a_n^-:=\beta(t_n^-)$ and $a_n^+:=\beta(t_n^+)$.
		%with $a_n:=\tu{sup}\{t: \beta(s)=\beta_n(s)~\forall s\in[0, t]\}$ and $b_n:=\tu{inf}\{t: \beta(s)=\beta_n(s)~\forall s\in[t, 1]\}$.
	\end{itemize}
	
\end{lemma}
\begin{proof}
	(1) If the statement is not true, there is an arc $[z,w]_{\beta}$ such that the closed subset $S:=[z,w]_{\beta}\cap \mc{J}_f$ of $\beta$ consists of only type 1a points.
	 Note that every type 1a point in $\beta$ is accumulated by the points in $\beta\cap\mc{J}_f$. Thus $S$ has no isolated point. Moreover, $S$ is totally disconnected, as $S$ contains only type 1a points. Therefore, $S$ is a cantor set. This contradicts the fact that the set of type 1a points in $\beta$ is countable.
	
	%The statement of (2) follows directly by the discussions above.
	Statement (2) follows directly from the conditions of pre-clean
	arcs for $\beta$ and the properties of resulting arcs after the first-in
	 and last-out rule given in Lemma \ref{lem:first-last}.
	
	(3) For any $\epsilon>0$, choose an open disk $D$ around the dusty point $z$ with $\tu{diam}\,D<\epsilon$ such that $\ol{D}\cap \beta$ is an arc whose interior is contained in $D$.
%	\begin{itemize}
%		\item $\ol{D}\cap \gamma$ is an arc with its interior in $D$;
	%	\item any Fatou domain meeting $\ol{D}$ is less than $\epsilon$ except $U_0$; the collection of these Fatou domains are denoted by $\mc{U}$.
%	\end{itemize}
	For large integer $n$, we may choose $\alpha_n\subseteq D$. Let $\mc{U}_n$ be the collection of Fatou domains, except $U_0$, whose closures meet $\alpha_n$. Since the point $z$ is of type 1a, we may assume $n$ is so large that the domains in $\mc{U}_n$ are contained in $D$. Let $\wt{\beta}_n:=(\beta\setminus[x_n,y_n]_{\beta})\cup\alpha_n$. By condition, we have $\beta_n=\tu{First-Last}(\wt{\beta}_n)$.
	
	We claim that every component $\eta$ of $\beta_n\setminus\wt{\beta}_n$ is covered by $\mc{U}_n$.
	%We claim that every component $\beta$ of $\gamma_n\setminus\wt{\gamma}_n$ is covered by  $S:=\ol{D}\cup\{\ol{U}, U\in\mc{U}\}$.
	Indeed, since $\beta_n\cap\mc{J}_f\subseteq \wt{\beta}_n$,	
	the component $\eta$ is contained in a Fatou domain $U$. If $\tu{end}(\eta)\cap \alpha_n\neq \es$, then $U\in \mc{U}_n$  by the choice of $\mc{U}_n$. Otherwise, we have $\tu{end}(\eta)\subseteq \beta\setminus[x_n,y_n]_\beta$ and $\eta$ is formed by two internal rays of $U$, say $R_1$ and $R_2$. Note that $\#(\ol{U}\cap \beta)\geq \#(\tu{end}(\eta)\cap\beta)=2$. By properties (ii) and (iii) for the pre-clean arc $\beta$, we have $\ol{U}\cap \beta=R_1\cup R_2$ . Thus $R_1\cup R_2\subseteq \beta\setminus [x_n, y_n]_\beta\subseteq \wt{\beta}_n$, contradicting the choice of $\eta$. This implies the case $\tu{end}(\eta)\cap \alpha_n=\es$ cannot happen. The claim follows.
	
	% Since both $\beta$ and $\beta_n$ are pre-clean,
%	$\eta$ is formed by two internal rays $R_1$ and $R_2$ of $U$; moreover, the landing points of $R_1$ and $R_2$ cannot be dusty for $\beta$ as $\#\ol{U}\cap\beta\geq 2$. By definition, $R_1\cup R_2\subseteq \beta\setminus [x_n,y_n]_\beta\subseteq\wt{\beta}_n$. It contradicts the choice of $\eta$. The claim follows.
	
	By the claim and the choice of $n$, we have
	\begin{equation}\label{eq:set-minus}
	\beta_n\setminus\beta\subseteq \beta_n\setminus\left(\beta\setminus [x_n,y_n]_\beta\right)\subseteq (\beta_n\setminus\wt{\beta}_n)\cup\alpha_n\subseteq D.
	\end{equation}
	Then $a_n^-$ and $a_n^+$ are contained in $D$. Thus $[a_n^-, a_n^+]_\beta\subset D$ and $\tu{diam}\,[a^-_n,a^+_n]_\beta<\epsilon$ by the choice of $D$. To show $[a_n^-,a_n^+]_{\beta_n}\subset D$, by \eqref{eq:set-minus} it suffices to prove that the set $[a_n^-,a_n^+]_{\beta_n}\cap \beta$ is contained in $D$. Indeed, the definitions of $a_n^-$ and $a_n^+$ imply that $]a_n^-, a_n^+[_{\,\beta_n}$ is disjoint from $[z_1, a_n^-]_\beta\cup [a_n^+, z_2]_\beta$. So $[a_n^-, a_n^+]_{\beta_n}\subset D$. The proof is complete.	
\end{proof}

We now show the inductive construction of $\{\gamma_{n+1}\}_{n\geq 1}$. % One may assume that $\g_n$ is a pre-clean arc joining $z_1$ and $z_2$ such that $\tu{int}(\g_n)\cap(\partial U_1\cup\cdots\cup \partial U_{n-1})$ contains no dusty point.
If $\tu{int}(\gamma_n)$ contains no dusty points from $\partial U_n$, let $\gamma_{n+1}:=\gamma_n$. Otherwise, we have $\ol{U_n}\cap \gamma_n=\{z\}$ with $z$ a dusty point of $\gamma_n$. In this case we choose an open disk $D_n$ around $z$ such that
 \begin{itemize}
 \item $D_n$ avoids $U_1,\ldots,U_{n-1}$ and  $\tu{diam}\,D_n<\epsilon/2^n$;
 \item $\g_n\cap D_n=]w_-,w_+[_{\gamma_n}=:\gamma_n(I_n)$ with $w_-$ and $w_+$ in the union of type 2a and buried points; this can be done by Lemma \ref{lem:pre-clean-property}\,(1);
 \item whenever $\ol{D_n}\cap \ol{D_m}\neq\es$ for an integer $m<n$, it holds that $D_n\Subset D_m$;
 \end{itemize}
By Lemma \ref{lem:pre-clean-property}\,(2)(3), we may modify ${\gamma_n}$ within $D_n$ such that the resulting arc $\gamma_{n+1}$ has the properties:
\begin{enumerate}
	\item $\gamma_{n+1}$ is pre-clean and $\tu{int}(\gamma_{n+1})\cap \ol{U}_n=\es$;
	\item $\gamma_{n+1}(t)=\gamma_n(t)$ for $t\in [0,1]\setminus I_n$;
	\item $\sup\{\tu{dist}(\gamma_n(t),\gamma_{n+1}(t)): t\in I_n\}<\epsilon/2^n$.
\end{enumerate}

The uniform convergence of $\g_n$ follows from the estimate:
$$
\sup\limits_{t\in [0,1]}\tu{dist}(\g_m(t), \g_n(t))\leq\sum_{k=m}^{n-1}\sup\limits_{t\in [0,1]}\tu{dist}(\g_k(t),\g_{k+1}(t))\leq\sum_{k=m}^{n-1}\epsilon/2^k< \epsilon/2^{m-1}
$$
for every $1\leq m<n$, due to properties (2)(3). Let $\g_\infty:[0,1]\to \wh{\mb{C}}$ be the limit. The choices of $D_n$ and $I_n$ yield that
\begin{itemize}
\item the intervals $I_n$ are either nested or disjoint with $\tu{diam}\,I_n\to 0$ as $n\to\infty$;
\item if $\g_m(t)\in D_m$ for an integer $m\geq 1$ and $t\in I_m$, then $\g_n(t)\in D_m$ for all $n\geq m$.
\end{itemize}
These two properties imply that $\g_\infty$ is injective. From the construction, $\g_\infty$ is pre-clean and contains no dusty point.
Note that $\g_1$ tends to each endpoint $z_i$ along an internal ray whenever $\tu{Dom}(z_i)\neq \es$. By definition $\gamma_1$ satisfies the joint condition at $z_i$, so does $\g_\infty$. Therefore, $\g_\infty$ is clean. Clearly $\gamma_\infty$ satisfies \eqref{eq:neighh}. The proof of Proposition \ref{prop:exist_well_regulated_arc} is complete.
\end{proof}

\begin{corollary}[Existence of clean arcs II]\label{cor:two_centers}
	Let $\gamma$ be an arc with its two endpoints $z_1$ and $z_2$ in the union of Fatou centers and buried points.
	Then for any $\epsilon>0$ there exists a  clean arc $\wt{\gamma}$ joining $z_1$ and $z_2$ such that
	$$\wt{\gamma}\cap\mc{J}_f\subseteq \mc{N}(\gamma\cap \mc{J}_f,\epsilon).$$
\end{corollary}
\begin{proof}
	Without loss of generality, we may assume that $\tu{int}(\gamma)$ avoids the union of types 2b, 3 and 4a points by Lemmas \ref{lemma:countable} and \ref{lem:avoiding}.
	
	We first deal with the case that both $z_i$ are the centers of Fatou domains $U_i$ for $i=1,2$. Let $w_1$ (resp. $w_2$) be the last-out (resp. first-in) place of $\gamma$ meeting $\ol{U}_1$ (resp. $\ol{U}_2$).  Clearly $w_i\in\partial U_i$.
	%By the choice of $\gamma$ and Lemma \ref{lem:first-last}, the points $w_i$ are of either type 1a or type 2a.
	
	 We claim that each $w_i$ is of either type 1a or 2a. To see this, by the property of $w_i$, there is a component of $\gamma\setminus w_i$ disjoint from $\ol{U_i}$ approaching $w_i$. Then the point $w_i$ cannot be semi-buried. It follows that $w_i$ is not of type 1b and 4b. Combining the assumption on the points of $\gamma$ at the beginning of the proof, the claim follows.
	  
	% To see this, for otherwise, the assumption on $\gamma$ implies that the point $w_i$ is of either type 1b or 4b. There is a component of $\gamma\setminus w_i$ disjoint from $\ol{U_i}$ approaching $w_i$. Then this component is contained in $\tu{int}(K_{U_i})$. It is impossible, since $w_i$ is assumed semi-buried in $\partial U_i$.
%	By the choices of $\gamma$ and $w_i$, we know that $w_i$ can only possibly be of types 1a and 2a; see the proof of Lemma \ref{lem:first-last}.
	
	Now we apply Proposition \ref{prop:exist_well_regulated_arc} to the arc $[w_1, w_2]_\gamma$, and obtain a clean arc $\alpha$ joining $w_1$ to $w_2$, which satisfies $\alpha\cap\mc{J}_f\subseteq \mc{N}(\gamma\cap\mc{J}_f,\epsilon)$. Let $\wt{z}_1$ (resp. $\wt{z}_2$) be the last-out (resp. first-in) place of $\alpha$ meeting $\ol{U}_1$ (resp. $\ol{U}_2$). It is possible that $\{\wt{z}_1,\wt{z}_2\}\neq \{w_1,w_2\}$. Since $\wt{z}_i$ are of type 1a or 2a by the claim, $[\wt{z}_1, \wt{z}_2]_\alpha$ tends to $\wt{z}_i$ along Fatou domains if and only if $\wt{z}_i$ is of type 2a; see the argument in the proof of Lemma \ref{lem:first-last}\,(2).
	Let $R_i$ be the internal rays of $U_i$ landing at $\wt{z}_i$. Then the new arc $$\wt{\gamma}:=R_1\cup[\wt{z}_1,\wt{z}_2]_{\alpha}\cup R_2$$
	joining $z_1$ and $z_2$ is clean and is as required.
	
	In the case that $z_1=c(U_1)$ and $z_2$ is a buried point, the proof is quite similar. Let $w_1$ be the last-out place of $\gamma$ meeting $\ol{U_1}$. By applying Proposition \ref{prop:exist_well_regulated_arc} to $[w_1, z_2]_{\gamma}$, we obtain a clean arc $\alpha$. Let $\wt{z}_1$ be last-out place of $\alpha$ meeting $\ol{U_1}$ and $R_1$ be the internal ray of $U_1$ landing at $\wt{z}_1$. 	Then $\wt{\gamma}:=R_1\cup[\wt{z}_1, z_2]_{\alpha}$ is as required.
	
	If both $z_i$ are buried points, it follows directly by Proposition \ref{prop:exist_well_regulated_arc}. The proof is complete.
\end{proof}
%The clean arc $\gamma$ obtained in Proposition \ref{prop:exist_well_regulated_arc} always goes to an accessible endpoint $z_i$ through a Fatou domain $U$ of $\tu{Dom}(z_i)$.
To get a better control of the behavior of a clean arc $\gamma$ near its accessible Julia-type endpoint $z\in\partial U$, we usually require that $\gamma$ tends to $z$ through a Fatou domain within a specific component of $\tu{int}(K_U)$ with $K_U:=\olC\setminus U$. This motivates the following definition for clean arcs.

 \begin{definition}\label{def:prop}
 	\tu{Let $\Omega$ be a component of $\tu{int}(K_U)$ for a Fatou domain $U$. An arc $\gamma$ is called \emph{properly clean} for $\Omega$ if $\g$ is a clean arc joining a point $z\in\partial\Omega$ and a Fatou center in $\Omega$, such that $\gamma\setminus\{z\}\subset \Omega$ and it approaches the endpoint $z$ along an internal ray whenever $\Omega$ contains a Fatou domain of $\tu{Dom}(z)$.}
 \end{definition}

The requirement for properly clean arcs tending Julia-type endpoints will be used to establish Lemma \ref{lem:strong-regulated}\,(3); see Figure \ref{fig:proper} for illustration.
\begin{proposition}[Existence of clean arcs III]\label{prop:prop}
	Let $\Omega$ be a component of $\tu{int}(K_{U_1})$ for a Fatou domain $U_1$. Let $\gamma$ be an arc joining a point $z_1\in \partial \Omega$ to a Fatou center $z_2=c(U_2)$ such that $\gamma\setminus\{z_1\}\subset \Omega$. Then for any $\epsilon>0$, there is a properly clean arc $\wt{\gamma}$ for $\Omega$ joining $z_1$ and $z_2$ such that $$\wt{\gamma}\cap\mc{J}_f\subseteq \mc{N}({\gamma}\cap\mc{J}_f,\epsilon).$$
\end{proposition}
\begin{proof}
	Without loss of generality, we may assume $\gamma$ is disjoint from the types 2b, 3 and 4a points by Lemmas \ref{lemma:countable} and \ref{lem:avoiding}. Moreover, by necessary modification we also assume $\gamma$ tends to $z_1$ along a Fatou domain in $\Omega$ whenever such a domain exists.
	
	Let $\wt{z}_2$ be the first-in place of $\gamma$ meeting $\ol{U_2}$. If $z_1=\wt{z}_2$, the desired arc $\wt{\gamma}$ is the closed internal ray of $U_2$ landing at $z_1$, we are done. Otherwise, let $\alpha_1:=[{z}_1, \wt{z}_2]_\g$. Then $\tu{int}(\alpha_1)$ is contained in a component $W$ of $\Omega\setminus\ol{U}_2$. Moreover, by the reason explained in the third paragraph of the proof of Corollary \ref{cor:two_centers}, the point $\wt{z}_2\in\tu{int}(\gamma)$ is of type 1a or 2a. In the latter case, assume $\tu{Dom}(\wt{z}_2)=\{U_2, U_2'\}$. Clearly $U_2'\neq U_1$ as $\wt{z}_2\notin \partial U_1$. Since $\#\tu{Acc}(\wt{z}_2, U_2)=1$, we have $U_2'\subset \Omega$ and so $U_2'\subset W$. By taking some modification on $\alpha_1$, which is similar as the operation on $\gamma$ in the first paragraph of the proof, we assume $\alpha_1$ tends to $\wt{z}_2$ along the domain $U_2'$ in this case.
	
	%From the definition of type 2a points, the point $\wt{z}_2$ is a limit of points in $\partial \Omega\cap \partial\Omega'$ in both directions. Thus $\alpha_1$ tends to $\wt{z}_2$ in the domain $\olC\setminus\ol{\Omega'}$
	
	%since $\alpha_1\setminus\{\wt{z}_2\}$ is disjoint from $\ol{U}_2$, the arc $\alpha_1$ approaches $\wt{z}_2$ from a Fatou domain in $\tu{Dom}(\wt{z}_2)$ contained in $W$ by the definition of type 2a points. 
	In what follows, we aim to turn $\alpha_1$ into a clean arc $\alpha_\infty$ joining ${z}_1$ and $\wt{z}_2$ with its interior still contained in $W$.
	
	The construction is similar to that of Proposition \ref{prop:exist_well_regulated_arc}. We provide only a sketch. Consider the non-trivial case that $\tu{Dom}(z_1)\cap \tu{Dom}(\wt{z}_2)=\es$. Then the arc $\alpha_2:=\tu{First-last}(\alpha_1)$ with its interior in $W$ is pre-clean. We enumerate all the Fatou domains except $U_1$ and $U_2$ such that
	$$\tu{diam}\,U_3\geq \tu{diam}\,U_4\geq\cdots.$$
	For $n\geq 3$, we inductively choose a small disk $D_n\Subset W$ around the dusty point $\alpha_n\cap\partial U_n$ (if any) such that $D_n$ is disjoint from $U_1,\ldots, U_{n-1}$ and it has some other properties as stated in the proof of Proposition \ref{prop:exist_well_regulated_arc}. Then we take modifications on $\alpha_n$ just within $D_n$ such that the resulting arc $\alpha_{n+1}$ has no dusty point from $\partial U_n$ by Lemma \ref{lem:pre-clean-property}\,(2)(3).
	
	Let $\alpha_\infty$ be the limit of $\alpha_n$. 
	Then $\tu{int}(\alpha_\infty)$ has no dusty point and is contained in $W$. Note that $\tu{int}(\alpha_\infty)\cap \ol{U}_i=\es, i=1,2$, and $\alpha_\infty$ approaches $\wt{z}_2$ along an internal ray
 if and only if $\wt{z}_2$ is of type 2a. Let $\wt{\gamma}:=\alpha_\infty\cup R$, where $R$ is the only internal ray of $U_2$ landing at $\wt{z}_2$. Then $\wt{\gamma}$ is clean. Thus, to show $\wt{\gamma}$ is properly clean for $\Omega$, we only need to verify that $\wt{\gamma}$ satisfies the conditions at the endpoint $z_1$.
	
If $\Omega$ contains a Fatou domain of $\tu{Dom}({z}_1)$, by the modification for $\gamma$ at the beginning near the point ${z}_1$, it holds that $\wt{\gamma}$ tends to ${z}_1$ along an internal ray. Suppose that $\Omega$ does not contain an element of $\tu{Dom}(z_1)$ and consider the two cases: $\#\tu{Acc}({z}_1, U_1)=1$ and $\#\tu{Acc}({z}_1,U_1)\geq 2$.
		
	If $\#\tu{Acc}({z}_1, U_1)=1$, then the domains of $\tu{Dom}({z}_1)\setminus\{U_1\}$ are contained in $\Omega$ and so ${z}_1$ is of type 1a under the assumption.
	If $\#\tu{Acc}({z}_1, U_1)\geq 2$, then $\wt{\gamma}$ tends to ${z}_1$ along a special entrance bounded by two internal rays of $U_1$. Thus $\wt{\gamma}$ satisfies the joint condition at the point $z_1$. The proof is complete.	
\end{proof}

\subsection{$\Omega$-clean arcs}
Consider two Fatou centers $z_1$ and $z_2$ from distinct components of $\tu{int}(K_U)~(=\wh{\mb{C}}\setminus \ol{U})$ for a Fatou domain $U$. Any clean arc joining $z_1$ and $z_2$ must pass through the center $c(U)$. However, in many cases, we need a regulated arc joining $z_1$ and $z_2$ within $K_U$, just like the DH-regulated arcs in the polynomial case. This motivates us to define $\Omega$-clean arcs relative to a Fatou domain $U$.

\begin{definition}[$\Omega$-clean arcs]\label{def:quasi_regulated}
	\tu{Let $f$ be a PCF and $U$ be a Fatou domain. Let $\gamma$ be an arc in $K_U$ with two endpoints in the union of Fatou centers and $\partial U$. Then $\gamma$ is called \emph{$\Omega$-clean (relative to $U$)}, provided that,
for any component $\Omega$ of $\tu{int}(K_U)$, if $\g_\Omega:=\g\cap\ol{\Omega}$ contains at least two points, then
 $\gamma_\Omega$ is either
 \begin{itemize}
 	\item a properly clean arc for $\Omega$ (then the Fatou-type endpoint of $\gamma_\Omega$ is that of $\gamma$), or
 	\item the union of two properly clean arcs for $\Omega$ with equal Fatou-type endpoints and their interiors disjoint.
 \end{itemize}}
\end{definition}

\begin{remark}\label{rem:K_U-regulated}\tu{(1) There is an overlap between clean arcs and $\Omega$-clean arcs: any $\Omega$-clean arc contained in the closure of a component of $\tu{int}(K_U)$ is clean; conversely, any properly clean arc is $\Omega$-clean.}

\tu{(2) If $f$ is a postcritically finite polynomial and $U$ is the basin of infinity, then $\Omega$-clean arcs relative to $U$ are just the DH-regulated arcs.}
\end{remark}
\begin{lemma}
	Every $\Omega$-clean arc satisfies the joint condition.
\end{lemma}
\begin{proof}
	Let $z$ be an accessible point in an $\Omega$-clean arc $\gamma$.  If $z\in \tu{int}(K_U)$, the joint condition for $\gamma$ at the point $z$ is satisfied since $z$ belongs to a clean arc by definition. In the following we assume $z\in\partial U$.
	
	Let $S:=\bigcup_{R\in \tu{Acc}(z,U)} \ol{R}$. Note that $S\cap \gamma=\{z\}$, and the one or two components of $\gamma\setminus\{z\}$
%	There is a one-to-one corresponding between $\tu{Comp}(K_U\setminus\{z\})$ and
are contained in distinct components of $\olC\sm S$. We assume $R_1, R_2\in \tu{Acc}(z,U)$ bound a component $W$ of $\olC \sm S$ that intersects $\gamma$. Note that $R_1=R_2\Leftrightarrow \#\tu{Acc}(z,U)=1\Rightarrow z\in\tu{end}(\gamma)$. Let $\alpha$ be the unique component of $\gamma\setminus\{z\}$ contained in $W$.
	
	If $W$ contains no Fatou domains of $\tu{Dom}(z)$, then $R_1$ and $R_2$ are adjacent. Thus $\alpha$ tends to the point $z$ along a special entrance. Otherwise, the arc $\alpha$ tends to $z$ along an internal ray by definition. With the same argument on $\wt{\alpha}:=\gamma\setminus\ol{\alpha}$ if $z\in\tu{int}(\g)$, we conclude that $\gamma$ satisfies the joint condition at the point $z$. From the arbitrariness of the choice of $z$, the lemma follows.
\end{proof}

\begin{proposition}[Existence of $\Omega$-clean arcs]\label{lemma:quasi_regulated}
Let $f$ be a PCF and $U$ be a Fatou domain. Let $\g$ be an arc in $K_U$ with endpoints $z_1$ and $z_2$
 in the union of Fatou centers and $\partial U$. Then for any $\epsilon>0$, there exists an $\Omega$-clean arc $\g_\infty$ relative to $U$ joining $z_1$ and $z_2$ such that
 $$\g_\infty\cap\mc{J}_f\subseteq \mc{N}(\g\cap\mc{J}_f,\epsilon).$$
\end{proposition}
\begin{proof}
We enumerate the components of $\tu{int}(K_U)$ as $\Omega_0,\Omega_1, \ldots$. Let $\gamma_0:=\gamma$. For $n\geq 0$, if the intersection $\gamma_{n}\cap\ol{\Omega}_n$ contains at most one point, let us define $\gamma_{n+1}:=\gamma_{n}$. Otherwise, we obtain the distinct first-in and last-out places of $\gamma_{n}$ meeting $\ol{\Omega}_n$, denoted by $w_1$ and $w_2$, respectively. Note that for each $i\in\{1,2\}$, $\tu{the point $w_i$ is a Fatou center}\Leftrightarrow w_i\in \Omega_n\Leftrightarrow w_i\in\{z_1,z_2\}\cap\mc{F}_f.$

We claim that $[w_1,w_2]_{\gamma_n}=\gamma_n\cap\ol{\Omega}_n$. Indeed, if it is not true, there is a component $\beta$ of $\gamma_n\setminus \ol{\Omega}_n$, whose two distinct endpoints are contained in $\partial \Omega$, such that $\beta\cup\partial \Omega$ bounds a Jordan domain $W$ containing $\Omega_n$ with $W\sm\Omega_n\neq \es$. Since $K_U$ is full and $\partial W\subseteq K_U$, the domain $W$ belongs to a component of $\tu{int}({K_U})$, which strictly contains $\Omega_n$. We arrive at a contradiction. %The claim follows.

By the claim we can pick an arc $\wt{\beta}$ with $\tu{int}(\wt{\beta})\subseteq \Omega_n$ joining $w_1$ and $w_2$ such that $\wt{\beta}\cap\mc{J}_f\subseteq \mc{N}(\gamma_n\cap\mc{J}_f,\epsilon/2)$. By taking necessary modification on a suitable subarc of $\wt{\beta}$, we assume further that $\wt{\beta}$ passes through a Fatou center, say $z$.

By Corollary \ref{cor:two_centers} when $w_i\in\Omega_n$ or by Proposition \ref{prop:prop} otherwise, we obtain two clean arcs $\alpha_i$ joining $w_i$ to $z$ such that the two sets $\alpha_1\cap\mc{J}_f$ and $\alpha_2\cap\mc{J}_f$ are disjoint; moreover, $\alpha_i$ are properly clean for $\Omega_n$ whenever $w_i\in\partial \Omega_n$.
	Let $w$ be the first-in place of $\alpha_1$ meeting $\alpha_2$.	By construction, $w$ is a Fatou center. We replace the segment $[w_1, w_2]_{\g_n}$ of $\gamma_n$ with $$[w_1,w]_{\alpha_1}\cup [w, w_2]_{\alpha_2}$$ and keep the other points of $\gamma_n$ unchanged. The resulting arc is defined as $\gamma_{n+1}$.
	
	By induction, we obtain a sequence of arcs $\gamma_n$. By Lemma \ref{lem:new}\,(1), the arcs $\gamma_n$ uniformly converge to a limit $\gamma_\infty$. From the construction, we see that $\gamma_\infty$ is an $\Omega$-clean arc relative to $U$ joining $z_1$ and $z_2$. The proof is complete.
\end{proof}

\begin{lemma}\label{fact:1}
	Let $U$ be a Fatou domain. Let $\gamma$ be a clean arc whose endpoints are not semi-buried in $\partial U$, or an $\Omega$-clean arc relative to $U'\neq U$. Then $\gamma\cap\partial U$ contains at most finitely many points, and $\gamma\setminus U$ is covered by the closures of finitely many components of $\tu{int}(K_U)$. 
\end{lemma}

\begin{proof}
	We first assume that $\g$ is clean with $\tu{end}(\g)$ not semi-buried in $\partial U$. By the definition of clean arcs, if $z\in\tu{int}(\g)\cap \partial U$, then the internal ray of $U$ landing at $z$ is contained in $\g$. It follows that $\tu{int}(\gamma)\cap \partial U$ contains at most two points.
	Hence $\tu{int}(\g)\sm \partial U$ has at most three components, and each of them is contained in either $U$ or a component of $\tu{int}(K_U)$. Combining the assumption that points in $\tu{end}(\g)$ are not semi-buried in $\partial U$, we need at most five components of $\tu{int}(K_U)$ such that their closures cover $\g\sm U$.
		
	If $\gamma$ is $\Omega$-clean relative to $U'\neq U$, we may assume $U\subseteq\Omega\in \tu{int}(K_{U'})$ and $\#\gamma\cap \ol{\Omega}\neq \es$. The lemma holds clearly if $\#(\gamma\cap \ol{\Omega})=1$. Otherwise, by definition $\gamma_\Omega:=\gamma\cap\ol{\Omega}$ is formed by one or two properly clean arcs. In particular, $\g_\Omega$ is a clean arc and the Julia-type endpoints of $\gamma_\Omega$ cannot be semi-buried in $\partial U$.
	Since $\gamma\cap \partial U=\gamma_\Omega\cap \partial U$, by the first case we have $\#(\g\cap \partial U)\leq 4$ and $\g_\Omega\sm U$ is covered by the closures of at most five components of $\tu{int}(K_U)$. Note that each component of $\gamma\setminus\gamma_\Omega$ is disjoint from $\ol{U}$ and so is contained in a component of $\tu{int}(K_U)$. Then the proof is complete. 	 
\end{proof}

The following lemma will be used in the proof of Proposition \ref{prop:new}.

\begin{lemma}\label{lem:dusty_linking}
	Let $U$ be a Fatou domain and $z\in \partial U$. Then for any $\epsilon>0$, there exists $\delta>0$ such that any two points $w_1, w_2\in\partial U$ with $\tu{dist}(w_i,z)<\delta$ can be connected by an $\Omega$-clean arc $\wt{\gamma}$ relative to $U$ with $\tu{diam}\,\wt{\gamma}\setminus S<\epsilon$, where $S:=\bigcup\,\{V,V\in\tu{Dom}(z)\}$.
\end{lemma}
\begin{proof}
	It is known that every compact and locally connected metric space is locally arcwise connected; see \cite[Lemmas 17.17 and 17.18]{Mi1}. Then $K_U$ is locally arcwise connected. There is an arcwise connected and open (relative to $K_U$) subset $W$ of $K_U$ containing the point $z$ such that $\tu{diam}\,W< \epsilon/3$. Let $\mc{U}$ be the collection of Fatou domains that intersect $W$ but are not in $\tu{Dom}(z)$. By choosing $W$ sufficiently small, we require further that $\tu{diam}\,V<{\epsilon}/{3}$ for all $V\in \mc{U}$.
	
	% One may require further that the diameter of any Fatou domain, which intersects $W$ but is not in $\tu{Dom}(z)$, is less than $\epsilon/2$.
	
	Choose a small number $\delta>0$ such that the set $D:=\{w\in \partial U:\tu{dist}(w,z)< \delta\}$ is nested in $W$. For any two points $w_1$ and $w_2$ in $D$, there is an arc $\gamma$ in $W$ joining them. From $\gamma$, by Proposition \ref{lemma:quasi_regulated}, we obtain an $\Omega$-clean arc $\wt{\gamma}$ relative to $U$ joining $w_1$ and $w_2$, such that $\wt{\g}\cap\mc{J}_f\subseteq W$. Note that $\wt{\gamma}$ is covered by the set $S\cup W\cup\bigcup_{V\in\,\mc{U}}V$. We thus have the estimate
	$$\tu{diam}\,\wt{\gamma}\setminus S\leq \tu{diam}\,W+2\,\sup\,\{\tu{diam}\,V: V\in\,\mc{U}\}<\epsilon.$$
	The proof is complete.
\end{proof}

In the end of this section, we give the definition of \emph{hull} for components of $\tu{int}(K_U)$ and semi-buried points in $\partial U$, which will be used in the proof of Theorem \ref{thm:admissible}.

Let $\Omega_1,\ldots,\Omega_l, l\geq 1$ be distinct components of $\tu{int}(K_U)$ and $Q$ be a set of finitely many semi-buried points in $\partial U$ (possibly $Q=\es$). A \emph{hull} of $\Omega_1,\ldots,\Omega_l$ and $Q$ is some minimal continuum in $K_U$ containing $\Omega_1,\ldots,\Omega_l$ and $Q$. The precise construction is as follows.

We first consider the case that $Q=\emptyset$. If $l=1$, then $[\Omega_1]:=\ol{\Omega_1}$ is a \emph{hull} of $\Omega_1$. If $l\geq 2$, by induction we assume $[\Omega_1,\ldots,\Omega_{l-1}]$ is a hull. If $\Omega_l\cap [\Omega_1,\ldots,\Omega_{l-1}]\neq\emptyset$, then $$[\Omega_1,\ldots,\Omega_{l}]:=[\Omega_1,\ldots,\Omega_{l-1}]\cup\ol{\Omega}_l.$$
Otherwise, we choose an $\Omega$-clean arc $\gamma$ relative to $U$ joining a Fatou center in $\Omega_1$ to another one in $\Omega_l$ by Proposition \ref{lemma:quasi_regulated}. Let $w_1$ (resp. $w_2$) be the last-out (resp. first-in) place of $\gamma$ meeting $[\Omega_1,\ldots,\Omega_{l-1}]$ (resp. $\ol{\Omega}_l$). Then we define in this case
$$[\Omega_1,\ldots,\Omega_l]:=\left\{\begin{array}{ll}[\Omega_1,\ldots,\Omega_{l-1}]\cup[w_1,w_2]_{\gamma}\cup\ol{\Omega}_l\cup\ol{\Omega}, & {\textup{if $w_1\in\Omega\in\tu{Comp}(\tu{int}(K_U))$};} \\
{[\Omega_1,\ldots,\Omega_{l-1}]\cup[w_1, w_2]_\gamma\cup \ol{\Omega}_l\cup \ol{\Omega},} & {\textup{if $w_1\in \partial \Omega$ and $\Omega\cap [w_1, w_2]_\g\neq \es$;}}\\
{[\Omega_1,\ldots,\Omega_{l-1}]\cup[w_1, w_2]_{\gamma}\cup\ol{\Omega}_l}, & \textup{otherwise.}
\end{array}\right.$$
The set $[\Omega_1,\ldots,\Omega_l]$ is called a \emph{hull of $\Omega_1,\ldots,\Omega_l$}. 

Now we consider the case that $Q\neq \es$. Assume $Q=\{z_1,\ldots, z_s\}$. Let $Q'=\es$ if $s=1$ and $Q'=\{z_1,\ldots, z_{s-1}\}\subset Q$ otherwise. By induction we assume $H:=[\Omega_1, \ldots, \Omega_l; Q']$ is a hull. If $z_s\in H$, then 
$[\Omega_1,\ldots, \Omega_l; Q]:=H$. If $z_s\notin H$, let $\gamma$ be an $\Omega$-clean arc relative to $U$ joining a Fatou center in $\Omega_1$ to the point $z_s$. Let $w$ be the last-out place of $\gamma$ meeting $H$. Similarly as above, let us define
$$[\Omega_1,\ldots,\Omega_l;Q]:=\left\{\begin{array}{ll}H\cup[w,z_s]_{\gamma}\cup\ol{\Omega}, & {\textup{if $w\in\Omega\in\tu{Comp}(\tu{int}(K_U))$};} \\
{H\cup[w, z_s]_\gamma\cup\ol{\Omega},} & {\textup{if $w\in \partial \Omega$ and $\Omega\cap [w_1, w_2]_\g\neq \es$;}}\\
{H\cup[w, z_s]_{\gamma}}, & \textup{otherwise.}
\end{array}\right.$$
The set $[\Omega_1,\ldots,\Omega_l; Q]$ is called a \emph{hull} of $\Omega_1,\ldots, \Omega_l$ and $Q$.

There are many choices of hulls. However, it always holds that $H\cap \partial U=\wt{H}\cap \partial U$ for distinct hulls $H$ and $\wt{H}$ of given $\Omega_1,\ldots,\Omega_l$ and $Q$.

\begin{lemma}\label{lem:joint}
	Let $H$ be a hull of finitely many components of $\tu{int}(K_U)$ and a finite set of semi-buried points in $\partial U$. Assume $\tu{int}(H)=\Omega_1\cup \cdots\cup \Omega_l$. Let $G_i$ be regulated graphs with $\ol{G_i\cap \Omega_i}=G_i$ for $1\leq i\leq l$. Suppose $G:=\left(H\setminus \ol{\tu{int}(H)}\right)\cup \bigcup G_i$ is a graph. If each $G_i$ satisfies the joint condition at accessible points, then $G$ also satisfies the joint condition at every accessible point. Moreover, every branch point of $G$ either is a branch point of a graph $G_i$ or is multi-accessible. 
\end{lemma}
\begin{proof}
	Let $\alpha$ be the closure of a component of $H\setminus\ol{\tu{int}(H)}$. Then $\alpha$ is a tree. The family of edges $\mc{E}_\alpha$ of $\alpha$ consists of $\Omega$-clean arcs relative to $U$ with endpoints in $\partial U$. Since any edge in $\mc{E}_\alpha$ intersects $G_i$ only possibly at the endpoints, then $G$ satisfies the joint condition at every accessible point in the interior of such an edge. 
	
	Let $\mc{E}=\{G_1,\ldots, G_l\}\cup \bigcup\mc{E}_\alpha$, where $\alpha$ runs over all closures of components of $H\setminus\ol{\tu{int}(H)}$. Let $v$ be a common point of at least two elements in $\mc{E}$, say $T_1,\ldots T_n$. Clearly $v\in\partial U$. To complete the proof that $G$ satisfies the joint condition at every accessible point, we just need to show that $G$ satisfies the joint condition at the point $v$. 
	
	We claim that $T_i$ belong to distinct components of $K_U\setminus\{v\}$. For otherwise, we may assume $T_1, T_2\subset C\in\tu{Comp}(K_U\setminus\{v\})$. Since $C$ is open in $K_U$, which is locally arc-wise connected, the component $C$ is also locally arc-wise connected. Note that a locally arc-wise connected and connected metric space is arc-wise connected. Thus there is an arc $\beta$ in $C$ joining $T_1\setminus\{v\}$ and $T_2\setminus\{v\}$. Let $w_1$ (resp. $w_2$) be the last-out (resp. first-in) place of $\beta$ meeting $T_1$ (resp. $T_2$). Let $\beta_i$ be arcs in $T_i$ joining $v$ and $w_i$. Then $\beta_1\cup [w_1, w_2]_\beta\cup \beta_2$ forms a loop, which lies in the closure of a component $\Omega$ of $\tu{int}(K_U)$. 
	
	If $\Omega\notin \{\Omega_1,\ldots,\Omega_l\}$, then $T_1$ and $T_2$ are $\Omega$-clean arcs from $\bigcup_\alpha \mc{E}_\alpha$. Thus the two clean arcs $\gamma_i:=T_i\cap \ol{\Omega}$ are tending to the common endpoint $v$ within $\Omega$. It contradicts the construction of hulls.	
	If $\Omega\in\{\Omega_1,\ldots,\Omega_l\}$, we may assume $\Omega=\Omega_1$. Then one of $T_1$ and $T_2$ equals $G_1$, say $T_1$, since $G\cap \ol{\Omega}_1=G_1$. Note that every $\alpha$ intersects $\ol{\Omega}_1$ in at most one point. 
	We conclude that $T_2$ is not from $\bigcup_\alpha\mc{E}_\alpha$. Since the closures of two distinct components of $\tu{int}(K_U)$ have at most one common point, $T_2$ cannot be one of $G_2,\ldots, G_l$. Therefore, $T_1=T_2=G_1$, a contradiction. The claim is proved.
	
	By the claim, the subgraphs $T_1, \ldots, T_n$ are separated by the rays in $\tu{Acc}(v, U)$. Since each $T_i$ satisfies the joint condition at $v$, by definition the graph $G$ also has the joint condition at $v$. 
	
	Now we assume $v$ is a branch point of $G$ but is not a branch point for every subgraph $G_i$. Then $v\in \partial U$ and it is the common point of at least two elements in $\mc{E}$. Let $s:=\#\{T: v\in T\in\mc{E}\}$. By the argument above there are at least $s$ internal rays of $U$ landing $v$. Hence $v$ is multi-accessible if $s\geq 3$. We are done in this case. If $s=2$, assume $v\in T_1\cap T_2$. Since $v$ is a branch point of $G$, we may assume $\tu{valence}(T_1,v)\geq 2$. Thus $T_1$ is a graph in $\{G_1,\ldots, G_l\}$, say $G_1$. Since $G_1\subset\ol{\Omega_1}$ and $G_1$ satisfies the joint condition at $v$, the domain $\Omega_1$ contains at least a ray of $\tu{Acc}(v)$. Therefore, $\#\tu{Acc}(v)\geq 3$. The proof is complete. 
\end{proof}

\section{Existence of admissible graphs}\label{sec_existence}
After the preparations given in \S \ref{sec_arcs}, we aim to construct admissible graphs in this section; see Definition \ref{def:admissible} for admissible graphs.

\subsection{Connectable graphs}
A regulated graph $G$ usually intersects a component $\Omega$ of $\tu{int}(K_U)$ in a disconnected set. In the construction of admissible graphs, we need to add some clean arcs within $\ol{\Omega}$ into $G$ so that the disconnected set $G\cap\ol{\Omega}$ becomes connected. To ensure the new graph still satisfies the joint condition, the added arcs are required to be disjoint from $G\cap \mc{J}_f$. This motivates us to introduce the so-called \emph{connectable property} for graphs.

A graph $G$ is called \emph{connectable} if for every Fatou domain $U$ and every component $\Omega$ of $\tu{int}(K_U)$ with $\Omega\cap G
\neq \emptyset$, the set $\Omega\setminus (G\cap \mc{J}_f)$ is connected. See Figure \ref{fig:proper}.

\begin{figure}[h]
\centering
\begin{tikzpicture}
\node at (0,0){\includegraphics[width=0.9\linewidth]{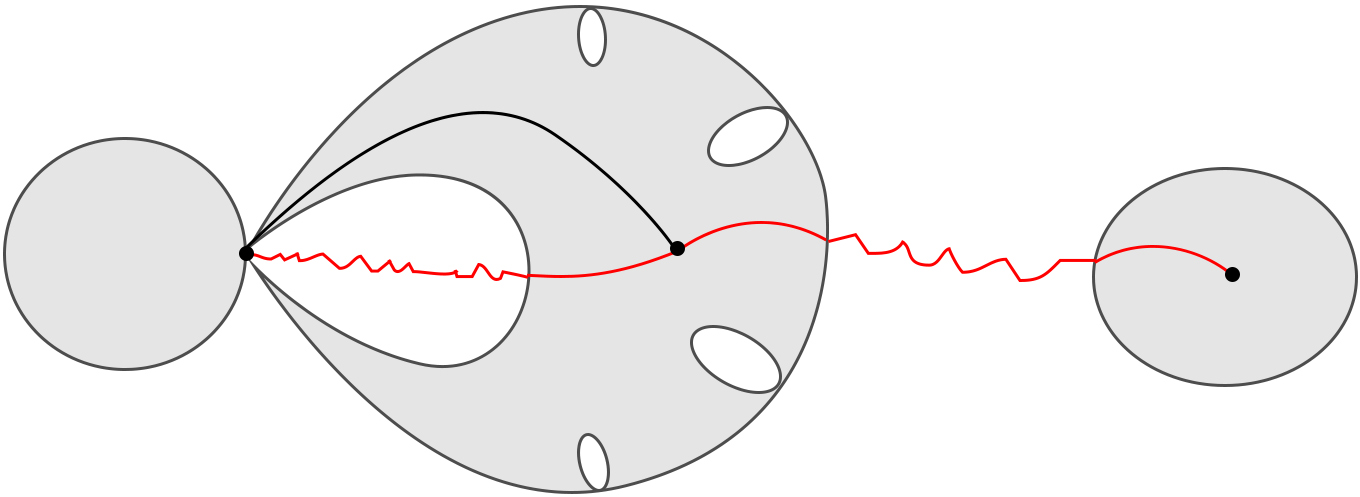}};
\node at (-3.25, -0.450){$\alpha$};
\node at (-5.5,0.25) {$U_*$};
\node at (-5.5,1.75) {$\Omega_*$};
\node at (-2.75,0.25) {$\Omega$};
\node at (-0.75,0.25) {$U$};
\node at (-2,1.75) {$R$};
\node at (6,0.25) {$U_0$};
\node at (3,0.25) {$\gamma$};
\node at (-4.55,0.27) {$z$};
%\ruler{7}{2};
\end{tikzpicture}
\caption{The curve $\gamma$, colored red, is a clean arc joining a point $z\in \mc{J}_f$ to a Fatou center $c(U_0)$ with $\tu{Dom}(z)=\{U,U_*\}$. As shown, $\gamma$ passes through just two Fatou domains $U$ and $U_0$. Here $\gamma$ is not connectable, since $\Omega\setminus(\gamma\cap\mc{J}_f)$ is disconnected. The other arc $\wt{\gamma}:=\{z\}\cup R\cup [c(U),c(U_0)]_{\gamma}$ is properly clean for $\Omega_*$. By Lemma \ref{lem:strong-regulated}\,(3), $\wt{\gamma}$ is connectable.}
\label{fig:proper}
\end{figure}
The next three lemmas tell us when a graph is connectable.
\begin{lemma}\label{lem:strong-regulated}
	The following statements hold.
	\begin{itemize}
		\item[(1)] Every clean arc with two Fatou-type endpoints is connectable.
		\item[(2)] Every clean arc joining a Fatou center to a buried point is connectable.
		\item[(3)] Every properly clean arc is connectable.
		\item[(4)] Every $\Omega$-clean arc is connectable.
		\item[(5)] Let $\wt{G}$ be the union of a connectable graph $G$ and finitely many internal rays. If $\wt{G}$ is a graph, then $\wt{G}$ is connectable.
	\end{itemize}
\end{lemma}
\begin{proof}
	Note that for an arc $\gamma$ and a Jordan domain $\Omega$, the closure of every component of $\gamma\cap \Omega$ is an arc whose endpoints lie in either $\partial \Omega$ or $\Omega$. Then, to show $\Omega\setminus(\g\cap\mc{J}_f)$ is connected, it is enough to consider each component $\alpha$ of $\gamma\cap\Omega$ whose two endpoints lie in $\partial \Omega$ and to prove $\alpha\cap\mc{F}_f\neq \es$.
 In the following we always assume that $U$ is any given Fatou domain, $\Omega$ is any component of $\tu{int}(K_U)$ such that $\Omega\cap\gamma\neq \es$, and $\alpha$ is a component of $\gamma\cap \Omega$ with $\tu{end}(\alpha):=\{z_1,z_2\}\subseteq \partial \Omega$.
	
	When $\gamma$ is clean, we claim that at least one of the two points $z_1$ and $z_2$ lies in $\tu{end}(\gamma)$.
	For otherwise, by definition both $z_1$ and $z_2$ are of type 1a or 2a in $\tu{int}(\gamma)$; moreover, the internal rays $R_1$ and $R_2$ of $U$ landing at $z_1$ and $z_2$ respectively belong to $\gamma$. It implies that $\gamma$ contains a loop $R_1\cup R_2\cup \ol{\alpha}$, a contradiction. The claim follows.
	
	In the statements of (1) and (2), the two endpoints of $\gamma$ are either Fatou centers or buried points in $\mc{J}_f$. By the claim, we would get a contradiction if $\alpha$ and $\Omega$ exist.
	
	(3) By definition, we may assume $\gamma$ is clean with a Julia-type endpoint $w_1\in\partial \Omega_*$ and a Fatou-type endpoint $w_2\in\Omega_*$ such that $\gamma\setminus\{w_1\}\subseteq\Omega_*$ for a component $\Omega_*$ of $\tu{int}(K_{U_*})$. By assumption we have $U\neq U_*$ and $U\subseteq \Omega_*$. By the claim, either $z_1=w_1$ or $z_2=w_1$. We assume $z_1=w_1$. Then $\{U_*,U\}\subseteq \tu{Dom}(z_1)$. Since $\gamma$ is properly clean for $\Omega_*$ and $U\subseteq\Omega_*$, it approaches $z_1$ along an internal ray and hence $\alpha\cap \mc{F}_f\neq \es$. Statement (3) holds.
	
 (4) Let $\gamma$ be an $\Omega$-clean arc relative to $U_*$ in statement (4). If $U=U_*$, $\alpha\cap \mc{F}_f\neq \emptyset$ by definition.
	
	 If $U\neq U_*$, we assume $U\subseteq \Omega_*$ for some $\Omega_*\in\tu{Comp}(\tu{int}(K_{U_*}))$. Then $\partial\Omega\subseteq \partial U\subseteq \ol{\Omega}_*$ and so $\Omega\subseteq \Omega_*$. It follows that $\alpha$ belongs to the clean arc $\gamma_{\Omega_*}:=\gamma\cap \ol{\Omega}_*$. By the claim again, one of $z_1$ and $z_2$ must be an endpoint of $\gamma_{\Omega_*}$, say $z_1$. By definition $z_1\in \partial \Omega_*$ and then $\{U, U_*\}\subseteq\tu{Dom}(z_1)$. Thus $\gamma_{\Omega_*}$ approaches $z_1$ along an internal ray. This also implies $\alpha\cap\mc{F}_f\neq \es$. Statement (4) holds.
	 
	 Statement (5) follows immediately by the equality $\Omega\setminus(\wt{G}\cap \mc{J}_f)=\Omega\setminus(G\cap\mc{J}_f)$.
\end{proof}

\begin{lemma}\label{lem:graph-connectable}
		Let $G_1$ and $G_2$ be two regulated graphs. If the two sets $G_1\cap \mc{J}_f$ and $G_2\cap \mc{J}_f$ are disjoint but $G_1\cap G_2\neq \emptyset$, then $G_1\cup G_2$ is a regulated graph. Furthermore, if $G_1$ and $G_2$ are  connectable, then $G_1\cup G_2$ is also  connectable.
\end{lemma}
\begin{proof}
	 The condition implies that $G_1\cap G_2$ consists of only Fatou centers. For otherwise, since the intersection of a regulated graph with $\mc{F}_f$ is a union of internal rays, $G_1\cap G_2$ would contain an internal ray $R$; thus the landing point of $R$ would be also contained in $G_1\cap G_2$, a contradiction.
	
	Note that if a Fatou domain meets both $G_1$ and $G_2$, then its closure must join the two sets $G_1\cap \mc{J}_f$ and $G_2\cap \mc{J}_f$, and its diameter is thus greater than the positive number $\epsilon_0:=\tu{dist}(G_1\cap\mc{J}_f,G_2\cap\mc{J}_f).$
	By Lemma \ref{lem:new}\,(1), there are at most finitely many such Fatou domains. Thus the intersection $G_1\cap G_2$ is a finite set.	Therefore, $\wt{G}:=G_1\cup G_2$ is a regulated graph.
	
	To show $\wt{G}$ is connectable, we argue by contradiction and assume it is not true. By definition, there is a component $\Omega$ of $\tu{int}(K_U)$ such that $\wt{G}\cap \Omega$ contains either a loop $\beta\subset \mc{J}_f$ or an open arc $\beta\subset \mc{J}_f$ with $\tu{end}(\beta)\subset\partial \Omega$.   Since $G_i$ are connectable, we see that $\beta$ is neither a subset of $G_1$ nor that of $G_2$. Thus $\beta$ contains a point of $G_1\cap G_2$, which is a Fatou center as shown above. This contradicts $\beta\subset\mc{J}_f$. The proof is complete.

%	Then by definition there is a component $\Omega$ of $\tu{int}(K_U)$ such that $\Omega\setminus (\wt{G}\cap\mc{J}_f)$ is disconnected. Let $\beta:=\wt{G}\cap\mc{J}_f$. Since $G_i$ are connectable, we see that $\beta$ is neither a subset of $G_1$ nor that of $G_2$. Thus $\beta$ contains a point of $G_1\cap G_2$, which is a Fatou center as shown above. This contradicts $\beta\subset\mc{J}_f$. The proof is complete.
\end{proof}

\begin{lemma}\label{lem:hull-connectable}
	Let $H$ be a hull of finitely many components of $\tu{int}(K_U)$ and a finite set of semi-buried points in $\partial U$. Assume $\tu{int}(H)=\Omega_1\cup \cdots\cup \Omega_l$. Let $G_i$ be regulated graphs with $\ol{G_i\cap \Omega_i}=G_i$ for $1\leq i\leq l$. Suppose $G:=\left(H\setminus \ol{\tu{int}(H)}\right)\cup \bigcup G_i$ is a graph. Then $G$ is connectable if $G_i$ are connectable.
\end{lemma}

\begin{proof}
	We fix an arbitrary Fatou domain $U'$ and a component $\Omega'$ of $\tu{int}(K_{U'})$. The argument goes as follows.
	
%	If $U'=U$ and $\Omega'\in\{\Omega_1,\ldots,\Omega_l\}$, say $\Omega'=\Omega_{1}$, then $\Omega'\cap G=\Omega'\cap G_1$. Since $G_1$ is connectable, the set $\Omega'\setminus (G\cap \mc{J}_f)=\Omega'\setminus(G_1\cap\mc{J}_f)$ is connected.
	
%	If $U'=U$ and $\Omega'\notin\{\Omega_1,\ldots,\Omega_l\}$, assume $\Omega'\cap G\neq \es$. By the definition of hulls, the closure of every component of $H\setminus\ol{\tu{int}(H)}$ is a tree whose edges are $\Omega$-clean arcs relative to $U$. Thus $\gamma_{\Omega'}:=\ol{\Omega'\cap G}$ is an $\Omega$-clean arc. By Lemma \ref{lem:strong-regulated}\,(4), $\Omega'\setminus(G\cap\mc{J}_f)=\Omega'\setminus(\gamma_{\Omega'}\cap\mc{J}_f)$ is connected.
	
%	If $U'\neq U$, assume $U'\subseteq \Omega\in\tu{Comp}(\tu{int}(K_{U}))$. There are two cases: $\Omega'\subseteq \Omega$ or $\Omega'\nsubseteq \Omega$. In the former case, $\Omega'\setminus(G\cap\mc{J}_f)=\Omega'\setminus (G\cap \Omega\cap \mc{J}_f)$. By the discussions above, $G\cap\ol{\Omega}$ equals either $G_i$ for some $1\leq i\leq l$ or an $\Omega$-clean arc $\gamma_{\Omega}$. Since both of them are connectable, the set $\Omega'\setminus(G\cap\mc{J}_f)$ is connected.
	
	If $\Omega'\setminus(G\cap \mc{J}_f)$ is disconnected, let $W$ be a component of $\Omega'\setminus(G\cap\mc{J}_f)$. Then $\partial W$ contains either an arc $\beta\subset G\cap \mc{J}_f\cap\ol{\Omega'}$ with $\tu{end}(\beta)\subset \partial \Omega'$ and $\tu{int}(\beta)\subset\Omega'$, or a loop $\beta\subset G\cap \mc{J}_f\cap\Omega'$, or a loop $\beta\subset G\cap \mc{J}_f\cap\ol{\Omega'}$ with $\#(\beta\cap \partial{\Omega'})=1$.
	
	If $\beta$ is a loop, let $\beta'=\beta$. Otherwise, let $\beta'$ be the union of $\beta$ and an arc in $\partial\Omega'$ joining $\tu{end}(\beta)$. So $\beta'$ is a loop in $\mc{J}_f$. So one of the open disks that $\beta'$ bounds contains $U$ and the other is in a component $\Omega$ of $K_U$.% any component $\wt{\Omega}$ of $\tu{int}(K_U)$ intersecting $\beta$ implies that $\wt{\Omega}\setminus\beta$ is disconnected. Since $\wt{\Omega}\sm \beta\supseteq\wt{\Omega}\sm(G\cap\mc{J}_f)$ and $\wt{\Omega}\sm(G\cap\mc{J}_f)$ is shown to be connected in the above arguments, such a component $\wt{\Omega}$ does not exist. Thus $\beta\subset \partial U\cap G$ in this case. Since $\partial U\cap G$ has no loops, $\beta$ cannot be a loop.
	
	%Note that any component $\wt{\Omega}$ in $\tu{Comp}(\tu{int}K_U)\setminus\{\Omega\}$ intersecting $\beta$ implies that $\wt{\Omega}\setminus \beta$ is disconnected. Since $\wt{\Omega}\setminus \beta\supseteq \wt{\Omega}\setminus(G\cap\mc{J}_f)$ and $\wt{\Omega}\setminus(G\cap\mc{J}_f)$ is shown to be connected in the above arguments, such a component $\wt{\Omega}$ does not exist. Thus $\beta\subset \partial U\cap G$. By the conditions, we see that $\partial U\cap G$ has no loops. So $\beta$ cannot be a loop.
%	We have shown that $\beta$ is an arc with $\tu{end}(\beta)\subset \partial \Omega'$ and $\tu{int}(\beta)\subset\Omega'$. We claim that $\beta\subset\ol{\Omega}$. Indeed, there is an arc $\beta'\subset \partial \Omega'\subset \partial U'$ such that $\beta\cup\beta'$ forms a Jordan curve separating $U$ and $U'$. Clearly the complementary component of $\beta\cup\beta'$ containing $U'$ is contained in a component of $\tu{int}(K_U)$. This component must be $\Omega$ as $U'\subset \Omega$. Then the claim is proved.
	
	Therefore, $\beta$ is contained in the set $G\cap\mc{J}_f\cap \ov{\Omega}$. If $\Omega=\Omega_i,i\in\{1,\ldots,l\}$, then $G\cap\ov{\Omega}=G_i$, and $\Omega'\setminus (G_i\cap\mc{J}_f)\,(\subseteq \Omega'\setminus(\beta\cap\mc{J}_f))$ is disconnected. It contradicts that $G_i$ is connectable. If $\Omega\not\in \{\Omega_1,\ldots,\Omega_l\}$, by the definition of hulls the intersection $G\cap\ov{\Omega}$ is an $\Omega$-clean arc, which is connectable according to Lemma \ref{lem:strong-regulated}(4). We also get a contradiction.
	
	Since we have proved that $\Omega'\sm(G\cap \mc{J}_f)$ is always connected for any Fatou domain $U'$ and any component $\Omega'$ of $\tu{int}(K_{U'})$, the graph $G$ is connectable by definition. 
\end{proof}

\subsection{Proof of Theorem \ref{thm:admissible}}
\begin{proof}[Proof of Theorem \ref{thm:admissible}]
	The special case is that $\mc{F}_f$ is connected. After a conformal conjugation, the map $f$ is a polynomial without bounded Fatou domains in $\mb{C}$. Let $T$ be the DH-regulated convex hull of $P\setminus\{\infty\}$ within the dendrite $\mc{J}_f$. Then $T$ is a tree with $\tu{end}(T)\subseteq P$. Let $\mc{V}$ be the union of $P$ and the branch points of $T$. Then the required graph $G=(\mc{V}, \mc{E})$ can be chosen as the union of $T$ and all external rays in the basin of infinity that land at points in $\mc{V}$.
	
	In what follows, we assume $\mc{F}_f$ has at least two components. The construction proceeds in several steps.	

\tb{Step 0.} Select marked Fatou domains. We decompose $P\cap \mc{J}_f$ into six disjoint subsets:
	\begin{itemize}
	\item[] $P_1=\{z\in P \tu{ is buried in }\mc{J}_f.\}$
	\item[] $P_2=\{z\in P: \#\tu{Acc}(z, U)\geq 2, \tu{Dom}(z)=\{U\}\tu{ with $U$ periodic}.\}$
	\item[] $P_3=\{z\in P \tu{ is of type 1a}, \tu{Dom}(z)=\{U\}\tu{ with $U$ periodic}.\}$
	\item[] $P_4=\{z\in P \tu{ is of type 1b},  \tu{Dom}(z)=\{U\}\tu{ with $U$ periodic}.\}$
	\item[] $P_5=\{z\in P: \tu{$\tu{Dom}(z)=\{U\}$ with $U$ preperiodic.}\}$
	\item[] $P_6=\{z\in P: \#\tu{Dom}(z)\geq 2.\}$
\end{itemize}

%We  first choose the Fatou domains that intersect $P$ as marked ones. The collection of these Fatou domains is denoted by $\mc{U}_0$.

For a point $z\in P_1$, let $\Omega_{U,z}$ be the component of $\tu{int}(K_U)$ containing the point $z$ for a Fatou domain $U$. The finite intersection $\bigcap\,\Omega_{U,z}$, where $U$ runs over all periodic Fatou domains, is non-empty. Let $D_z$ be the component of the intersection containing the point $z$. We choose a Fatou domain $U_z$ in $D_z$ as a marked one associated to the point $z$. The collection of $U_z, z\in P_1$, is denoted by $\mc{U}_1$.

A point $z$ in $P_2$ disconnects $K_U$ into at least two components. We choose two Fatou domains from distinct components as marked domains. Let $\mc{U}_2$ be the collection of these marked domains, as the point $z$ runs over all points in $P_2$.

When $z\in P_3$, in the component of $\tu{int}(K_U)$ whose boundary contains $z$, we choose a Fatou domain as a marked one. Let $\mc{U}_3$ be the collection of them as $z$ is taken over all points in $P_3$.

Let $\mc{U}_4:=\bigcup_{z\in P_5\cup P_6}\tu{Dom}(z)$. Finally, the union $\mc{U}=\bigcup_{1\leq i\leq 4}~\mc{U}_i$ is the whole family of marked Fatou domains. By the way, since $\mc{F}_f$ has at least two Fatou domains, by adding some Fatou domains if necessary, we may assume that $\mc{U}$ contains $P\cap\mc{F}_f$ and $\#\mc{U}\geq 2$. Therefore, all periodic Fatou domains are contained in $\mc{U}$.
	
\tb{Step 1.} Connect the buried points in $P$.

For each point $z\in P_1$, there are the associated domains $D_z$ and $U_z$ chosen in Step 0. We draw an arc $\gamma_z\subset D_z$ joining $z$ to the center $c(U_z)$. We assume that the sets $\gamma_z\cap\mc{J}_f, z\in P_1$ are pairwise disjoint.

By Corollary \ref{cor:two_centers}, one can turn each $\gamma_z$ to a clean arc $\wt{\gamma}_z$ with the same endpoints as $\gamma_z$ such that
the two sets $\wt{\gamma}_z\cap \mc{J}_f$ and $\wt{\gamma}_w\cap \mc{J}_f$ are still disjoint for distinct $z,w\in P_1$. Clearly we can let $\wt{\gamma}_z\subset D_z$ by the choice of $D_z$.

Let $G_1:=\bigcup_{z\in P_1}\wt{\gamma}_z$. If $G_1$ is connected, then $G_1$ is a regulated and connectable graph by Lemmas \ref{lem:strong-regulated}\,(2) and \ref{lem:graph-connectable}.

\tb{Step 2.} Connect centers of the marked Fatou domains.
	
	Since $\mc{U}$ finite, we can pick a Jordan curve $\gamma$ which passes through all $c(U)$ for $U\in \mc{U}$.  We arrange these Fatou centers as $z_1, z_2,\ldots, z_n, z_{n+1}:=z_1$ in $\gamma$ in the cyclic order. As the set $\olC\setminus(G_1\cap \mc{J}_f)$ is connected, we may assume that $\gamma$ avoids $G_1\cap \mc{J}_f$.  According to Corollary \ref{cor:two_centers}, one can turn each $[z_k,z_{k+1}]_{\gamma}$, $1\leq k\leq n$,
	to a clean arc $\alpha_k$ joining $z_k$ and $z_{k+1}$, such that $$\alpha_i\cap G_1\cap \mc{J}_f=\es\tu{ and }(\alpha_i\cap\mc{J}_f)\cap (\alpha_j\cap\mc{J}_f)=\es\,\, \text{for }1\leq i\neq j\leq n.$$
By Lemmas \ref{lem:strong-regulated}\,(1) and \ref{lem:graph-connectable}, we obtain a connectable graph $G_2:=G_1\cup\bigcup_k\alpha_k$. The branch points of $G_2$ are Fatou centers, while the set of endpoints of $G_2$ coincides with $P_1$.

\tb{Step 3.} Turn periodic Fatou centers into non-cut points.
	
We enumerate the periodic Fatou domains as $U_0,\dotsc,U_{m}$. Let $G_2^{0}:=G_2$. This step is to inductively construct $G_2^{1},\ldots, G_2^{m+1}$ such that
  \begin{itemize}
 \item  [(i)] each $G_2^{k}$ is a connectable graph;
 \item [(ii)] each $G_2^k$ is the union of $G^0_2$, and $\tu{Acc}(z), z\in P_4\cap (\partial U_0\cup\cdots\cup \partial U_{k-1})$, and finitely many clean arcs with Fatou-type endpoints, and finitely many $\Omega$-clean arcs relative to $U_j$ for $0\leq j\leq k-1$; 
 \item [(iii)] the centers $c(U_0),\ldots,c(U_{k-1})$ are non-cut points for $G_2^{k}$;
% \item [(iii)] it holds that $\#\tu{Acc}(z)\geq3$ at every Julia-type branch point of $G_2^{k}$;
 \item [(iv)] $\tu{end}(G_2^{k})\cap \mc{J}_f=P_1$;
 
 \item[(v)] each $G_2^k$ satisfies the joint condition.
% \item [(iv)] $G_2^k$ is connectable graph.
   \end{itemize}

If we set $U_{-1}:=\emptyset$, then $G_2^0$ satisfies the above properties.
Since for a regulated graph $\G$
\begin{equation}\label{eq:non-cut}
\tu{the center $c(U_k)$ is a non-cut point $\Leftrightarrow$ $\G\setminus U_k~(=\G\cap K_{U_k})$ is connected,}
\end{equation}
to obtain $G_2^{k+1}$ with property (iii), we need to join all components of $G_2^k\cap K_{U_k}$ by clean arcs or $\Omega$-clean arcs relative to $U_k$.

Note that $G_2^0$ is formed by clean arcs whose Julia-type endpoints are $P_1$. From Lemma \ref{fact:1} and the inductive property (ii) for $G_2^k$, there are finitely many components $\Omega_1, \ldots, \Omega_l$ of $\tu{int}(K_{U_k})$ such that $\ol{\Omega_i}\cap G_2^k\neq \es$ and 
 $G_2^k\setminus U_k=\bigcup_{1\leq i\leq l}\ol{\Omega_i}\cap G_2^k$. Let $G_{\Omega_i}:=\ol{\Omega_i\cap G_2^k}$ and $G_{\Omega_i}':=\ol{\Omega_i}\cap G_2^k$. Note that $G_{\Omega_i}\subseteq G_{\Omega_i}'$ and $G_{\Omega_i}'\setminus G_{\Omega_i}$ (possibly empty) consists of finitely many points in $\partial \Omega_i$. 

\noindent\tb{Claim 1.}~\emph{Every component of $G_{\Omega_i}$ contains a Fatou center.}
\begin{proof}[Proof of Claim 1]
	If it is not true, assume $T$ is a component of $G_{\Omega_i}$ such that $T\subset \mc{J}_f$. Since $G_2^k$ is connectable, $T$ has no loops and so is a tree, moreover, the intersection $T\cap \partial \Omega_i$ contains at most one point. Thus there is an endpoint of $T$ contained in $\Omega_i$, say $z$. Clearly $z\in \tu{end}(G_2^k)\cap\mc{J}_f=P_1$ by property (iv).
	
	From the construction of $\wt{\gamma}_z$ in Step 1 for the buried point $z$, we known that $\wt{\gamma}_z\subset D_z\subset \Omega_i$. Thus $\wt{\gamma}_z$ is contained in a component of $G_{\Omega_i}$, which must be $T$. Hence $T$ contains the Fatou center $c(U_z)$, a contradiction.
\end{proof}
Let $H$ be a hull of $\Omega_1,\ldots,\Omega_l$ and $P_4\cap \partial U_k$ in $K_{U_k}$. For a component $\Omega$ of $\tu{int}(H)$, the set $\ol{H\setminus\ol{\Omega}}\cap\partial \Omega$ consists of finitely many points, which are precisely the cut points of $H$ within $\ol{\Omega}$. We denote them by $S_\Omega$.
There are two cases: either $\Omega\in\{\Omega_1,\ldots,\Omega_l\}$ or $\Omega\notin \{\Omega_1,\ldots,\Omega_l\}$.

If $\Omega\in\{\Omega_1,\ldots,\Omega_l\}$, similarly as Step 2, we first join all components of $G_{\Omega}$ by a Jordan curve $\gamma$ such that $\gamma\subset \Omega\setminus (G_\Omega\cap\mc{J}_f)$. This can be done by Claim 1. Then we piecewise turn $\gamma$ into clean arcs by Corollary \ref{cor:two_centers}. The union of these clean arcs and $G_\Omega$, still denoted by $G_{\Omega}$, is a connectable graph. If $G_\Omega=\es$, we just let $G_\Omega$ be a Fatou center in $\Omega$. For a point $z\in (S_\Omega\cup G_\Omega')\setminus G_\Omega$, we choose a properly clean arc $\gamma_z$ for $\Omega$ joining $z$ to a Fatou center in $G_{\Omega}$ by Proposition \ref{prop:prop}. One may assume $\gamma_z\setminus\{z\}\subset \Omega\setminus(G_\Omega\cap\mc{J}_f)$, since $G_\Omega$ is connectable. Finally we obtain a connectable graph $\wt{G}_\Omega\subset \ol{\Omega}$ with
\begin{equation}\label{eq:Omega}
\wt{G}_\Omega:=G_\Omega\cup\{\gamma_z: z\in (S_\Omega\cup G_\Omega')\setminus G_\Omega\}.
\end{equation}

If $\Omega\notin \{\Omega_1,\ldots,\Omega_l\}$, let $G_\Omega$ be an arbitrary Fatou center in $\Omega$. As above, we can obtain the properly clean arcs $\gamma_z$ for $\Omega$ joining the points of $S_\Omega$ to $G_\Omega$, and the graph $\wt{G}_\Omega$ is then defined as \eqref{eq:Omega}. Note that $\#S_\Omega\geq 3$ in this case.

Having finished these processes for each component of $\tu{int}(H)$, we define $G_2^{k+1}$ as the union of $G_2^k\cap U_k$ and $\bigcup_{z\in P_4\cap \partial U_{k}} R_z$, and the new graph
\begin{equation}\label{eq:subgraph}
\left(H\setminus\bigcup_{\Omega\in\mc{H}}\ol{\Omega}\right)\cup \bigcup_{\Omega\in\mc{H}} \wt{G}_{\Omega}
\end{equation}
where $\mc{H}$ is the collection of components of $\tu{int}(H)$ and $R_z$ is the unique ray in $\tu{Acc}(z)$. By construction, $\tu{end}(G_2^{k+1})\cap \mc{J}_f=P_1$.

Property (i) is clearly satisfied for $G_2^{k+1}$ by Lemmas \ref{lem:strong-regulated}, \ref{lem:graph-connectable} and \ref{lem:hull-connectable}. 
For property (ii), indeed, the added properly clean arcs $\gamma_z$ are $\Omega$-clean relative to $U_k$ by definition. Note that the closure of each component of $H\setminus\bigcup_{\Omega\in \mc{H}}\ol{\Omega}$ is a tree whose edges are $\Omega$-clean relative to  $U_k$ by definition of hulls. Hence property (ii) holds for $G_2^{k+1}$. Note that each $\wt{G}_\Omega, \Omega\in\mc{H}$ satisfies the joint condition at accessible points by property (v) for $G_2^{k}$. It follows from Lemma \ref{lem:joint} that the graph in \eqref{eq:subgraph} also satisfy the joint condition at each accessible point. Since $\tu{end}(G_2^{k+1})\cap\mc{J}_f=P_1$, property (v) holds for $G_2^{k+1}$  by definition.

 For property (iii), by construction and \eqref{eq:non-cut}, we know that $c(U_k)$ is a non-cut point of $G_2^{k+1}$. To see that this also holds for the points $c(U_0),\ldots,c(U_{k-1})$, one just need to check the following fact:
\begin{facts}\label{facts}
	Let $\Gamma$ be a graph and $\gamma$ be an arc such that $\tu{end}(\gamma)\subset \Gamma$. Suppose $\wt{\Gamma}=\Gamma\cup\gamma$ is a graph. If $z\in\Gamma$ is a non-cut point for $\Gamma$, then $z$ is also a non-cut point for $\wt{\Gamma}$.
\end{facts}

Finally, let us define $G_3:=G_2^{m+1}$. Then $G_3$ satisfies the inductive properties (i)-(v).% for $k=m+1$.

\tb{Step 4.} Connect the Julia-type points in $P\setminus (P_1\cup P_4)$.

First, each point $z\in P_2$ is already contained in $G_3$. Indeed, we assume $z\in\partial U$ and $U$ is periodic. By the choices of $\mc{U}_2$ in Step 0, the hull in $K_U$ constructed in Step 3 joins at least two components of $K_{U}\setminus\{z\}$, and thus it passes through the point $z$.

Now we connect $P_3$ to the graph $G_3$ by induction. Suppose $P_3:=\{z_1,\dotsc,z_l\}$ and $G_3^{0}:=G_3$. For $1\leq k\leq l$, we may assume that $z_k\in \partial \Omega$ and $\tu{Acc}(z_k)=\{R_k\}$ with $\Omega$ a component of $\tu{int}(K_U)$ and $U$ a periodic Fatou domain. If $z_k\in G_3^{k-1}$, we are done by letting $G_3^k=G_3^{k-1}$. Otherwise, since $G_3^{k-1}$ is connectable, there is a properly clean arc $\gamma_k$ for $\Omega$ joining $z_k$ to $c(U_{z_k})$, where $U_{z_k}$ is the marked domain associated to the point $z_k$. Note that $$\wt{\gamma}_k:=R_k\cup  \gamma_k$$ is a clean arc with its two Fatou-type endpoints in $G_3^{k-1}$. Let $G_3^k:=G_3^{k-1}\cup\wt{\gamma}_k$. Then by Lemma \ref{lem:strong-regulated}\,(1), Lemma \ref{lem:graph-connectable} and Fact \ref{facts}, $G_3^k$ is a connectable graph with no more additional cut points and endpoints.
Having finished the induction, we let $G_4$ be the graph $G_3^l$.

Finally, we define the graph $\wt{G}:=G_4\cup \bigcup \{\ol{R}: R\in \tu{Acc}(z), z\in P_2\cup P_5\cup P_6\}$, which contains the set $P$, with the vertex set $\wt{\mc{V}}$ chosen as
$\wt{\mc{V}}:=P\cup \{v: \tu{valence}(v, \wt{G})\geq 3\}.$ Then $P_1=\tu{end}(\wt{G})\cap \mc{J}_f$. Note that every Julia-type vertex $v\in \wt{\mc{V}}\setminus P$ is a branch point created in Step 3. Thus $\#\tu{Acc}(v)\geq 3$ by Lemma \ref{lem:joint}. We add the finitely many internal rays that land at the periodic vertices of $\wt{\mc{V}}\setminus P$ into $\wt{G}$. The resulting graph is denoted by $G$. By Lemma \ref{lema:facts}\,(6), this operation does not create new vertices for $G$. Thus conditions (C.4) and (C.5) hold for $G=(\mc{V}, \mc{E})$ with $\mc{V}:=\wt{\mc{V}}$.  

We now verify that $G=(\mc{V}, \mc{E})$ is admissible for $(f, P)$. Obviously, $G$ is regulated. By Fact \ref{facts} condition (C.2) holds for $G$. Since the graph $G_3$ satisfies the joint condition at every accessible point by Lemma \ref{lem:joint}, this property also holds for the graph $G$. Note that every buried point in $\mc{V}$ is an endpoint of $G$. Therefore, condition (C.3) holds. We conclude that the graph $G=(\mc{V}, \mc{E})$ is as required. The proof of Theorem \ref{thm:admissible} is complete.
\end{proof}

\section{Tiles and multi-accessible points}\label{sec_tiles}

In this section, we shall define $n$-tiles with respect to graphs which satisfy Assumption \ref{ass_1} and study some dynamical properties of the tiles and graphs. As an application, we prove Theorem \ref{thm:multi}. Throughout this section, we maintain the following assumption.

\begin{assumption}\label{ass_1}
	Let $f$ be a PCF and $G$ be a graph containing $\tu{Post}(f)$ such that it satisfies conditions (C.1) and (C.2) in Definition \ref{def:admissible}. The existence of such a graph is according to Theorem \ref{thm:admissible}.
\end{assumption}

\subsection{$n$-faces and $n$-sectors}\label{sec:tiles_const}
	Since $\tu{Post}(f)\subset G$, each iterated preimage $f^{-n}(G), n\geq 0$ is connected, or equivalently, each $X_n$ in $\tu{Comp}(\wh{\mb{C}}\setminus f^{-n}(G))$ is simply connected. We call $X_n$ an \emph{$n$-face} (with respect to $(f,G)$).

For every periodic Fatou domain $U$,
we choose a neighborhood $B_z=\phi_U^{-1}(\mb{D}_{1/2})$ of the point $z:=c(U)$, where $\phi_U$ is the B\"{o}ttcher coordinate of $U$. Clearly $f(B_z)\subseteq B_{f(z)}$.	
	Each $0$-face $X_0$ has two disjoint open subsets: $Y_0$ and $\mc{Z}_0$, defined by
	$$Y_0:=X_0\setminus\bigcup\ol{B_z}, \ \ \mc{Z}_0:=X_0\cap \bigcup B_z,$$
where the subscript $z$ goes through all periodic Fatou centers. The connected set $Y_0$ is called a \emph{$0$-subface} (of $X_0$), and
	 each component of ${\mc{Z}_0}$ is called a \emph{$0$-sector} (of $X_0$). The closure of a $0$-sector possibly equals the whole $\ol{B_z}$. Because $G$ satisfies (C.2), this happens if and only if $z\in\tu{end}(G)$.
	
	 We emphasize that sectors of $X_0$ belong to distinct Fatou domains.
	This is because periodic Fatou centers are non-cut points of $G$; see Figure \ref{fig:cutpoints} for a counter-example. The reason we require condition (C.2) in Assumption \ref{ass_1} is just to obtain this property .
	
	\begin{figure}
		\centering
		\begin{tikzpicture}
		\node at (0,0){\includegraphics[width=0.7\linewidth]{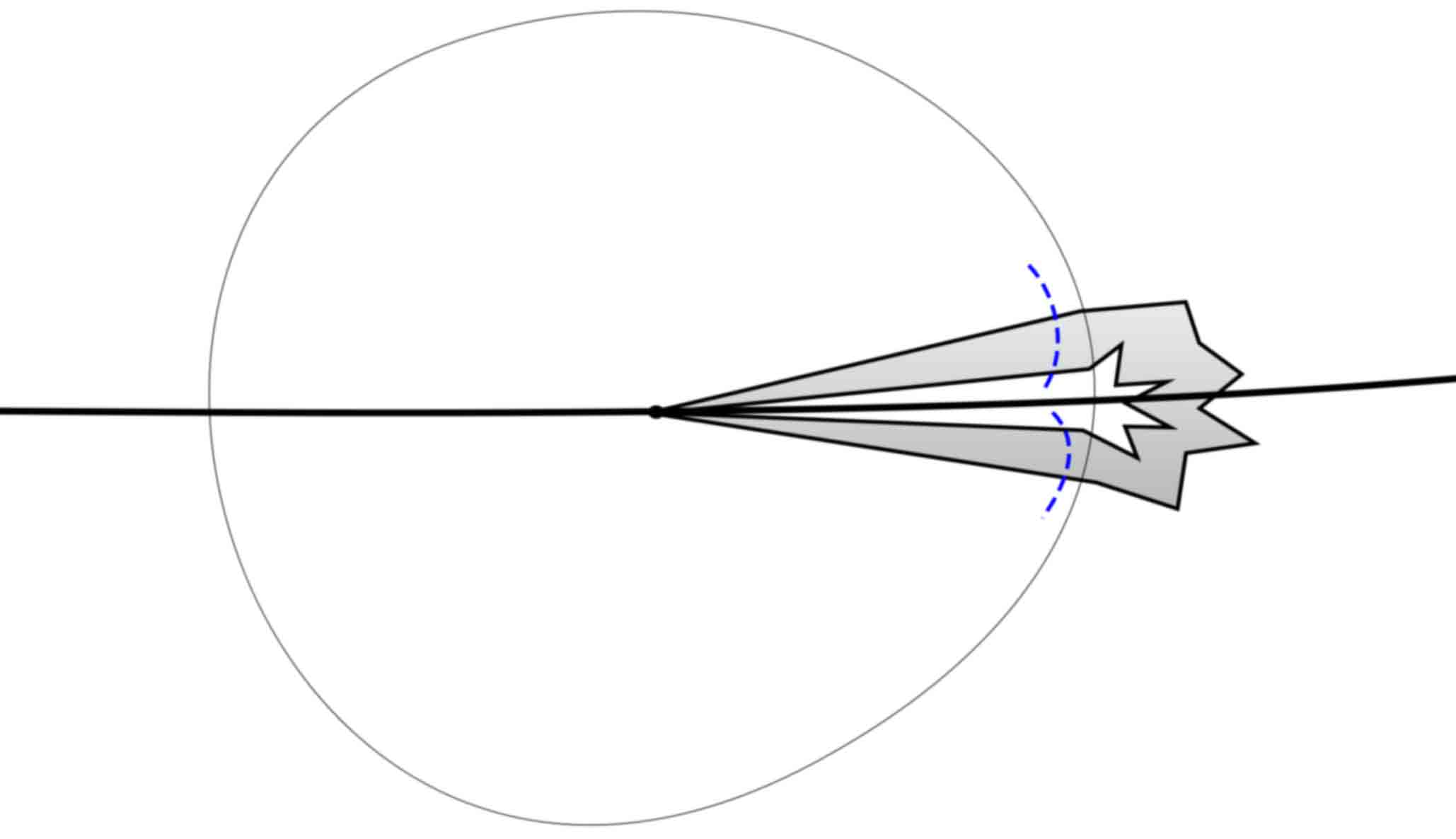}};
		\node at (-0.5,-0.25) {$z$};
		%\ruler{3}{1};
		\end{tikzpicture}
		\caption{Counterexample: $z$ is a cut point, the shaded area is an $n$-face.}
		\label{fig:cutpoints}
	\end{figure}

	For $n\geq 1$, note that $f^n$ maps an $n$-face $X_n$ onto a $0$-face $X_0$ homeomorphically. This induces bijections
	$$f^n:(X_n,Y_n,\mc{Z}_n)\to (X_0,Y_0,\mc{Z}_0),$$
	where $Y_n$ and $\mc{Z}_n$ are the pullbacks of $Y_0$ and $\mc{Z}_0$ by $f^n|_{X_n}$, respectively. We call $Y_n$ an \emph{$n$-subface} of $X_n$ and a component of ${\mc{Z}}_n$ an \emph{$n$-sector} of the face $X_n$.

 For $\theta_1\neq \theta_2\in\R/\Z$, let $]\theta_1,\theta_2[$ denote the open segment of $\R/\Z$ consisting all angles $\theta$ such that $\theta_1,\theta$ and $\theta_2$ are in the counterclockwise cyclic order. The length of $]\theta_1,\theta_2[$ is denoted by $|\theta_2-\theta_1|$.

	By definition,  each $n$-sector $Z_n$ is compactly contained in a Fatou domain $U_n$. The boundary $\partial Z_n$ is the union of two segments of internal rays, say $R_{U_n}(\theta_n^-)$ and $R_{U_n}(\theta_n^+)$, satisfying that $R_{U_n}(\theta)$ intersects $Z_n$ for $\theta\in]\theta_n^-,\theta_n^+[$, together with a segment of an equipotential curve $\phi^{-1}_{U_n}(\partial\mb{D}_{r_n})$. Thus $Z_n$ can be characterized by $(U_n,\theta_n^-,\theta_n^+,r_n).$ If $f$ maps an $(n+1)$-sector $Z_{n+1}$ onto $Z_n$, we have
	\begin{equation}\label{sending}
	U_{n}=f(U_{n+1}),~~ |\theta_n^+-\theta_n^-|=\delta_{U_{n+1}}\cdot|\theta^+_{n+1}-\theta_{n+1}^-| \tu{ ~and~  }r_n=(r_{n+1})^{\delta_{U_{n+1}}},	
	\end{equation}
where $\delta_{U_{n+1}}:=\tu{deg}(f|_{U_{n+1}})$. The \emph{width} of a sector $Z_n$ is defined as
		$$\tu{width}(Z_n):=\sup_{ 0<r\leq r_n}\tu{diam}\,\phi_{U_n}^{-1}\{re^{2\pi \tb{i}\theta}:\theta\in]\theta^-_n,\theta_n^+[\,\}.$$

	\begin{lemma}\label{lemma:converge_faces}
Following the notation above, we have that
\begin{enumerate}
\item the maximum diameters among all the $n$-subfaces converge to zero as $n\to\infty$;
\item the maximum widths among all the $n$-sectors converge to zero as $n\to\infty$;
\item if the closures of $n_k$-sectors $Z_{n_k}$ converge in the sense of Hausdorff topology, then the limit is either a point in $\mc{J}_f$ or a closed internal ray.
\end{enumerate}
	\end{lemma}

\begin{proof}
	(1) Any $0$-subface $Y_0$ intersects $\tu{Post}(f)$ only possibly in its boundary. We may assume $$\partial Y_0\cap \tu{Post}(f)=\{x_1,\dotsc,x_s\}.$$
	Choose disjoint open disks $B_k$, $1\leq k\leq s$, centered at $x_k$, such that $B_k\cap\tu{Post}(f)=\{x_k\}$. Then each $B_k^{-n}\in \tu{Comp}(f^{-n}(B_k))$ is a topological disk. Furthermore, the degree of the covering $f^n:B_k^{-n}\to B_k$ is bounded above by the constant $\prod_{c\in \tu{Crit}(f)}\tu{deg}_f(c)$, because the itinerary
	$$B_k^{-n}, f(B_k^{-n}),\dotsc,f^{n-1}(B_k^{-n})$$
	meets each critical point at most one time. By Lemma \ref{lemma:shrinking_lemma}, it holds that
	$$\tu{max}\left\{\tu{diam}\,B_k^{-n}, B_k^{-n}\in\tu{Comp}(f^{-n}(B_k))\right\}\to 0\tu{ as }n\to\infty.$$
	Consider the compact set $W:=\ol{Y_0\setminus\bigcup_{1\leq k\leq s} B_k}$. It is disjoint from $\tu{Post}(f)$. By Lemma \ref{lemma:shrinking_lemma} again, the diameters of the components of $f^{-n}(W)$ tend to $0$, as $n\to \infty$.
	Since $Y_n$ is mapped onto $Y_0$ homeomorphically, $Y_n$ is covered by just one component of $f^{-n}(W)$ and one component of  $f^{-n}(B_k)$ for each $1\leq k\leq s$.
	Combining  the results above, we have
	$$\tu{diam}\,Y_n\leq \tu{max}\left\{\tu{diam}\,W^{-n}\right\}+\sum_{1\leq k\leq s}\tu{max}\left\{\tu{diam}\,B_k^{-n}\right\}\to 0\tu{ as }n\to\infty,$$
	where $W^{-n}$ and $B_k^{-n}$ go over all components of $f^{-n}(W)$ and $f^{-n}(B_k)$, respectively. Thus $(1)$ is proved.
	
(2) Let $\epsilon$ be a real number with $0<\epsilon<1$. Let $U_1,\ldots, U_{n_\epsilon}$ be the finitely many Fatou domains with diameters $\geq \epsilon$.
	Lemma \ref{lem:new}\,(2) implies that by adding finitely many Fatou domains if necessary, one may assume $$f(U_1\cup\cdots\cup U_{n_\epsilon})\subseteq U_1\cup\cdots\cup U_{n_\epsilon}.$$	
	Since $\phi_{U_k}^{-1}$ are uniformly continuous on $\ol{\mathbb{D}}$, there is a positive number $\eta_\epsilon$ such that the width of a sector $Z=(U_k,\theta^-,\theta^+,r)$ with
	$|\theta^+-\theta^-|<\eta_\epsilon$  and $1\le k\le n_\epsilon$
	is less than $\epsilon$.

Suppose that we are given an $n$-sector $Z_n=(U,\theta^-,\theta^+,r)$. If $U\notin\{U_1,\dotsc,U_{n_\epsilon}\}$, we have
	$\tu{width}(Z_n)\leq \tu{diam}\,U< \epsilon.$
	Otherwise, the itinerary of $U$ under $f$ is always contained in $\{U_1,\dotsc,U_{n_\epsilon}\}$. Since every Fatou domain cycle is superattracting, there is an integer $N$ such that,
	\begin{equation}\label{eq:enlarge}
	\tu{deg}(f^N|_{U'})\geq \tu{max}\left\{\frac{1}{\eta_\epsilon},\, \frac{\tu{ln}\,1/2}{\tu{ln}\,(1-\epsilon)}\right\}\, \, \forall \,\,U'\in\{U_1,\dotsc,U_{n_\epsilon}\}.
	\end{equation}
	For $n\geq N$,
by applying formulas \eqref{sending} and \eqref{eq:enlarge} to $f^n$, we have
	\begin{equation}\label{angles}
     |\theta^+-\theta^-|\leq \frac{1}{\tu{deg}(f^n|_U)}\leq \eta_\epsilon\tu{\ \ and \ \ }	r=\left(\frac{1}{2}\right)^{1/\tu{deg}(f^n|_{U})}\geq 1-\epsilon.
	\end{equation}
	From the choice of $\eta_\epsilon$, we have $\tu{width}(Z_n)<\epsilon$ when $n\geq N$.
	So statement $(2)$ follows.
	
	(3) We may assume $Z_{n_k}=(U_{n_k},\theta_{n_k}^-,\theta_{n_k}^+,r_{n_k})$. If $\tu{diam}\,U_{n_k}\to 0$ as $k\to\infty$, clearly the limit $Z$ is a point in $\mc{J}_f$. Otherwise, by discarding finitely many $Z_{n_k}$, one may assume that $U_{n_k}=U$ for all $k$. From inequality \eqref{angles}, it holds that
	$$|\theta_{n_k}^+-\theta_{n_k}^-|\to 0\tu{ and }r_{n_k}\to 1\tu{ as }k\to\infty.$$
	Since $Z_{n_k}$ is convergent in the sense of Hausdorff topology, we have $\lim\limits_{k\to\infty}\theta_{n_k}^-=\lim\limits_{k\to\infty}\theta_{n_k}^+=:\theta$. Therefore $\ol{R_U(\theta)}\subseteq Z$.
	
On the other hand, by condition, for any $z\in Z$, there exists a sequence $z_{n_k}\in Z_{n_k}$ tending to $z$ as $k\to\infty$. Each $z_{n_k}$ lies in an internal ray, say $R_{U}(\eta_{n_k})$. Then  $\lim\limits_{k\to\infty}\eta_{n_k}=\theta$ as $\eta_{n_k}\in[\theta_{n_k}^-,\theta_{n_k}^+]$. It follows that $z\in \lim\limits_{k\to\infty}\ol{R_U(\eta_{n_k})}=\ol{R_U(\theta)}$. Therefore, $Z\subseteq\ol{R_U(\theta)}$. Statement (3) is proved.
\end{proof}

\subsection{From $n$-faces to $n$-tiles}
We still assume $(f,G)$ satisfies Assumption \ref{ass_1}.
	Suppose that we are given a positive number $\epsilon_0$. Then, according to Lemma \ref{lemma:converge_faces}\,(1), there exists an integer $n_0$ such that the diameter of any $n$-subface $Y_n$ with level $n\geq n_0$ is less than $\epsilon_0$. Thus, if $\epsilon_0$ is small enough, there is a unique component $Q$ of $\wh{\mb{C}}\setminus\ol{Y_n}$, such that $\tu{diam}\,Q\geq \epsilon_0.$
Since $Y_n$ is simply connected and has locally connected boundary, then $Q$ is a Jordan domain; see \cite[Proposition 2.3]{DH2}.
	We call the closed Jordan domain $\wh{Y}_n:=\wh{\mb{C}}\setminus Q$ an \emph{$n$-subtile} for $(f,G)$. Let $X_n$ be the $n$-face associated to $Y_n$, and $Z_1,\dotsc,Z_s$ be  the $n$-sectors of $X_n$ lying outside of $\wh{Y}_n$. We then set
	$$\wh{\mc{Z}}_n:=\ol{Z_1}\cup\cdots\cup \ol{Z_s}\tu{\ \ and\ \ }\wh{X}_n:=\wh{Y}_n\cup\wh{\mc{Z}}_n.$$
	Since $Z_1,\dotsc,Z_s$ come from distinct Fatou domains, $\wh{X}_n$ is still a closed Jordan domain,  called an \emph{$n$-tile} for $(f,G)$. The closed sets $\ol{Z_1},\ldots,\ol{Z_s}$ are called the \emph{$n$-sectors} of $\wh{X}_n$. Sometimes we may also call the triple $(\wh{X}_n,\wh{Y}_n,\wh{\mc{Z}}_n)$  an \emph{$n$-tile}. By definition, the $n$-faces, $n$-subfaces and $n$-sectors of faces are all open sets, while the  $n$-tiles, $n$-subtiles and $n$-sectors of tiles are all closed.

From this definition, it is clear that the interiors of any two $n$-tiles (resp. $n$-subtiles) are either disjoint or nested.
We are mainly interested in the \emph{maximal} $n$-tiles (resp. $n$-subtiles), in the sense that: a maximal $n$-tile (resp. $n$-subtile) is not contained in any $n$-tile (resp. $n$-subtile) except itself. The collection of maximal $n$-tiles is denoted by $\wh{\tb{X}}_n$.

In order to describe the shapes of tiles, we introduce the following definition
$$\rho(\wh{X}_n):=\tu{max}\left\{\tu{diam}\,\wh{Y}_n,\tu{width}(\wh{Z}_n), \wh{Z}_n\in\tu{Comp}(\wh{\mc{Z}}_n)\right\}.$$
The following result is a direct consequence of the construction of $n$-tiles and Lemma \ref{lemma:converge_faces}. We omit the proof.

\begin{lemma}\label{lemma:n_tiles}
Let $n\geq n_0$ such that $n$-tiles can be defined. Then the following statements hold.
\begin{enumerate}
\item $\partial \wh{X}_n\subseteq f^{-n}(G)$.
\item $\wh{\tb{X}}_n$ covers the whole sphere.
\item $\tu{max}\left\{\rho(\wh{X}_n),\wh{X}_n\in \wh{\tb{X}}_n\right\}\to0$ as $n\to\infty$.
\item Suppose that we are given a ray $R:=R_U(\theta)$ and an arbitrarily small interval $]\theta_\epsilon^-,\theta_\epsilon^+[$ including $\theta$. Then there exists an integer $n_1$ such that, for any $n\geq n_1$ and any $n$-tile $(\wh{X}_n,\wh{Y}_n,\wh{\mc{Z}}_n)\in\wh{\tb{X}}_n$, there is at most one $n$-sector in $\wh{\mc{Z}}_n$  contained in $U$, and the relationship between $R$ and $\wh{X}_n$ is just one of the following:
	\begin{itemize}
	\item[(4.1)] $\wh{X}_n\cap R=\es$, thus $\wh{X}_n\cap U=\es$;
	
	\item[(4.2)]  $\wh{X}_n\cap R=\{c(U)\}$;
	
	\item[(4.3)]  $R\subset\partial \wh{X}_n$. Moreover, there exists another $n$-tile $\wh{X}_n'$, such that $$\wh{X}_n\cap U=\bigcup_{\eta\in [\theta^-, \theta]}R_U(\eta)\tu{ and }\wh{X}_n'\cap U=\bigcup_{\eta\in[\theta, \theta^+]}R_U(\eta)$$ with $\theta\in\,]\theta^-,\theta^+[\,\Subset\,]\theta_\epsilon^-,\theta_\epsilon^+[$ by exchanging $\wh{X}_n$ and $\wh{X}_n'$ if necessary. %Here $R_U(\theta^-)$ and $R_U(\theta^+)$ land at distinct points;
	
	\item[(4.4)] $R\setminus\{c(U)\}\subseteq\tu{int}(\wh{X}_n)$. Moreover, $\wh{X}_n\cap U\subseteq \bigcup_{\eta\in[\theta_\epsilon^-, \theta_\epsilon^+]}R_U(\eta)$.
	\end{itemize}
\end{enumerate}
\end{lemma}

\subsection{Dynamical properties of graphs}

\begin{lemma}\label{lem:property_G_0}
	Let $(f,G)$ satisfy Assumption \ref{ass_1}. Let $z$ be a point in $\mc{J}_f$ such that $\#\tu{Acc}(z,U)\geq 2$ for a Fatou domain $U$. Let $z_n:=f^n(z)$.
	%and $U_n:=f^n(U)$. Let $(f,G)$ satisfy Assumption \ref{ass_1}.
	Then for each sufficiently large integer $n$, we have $z_n\in G$. Moreover, if the orbit of $z$, i.e., $\{z_0, z_1, \ldots\}$, avoids $\tu{Crit}(f)$, then
	$\#\tu{Acc}(z,U)\leq \tu{valence}(z_n,G).$
\end{lemma}

\begin{proof}
Since $\tu{Post}(f)\subseteq G$, we just consider the case that the orbit of a point $z$ avoids $\tu{Crit}(f)$.
%Let  $U\in\tu{Dom}(z)$ such that $\#\tu{Acc}(z,U)\geq 2$.	
%By Lemma \ref{lema:facts}, we have $$\#\tu{Comp}(K_U\sm\{z\})=\#\tu{Acc}(z,U).$$
Let $s$ be an integer such that $2\leq s\leq\#\tu{Acc}(z, U)$.
We choose $s$ rays $R_i:=R_U(\theta_i)$ in $\tu{Acc}(z,U)$. They divide $\olC\setminus\{z\}$ into $s$ mutually disjoint Jordan domains $D_i$. We list these rays and domains in the cyclic order:
	$$R_0(=R_{s}), D_0(=D_{s}), R_1, D_1,\ldots, R_{s-1}, D_{s-1}.$$
Suppose we are given a ray $R_i$. By Lemma \ref{lemma:n_tiles}\,(4), for each sufficiently large $n$, any $n$-tile $\wh{X}_n$ that contains $R_i$  is disjoint from $R_j\sm\{c(U)\}$ for $j\neq i$. Let $\wh{X}_n$ be such a tile. We assume $\wh{X}_n\cap U=\bigcup_{\eta_1\leq \eta\leq\eta_2} R_U(\eta)$.
	%z_{i-1}:=\phi^{-1}_U(e^{2\pi \tb{i}\eta_1})$ and  $z_i:=\phi^{-1}_U(e^{2\pi\tb{i}\eta_2})$,
	One of the following three cases happens:
	
	Case 1: $\theta_i=\eta_1$. Then $w:=\phi^{-1}_U(e^{2\pi\tb{i}\eta_2})\in D_i$. The set $\gamma_i:=\partial\wh{X}_n\sm U$ is an arc joining $w$ to the point $z$ with $\gamma_i\setminus\{z\}\subset D_i$.

	Case 2: $\theta_i=\eta_2.$ Then $\phi^{-1}_U(e^{2\pi \tb{i}\eta_1})\in D_{i-1}$. By Lemma \ref{lemma:n_tiles}\,(4.3), there exists another $n$-tile $\wh{X}_n'$ whose boundary contains $R_i$. We assume $\wh{X}'_n\cap U=\bigcup_{\eta\in[\eta_2, \eta_1']} R_U(\eta)$. Then similar to Case 1, we obtain an arc $\gamma_i:=\partial\wh{X}'_n\setminus U$ joining $w:=\phi_U^{-1}(e^{2\pi i\eta_1'})$ to the point $z$ with $\gamma_i\setminus\{z\}\subset D_i$.
	
	Case 3: $\eta_1<\theta_i<\eta_2$. Let $w':=\phi^{-1}_U(e^{2\pi\tb{i}\eta_1})\in D_{i-1}$ and $w:=\phi^{-1}_U(e^{2\pi\tb{i}\eta_2})\in D_{i}$. Then the arc $\alpha:=\partial\wh{X}_n\sm U$ joins $w$ and $w'$. So $z\in \alpha$ because $s\geq 2$ implies $D_{i-1}\neq D_{i}$. Since in this case $\alpha$ intersects $R_1, \ldots, R_s$ in just one point $z$, the arc $\gamma_i:=[w,z]_\alpha$ satisfies that $\gamma_i\setminus\{z\}\subset D_i$.
	
As a consequence, for each sufficiently large integer $n$, we obtain arcs $\gamma_1,\dotsc,\gamma_s$ with $\gamma_i\setminus\{z\}\subseteq D_i$ and $\gamma_i\subset f^{-n}(G)$. Locally consider the orientation-preserving homeomorphism $f^n|_{B_n}$ on a neighborhood $B_n$ of the point $z$. By choosing subarcs of $\gamma_i$ incident to $z$ if necessary, one may assume all arcs $\gamma_i$ belong to $B_n$. Then the images $f^n(\gamma_i)\setminus\{z_n\}, 1\leq i\leq s$, are $s$ disjoint semi-open arcs in the graph $G$ incident to the point $z_n$. Thus $z_n\in G$ and $s\leq \tu{valence}(z_n, G)$. The proof is complete.
\end{proof}

\begin{corollary}\label{cor:admissible}
	Let $(f,G)$ satisfy Assumption \ref{ass_1}. Let $z$ be a point in $\mc{J}_f$ such that $\#\tu{Acc}(z,U)\geq 2$ for a Fatou domain $U$ and the orbit of $z$ avoids $\tu{Crit}(f)$. Let $z_n:=f^n(z)$, $U_n:=f^n(U)$ and $S_n:=f^n( \bigcup\,\{\ol{R}: R\in\tu{Acc}(z, U)\} )$. Then for each sufficiently large integer $n$, we have $z_n\in G$ and every component $W$ of $\olC\setminus S_n$ satisfies that $W\cap G$ contains either a vertex of $G$ or the open arc $\tu{int}(R_*)$ for an internal ray $R_*$ of $U_n$.% $\tu{int}(R_*)\subseteq G$ for an internal ray $R_*$ of $U_n$.% This means $z$ is eventually periodic.
\end{corollary}
\begin{proof}
Let $s:=\#\tu{Acc}(z, U)$. From the proof of Lemma \ref{lem:property_G_0}, we see that for each sufficiently large integer $n$, the image $f^{n}(B_{n})$ is a neighborhood of $z_n$ and every component of $f^{n}(B_{n})\setminus S_n$ contains a semi-open arc from $f^n(\gamma_i)\sm\{z_n\}, 1\leq i\leq s$. Thus all components of $\olC\sm S_n$ meet $G$.
	
	If one of these components, say $W$, contains no vertex of $G$, then a component $C$ of $W\cap G$, whose closure contains $z_n$, lies in an edge $e$ of $G$. Thus $\ol{C}$ is an arc joining $z_n$ to another point $w$. Clearly $w\in\partial W$, for otherwise, $w\in\tu{end}(G)\subseteq\mc{V}$. Hence $w=c(U_n)$ and $C=]z_n, c(U_n)[\,_e$. Since $e\cap U_n$ consists of internal rays, then $C\cap U_n$ is the interior of an internal ray of $U_n$, which is as required. The proof is complete.
\end{proof}	

\subsection{Proof of Theorem \ref{thm:multi}}
\begin{proof}[Proof of Theorem \ref{thm:multi}]
	By Theorem \ref{thm:admissible} there is a graph $G=(\mc{V},\mc{E})$ satisfying Assumption \ref{ass_1}. We now consider a multi-accessible point $z$, i.e., $\#\tu{Acc}(z)\geq 3$, such that its orbit $z_n:=f^n(z), n\geq 0$, is disjoint from $\tu{Crit}(f)$.	By Lemma \ref{lema:facts}\,(1), for each sufficiently large integer $n$ the number $\#\tu{Dom}(z_n)$ is a constant independent of $n$, say $m$. There are three cases:
	
	Case $1$: $m\geq 3$. By Lemma \ref{lem:new}\,(2), the point $z$ eventually falls into the set
	$$S:=\{w\in \mc{J}_f: \tu{Dom}(w)\tu{ contains at least three periodic Fatou domains}.\}$$
	By Lemma \ref{lemma:intersection_three_component}, $S$ is a finite set. Thus the point $z$ is eventually periodic in this case.
	
	Case $2$: $m=1$. Then at least three rays in $\tu{Acc}(z)$ are eventually iterated into the unique Fatou domain of $\tu{Dom}(z_{n_0})$ for some $n_0$. By Lemma \ref{lem:property_G_0}, $z_{n_0+n}$ are branched points of $G$ for all large integers $n$. Since $G$ has at most finitely many branch points, the point $z$ eventually falls into a cycle of $\mc{V}$.

	Case $3$: $m=2$. If one of $\tu{Dom}(z_{n_0})$ contains more than two rays of $\tu{Acc}(z_{n_0})$ for an $n_0$, then by the argument in Case 2, the point $z$ is eventually iterated into a cycle of $\mc{V}$.	
	Thus we are reduced to the case that for some $n_0$ the only two domains $U$ and $U'$ in $\tu{Dom}({z_{n_0}})$ are periodic with a common period, say $p$, and each of them contains at most two rays from $\tu{Acc}(z_{n_0+kp})$ for all $k\geq 0$.
% $\#\tu{Acc}(z_{n+kp}, U_1)=2$ and $1\leq \#\tu{Acc}(z_{n+kp}, U_2)\leq 2$
%	for all $k\geq 0$ and a large $n$, where $p$ is a common period of $U_1$ and $U_2$.
 Since the points $z_{n_0+kp}$ are multi-accessible, we may assume that $$\tu{Acc}(z_{n_0+kp},U)=\{R_k,R_k'\}\tu{ and }U'\subseteq \Omega\in\tu{Comp}(\tu{int}(K_{U})).$$
	
	Note that $z_{n_0+kp}\in\partial U'\cap\partial \Omega\subseteq \partial U\cap\partial U'$. Let $W_k$ be the component of $\olC\sm\ol{R_k\cup R_k'}$ that is disjoint from $U'$. Obviously $W_k\subseteq \olC\setminus\ol{\Omega}$. Then either $$W_{k_1}\cap W_{k_2}=\es\tu{ or }W_{k_1}=W_{k_2}$$
	for $ k_1\neq k_2$ and the latter happens if and only if $z_{n_0+k_1p}=z_{n_0+k_2p}$. Since $\#\mc{V}<\infty$ and $\tu{valence}(c(U), G)<\infty$, by Corollary \ref{cor:admissible} there are at most finitely many disjoint domains $W_k$. It follows that the point $z$ is eventually periodic.

	The above arguments show that multi-accessible points are eventually periodic and they are iterated into one of the four sets: $\tu{Post}(f)$, $S$, $\mc{V}$ and $\wt{S}$. Here $\wt{S}$ is the union of all periodic points $w\in\mc{J}_f$ satisfying that each point $x$ in the cycle $\{w, f(w), f^2(w), \ldots\}$ has $\#\tu{Dom}(x)=2$ with $\#\tu{Acc}(x, U_x)=2$ and $\#\tu{Acc}(x, U'_x)\leq 2$. Obviously, the domains $U_x$ and $U_x'$ are periodic.

	To complete the proof of Theorem \ref{thm:multi}, we just need to check the finiteness of $\wt{S}$. If it is not finite, there exist infinitely many points $w_k\in \wt{S}$ and two Fatou domains $U$ and $U'$ such that $w_k\in \partial U\cap\partial U'$ and $\#\tu{Acc}(w_k, U)=2$. These $w_k$ induce mutually disjoint Jordan domains $W_k$ as stated in Case 3. Since $w_k$ is periodic, by Corollary \ref{cor:admissible} each $W_k$ contains either a vertex of $G$ or an internal ray in $U\cap W\cap G$. Then
	either $\#\mc{V}=\infty$ or $\tu{valence}(c(U), G)=\infty$. We arrive at a contradiction. %The proof of Theorem \ref{thm:multi} is complete.
\end{proof}

\begin{remark}\tu{
There are some Julia sets that contain infinitely many periodic points with $\#\tu{Acc}(z)=2$; for examples: the Julia set of $f(z)=z^2$ and the airplane in Figure \ref{fig:airplane}.
The results in Theorem \ref{thm:multi}  are known for postcritically finite polynomials; see \cite{Ki,Poi}. Here we generalize them to the rational situation.}
\end{remark}

\section{Approximating edges with Julia-type endpoints}\label{sec_homotopy}
As discussed in the introduction, the first step towards the construction of invariant graphs is to find a graph $G_0$ such that each edge of $G_0$ can be arbitrarily approximated by an arc in $f^{-n}(G_0)$ for large $n$. This section is devoted to proving this point for some edges of admissible graphs; see the following proposition.

\begin{proposition}\label{prop:epsilon_neighbor}
	Let $f$ be a PCF. Let $P$ be a finite and $f$-invariant set containing the set $P_0$ in Theorem \ref{thm:invariant}. Let $G=(\mc{V},\mc{E})$ be an admissible graph for $(f, P)$. Let $\gamma$ be a regulated arc such that it joins two Julia-type endpoints $z_1$ and $z_2$ and satisfies the joint condition. Assume $z_1$ and $z_2$ each eventually fall into a cycle in $G$. Then, for any $\epsilon>0$ there exists an arc $\wt{\gamma}\subseteq f^{-n}(G)$ joining $z_1$ and $z_2$ such that $\wt{\gamma}\subset \mc{N}(\gamma,\epsilon)$ for each sufficiently large integer $n>0$.
\end{proposition}

\begin{figure}[h]
	\centering
	\begin{tikzpicture}
	\node at (0,0){\includegraphics[width=0.85\linewidth]{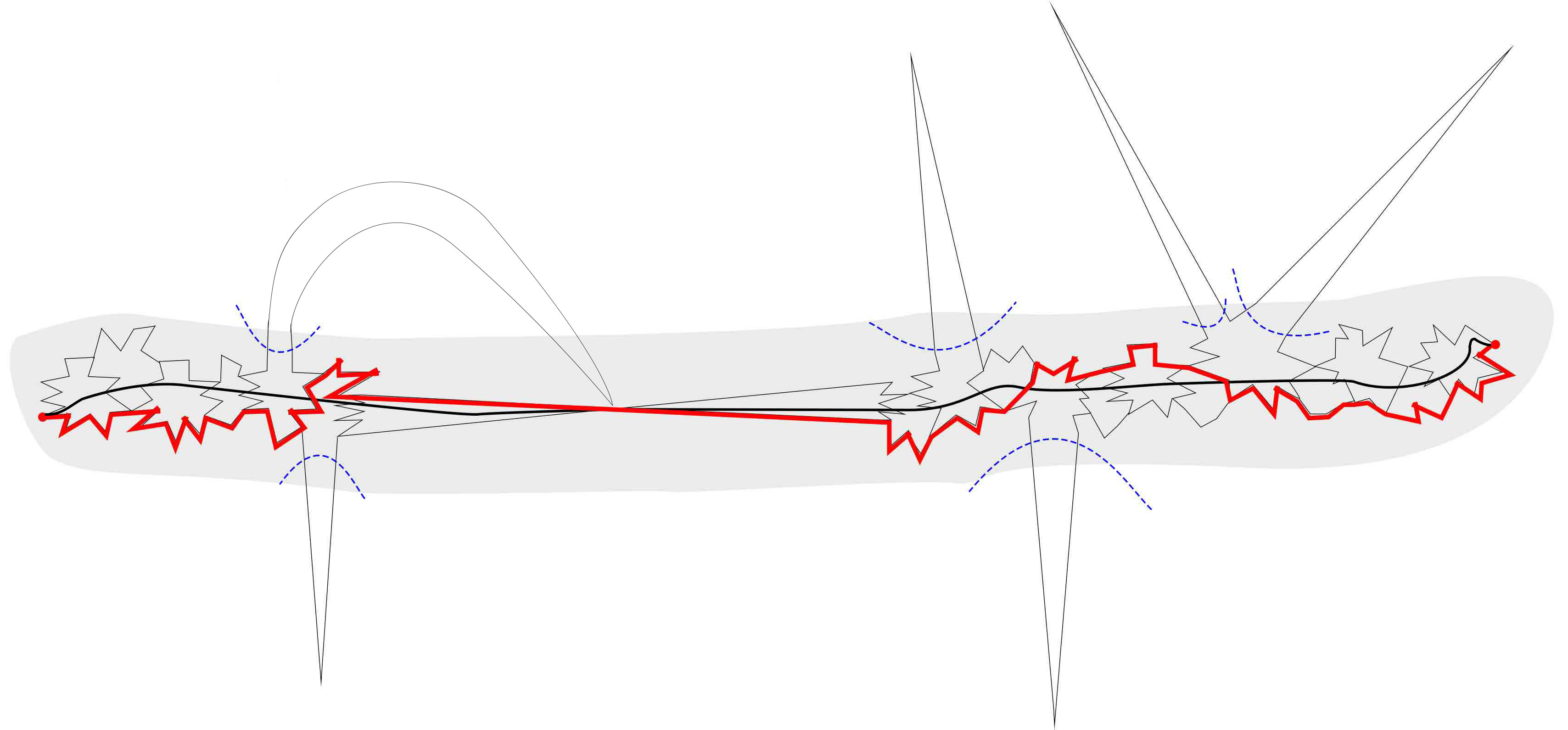}};
	\node at (-3.75,1.80){$\wh{Z}$};
	\node at (3.35,1.80){$\wh{Z}'$};
	\node at (5.10,1.80){$\wh{Z}''$};
	\node at (-2.5,-0.95){$\wh{Z}'''$};
	%\ruler{7}{2};
	\end{tikzpicture}
	\caption{Illustration of the proof of Proposition \ref{prop:epsilon_neighbor}. The shaded area shows the $\epsilon$-neighborhood of the edge $\gamma$ (the black smooth arc). As shown, $\gamma$ is covered by many irregular polygon regions, namely $n$-tiles. The wicked sectors (long spines) within a tile are distributed in the same side of $\gamma$. Thus one can find a path $\wt{\gamma}$ (the red arc) along the boundary of tiles avoiding these wicked sectors. }
	\label{fig:obstruction}
\end{figure}

\begin{proof}%[Proof of Proposition \ref{prop:epsilon_neighbor}]
	Since the graph $G$ satisfies conditions (C.1) and (C.2), all material in Section \ref{sec_tiles} can be used. According to Lemma \ref{lemma:n_tiles}\,(3), there is an integer $n_0$ such that for all $n$-tiles $(\wh{X}_n,\wh{Y}_n,\wh{\mc{Z}}_n)$ with levels $n\geq n_0$ we have
	\begin{equation}\label{eq:1}
\rho(\wh{X}_n):=\tu{max}\left\{\tu{diam}\,\wh{Y}_n,\tu{width}(\wh{Z}_n), \wh{Z}_n\in\tu{Comp}(\wh{\mc{Z}}_n)\right\}<\epsilon/3.
\end{equation}
	For such an $n$, let $\wh{\tb{X}}_n(\gamma)$ denote the collection of all $n$-tiles $(\wh{X}_n,\wh{Y}_n,\wh{\mc{Z}}_n)$ with
	$\gamma\cap\wh{Y}_n\neq \emptyset.$
Since a sector of $\wh{\mc{Z}}_n$ meeting $\gamma$ in more than one point implies that $\wh{Y}_n\cap \gamma\neq \emptyset$, the collection $\wh{\tb{X}}_n(\gamma)$ forms a covering of $\gamma$.
	
	Now we try to modify all $n$-tiles $(\wh{X}_n,\wh{Y}_n,\wh{\mc{Z}}_n)$ in $\wh{\tb{X}}_n(\gamma)$ to get the so-called \emph{modified tiles} $\wt{X}_n$ such that $\wt{X}_n\subseteq\mc{N}(\gamma,\epsilon)$. It proceeds as follows.
	\begin{itemize}
		\item[(i)] If $\wh{X}_n\subseteq \mc{N}(\gamma,\epsilon)$, we set $\wt{X}_n:=\wh{X}_n$  in this case.
		\item[(ii)] Otherwise, $\tu{diam}\,\wh{X}_n\geq \epsilon$. By the structure of $n$-tiles, we have the estimate $$\tu{diam}\,\wh{Y}_n+ 2\cdot \tu{max}\,\{\tu{diam}\,\wh{Z}:\wh{Z}\in\tu{Comp}(\wh{\mc{Z}}_n)\}\geq \tu{diam}\,\wh{X}_n\geq \epsilon.$$
		 Since $\tu{diam}\,\wh{Y}_n< \epsilon/3$ by \eqref{eq:1}, there are $n$-sectors $\wh{Z}$ in $\wh{\mc{Z}}_n$ so that $$\tu{diam}\,\wh{Z}>\epsilon/3\tu{ and }\wh{Z}_n\nsubseteq\mc{N}(\gamma,\epsilon).$$ Such sectors are called \emph{wicked} sectors of $\wh{X}_n$ (e.g. $\wh{Z},\wh{Z}',\wh{Z}''$ in Figure \ref{fig:obstruction}). Note that a wicked sector $\wh{Z}$ meets $\gamma$ in at most one point. For otherwise, $\wh{Z}\cap\gamma$ contains a segment of an internal ray; since $\tu{width}(\wh{Z})<\epsilon/3$, we have $\wh{Z}\subseteq \mc{N}(\gamma,\epsilon)$, impossible. 		 
		  %and the remaining ones are called \emph{non-wicked} (e.g. $\wh{Z}'''$ in Figure \ref{fig:obstruction}). Any non-wicked sector contains a segment of $\gamma$ that lies in an internal ray. Thus, each non-wicked sector $\wh{Z}$ satisfies		$$\wh{Z}\subseteq\mc{N}(\gamma,\epsilon)$$by \eqref{eq:1}.
%If wicked $n$-sectors of $\wh{X}_n$ do not exist, we define $\wt{X}_n:=\wh{X}_n$.
	Let $\wh{Z}_1,\dotsc,\wh{Z}_s$ be the collection of wicked sectors in $\tu{Comp}(\wh{\mc{Z}}_n)$. Then the mutually disjoint closed arcs
		$$\ell_i=\wh{Y}_n\cap \wh{Z}_i, 1\leq i\leq s,$$
	do not intersect $\gamma$. In this case, the following closed Jordan disk
		$$\wt{X}_n:=(\wh{X}_n\setminus\wh{Z}_1\cup\cdots\cup \wh{Z}_s)\cup \ell_1\cup \cdots\cup \ell_s$$
is defined as a \emph{modified $n$-tile}, and the arcs $\ell_1,\ldots,\ell_s$ are called the \emph{wicked edges} of $\wt{X}_n$.
			\end{itemize}

	From the construction, we see that $\wt{\tb{X}}_n(\gamma)$, the collection of all modified $n$-tiles $\wt{X}_n$ with $\wh{X}_n\in\wh{\tb{X}}_n(\g)$, is still a covering of $\gamma$. Let $x_n$ (resp. $y_n$) be the first-in (resp. last-out) place of $\gamma$ meeting a modified $n$-tile $\wt{X}_n$. 
	
	\noindent\tb{Claim}~~\emph{There is an integer $n_1(\geq n_0)$ such that for any $n\geq n_1$ every $\wt{X}_n\in \wt{\tb{X}}_n(\gamma)$ with $x_n\neq y_n$ has at most one wicked edge.
	\begin{itemize}
		\item[(1)] If $x_n, y_n\in \tu{int}(\gamma)$, then $x_n, y_n\in\partial\wt{X}_n$. Let $L(\wt{X}_n)\subseteq\partial\wt{X}_n$ be the arc disjoint from the possible wicked edge connecting $x_n$ and $y_n$.
		\item[(2)] Otherwise, one of $x_n$ and $y_n$ is an endpoint of $\gamma$, say $x_n=z_1$, then there is an arc in $f^{-n}(G)\cap \wt{X}_n$ joining $z_1$ and $y_n$. Let $L(\wt{X}_n)$ be the arc in this case.
	\end{itemize}}

We now complete the proof of Proposition \ref{prop:epsilon_neighbor} under the claim. The required arc $\wt{\gamma}$ is obtained by applying the first-in and last-out rule to the modified tiles in $\wt{\tb{X}}_n(\gamma)$. It proceeds as follows.

We first enumerate all the modified tiles in $\wh{\tb{X}}_n(\gamma)$ as $\wt{X}_n^{1}, \ldots, \wt{X}_n^{s}$. Let $\gamma_0:=\gamma$. Inductively, for $k\geq 1$, if $\#\gamma_{k-1}\cap \wt{X}_n^{k}\leq 1$, let $\gamma_{k}:=\gamma_{k-1}$. Otherwise, let $x_n^k$ and $y_n^k$ be the distinct first-in and last-out places of $\gamma_{k-1}$ meeting $\wt{X}_n^k$. Let $\beta_n^k$ be an arc joining $x_n^k$ and $y_n^k$ with its interior in $\tu{int}(\wt{X}^k_n)$. Then in this case we define $$\gamma_{k}:=\left(\gamma_{k-1}\setminus[x_n^k, y_n^k]_{\gamma_k}\right)\cup \beta_n^k.$$
Since the interiors of modified $n$-tiles are mutually disjoint, $\gamma_k$ is an arc.
By induction, we obtain an arc $\gamma_{s}$, which is the union of at most $s$ arcs $\beta_n^k$. We then replace each subarc $\beta_n^k$ of $\gamma_s$ by the arc $L(\wt{X}_n^k)$ obtained in the claim for $\wt{X}^k_n$. The resulting set, denoted by $\Gamma$, joins $z_1$ and $z_2$.
Moreover, by the claim, we have $\Gamma\subseteq f^{-n}(G)$ and $\Gamma\subseteq \mc{N}(\gamma,\epsilon)$. Finally, one can derive an arc $\wt{\gamma}$ within $\Gamma$ connecting $z_1$ and $z_2$, which is as required.

\begin{proof}[Proof of the Claim]
If the claim is not true, there exists a sequence of $n_k$-tiles $(\wh{X}_{n_k},\wh{Y}_{n_k},\wh{\mc{Z}}_{n_k})\in \wh{\tb{X}}_{n_k}(\gamma)$ with $x_{n_k}\neq y_{n_k}$ and $n_k\to \infty$ as $k\to\infty$ such that $\wh{\mc{Z}}_{n_k}$ contains two sectors $\wh{Z}_{n_k}^{\pm}$ satisfying
		\begin{enumerate}
			\item[(i)] the diameters of $\wh{Z}^{\pm}_{n_k}$ are greater than $\epsilon/3$;
			\item[(ii)] both $\wh{Z}^{\pm}_{n_k}$ are not contained in $\mc{N}(\gamma,\epsilon)$;
		%	\item [(iii)]the wicked edges $\ell_{n_k}^{\pm}:=\wh{Z}^{\pm}_{n_k}\cap \wh{Y}_{n_k}$ lie in distinct components of $\partial \wt{X}_{n_k}\setminus\{x_{n_k},y_{n_k}\}$;
			\item[(iii)]  the compact sets $\wh{X}_{n_k},\wh{Y}_{n_k},\wh{Z}^-_{n_k}$ and $\wh{Z}^+_{n_k}$ converge to $X,Y, \ol{R^-}$ and $\ol{R^+}$ in the sense of Hausdorff topology, respectively.
		\end{enumerate}
By Lemma \ref{lemma:n_tiles}\,(3), $Y=\{z\}$ with $z$ a point in $\gamma\cap\mc{J}_f$. Since there are only finitely many Fatou domains whose diameters are greater than $\epsilon/3$, by passing to a subsequence we may assume $\wh{Z}_{n_k}^{\pm}$ always belong to the two distinct Fatou domains $U^{\pm}$, respectively. By Lemma \ref{lemma:converge_faces}\,(3), the limits $\ol{R^{\pm}}$ are closed internal rays with $R^{\pm}$ from $U^{\pm}$, respectively. Moreover, it holds that $R^-, R^+\in\tu{Acc}(z)$. Property (ii) implies $\tu{int}(R^{\pm})\cap\gamma=\es$.

If $\#\tu{Acc}(z)=2$, then $R^+$ and $R^-$ are adjacent. They bound two ordinary entrances. 
Since $z\in\gamma$ and $\tu{int}(R^{\pm})\cap\gamma=\es$, at least one component of $\gamma\setminus \{z\}$ approaches the point $z$ along an ordinary entrance. This contradicts the joint condition for $G$ at the point $z$.

If $\#\tu{Acc}(z)\geq 3$, then by Theorem \ref{thm:multi} the point $z$ is eventually periodic. Moreover, it falls into a cycle within $P_0$ from the choice of $P_0$. Condition (C.5) of admissible graphs gives that the set $S:=\bigcup \{\ol{R}: R\in\tu{Acc}(z)\}$ will be iterated into the graph $G$. Thus for each sufficiently large integer $k$, we have $S\subseteq f^{-n_k}(G)$. We assume further that $k$ is so large that
$$\tu{max}\{\tu{diam}\,\wh{Y}_{m}: (\wh{X}_{m}, \wh{Y}_m,\wh{\mc{Z}}_m)\in \wh{\bf{X}}_{m}, m\geq n_k\}<\tu{min}\{\tu{diam}\, R: R\in\tu{Acc}(z)\}.$$
Then each internal ray $R$ in $S$ is contained in the boundary of an $m$-tile $\wh{X}_m$ for any $m\geq n_k$. Indeed, the above inequality implies $R$ is not covered by any $m$-subtiles. Thus $R\subset \wh{Y}_m\cup \wh{Z}_m$ for some sector $\wh{Z}_m$ in $\wh{\mc{Z}}_m$. Since $\tu{int}(\wh{Z}_m)\cap f^{-m}(G)=\es$ and $R\subset f^{-m}(G)$, it follows that $R\subset \partial \wh{X}_m$.

 This implies $S$ cannot disconnect subtiles $\wh{Y}_{n_k}$ for all large integers $k$. We choose an open disk $B$ around $z$ such that $R\subset \ol{B}$ and $\partial B\cap R=\tu{end}(R)\setminus\{z\}$ for all $R\in \tu{Acc}(z)$. Since the sequence $\wh{Y}_{n_k}$ tends to the point $z$, by passing to a subsequence if necessary, there is a component $D$ of $B\setminus S$ such that it contains all $\wh{Y}_{n_k}$ for sufficiently large integers $k$. Then $\wh{Z}_{n_k}^\pm\subset \ol{D}$ and the limit $\ol{R^{+}}\cup\ol{R^-}$ lies in $\partial D\cap S$. We conclude that $R^+$ and $R^-$ are adjacent. 

We claim that at least a component of $\gamma\setminus\{z\}$ approaches $z$ along the ordinary entrance bounded by $R^+$ and $R^-$. Indeed, by the condition that $x_{n_k}\neq y_{n_k}$, we have $\gamma\cap(\wh{Y}_{n_k}\setminus\{z\})\neq \es$. Hence, since $\wh{Y}_{n_k}$ converges to $z$ in $D$, so does $\gamma$. The claim holds.
%let $\wh{Z}_{n_k}\neq \wh{Z}_{n_k}^\pm$ be an $n_k$-sector of $\wh{\mc{Z}}_{n_k}$. Then $\wh{Z}_{n_k}$ is disjoint from $U^{\pm}$. If $\limsup_{k\to\infty}\,\tu{diam}\, \wh{Z}_{n_k}>0$, by taking a subsequence, $\wh{Z}_{n_k}$ converges to a closed internal ray $R$ in $\tu{Acc}(z)\setminus\{R^+,R^-\}$. It follows that $R\subset \ol{D}$, contradicting that $R^+$ and $R^-$ are adjacent. Then $\wh{Z}_{n_k}\to z$ as $k\to\infty$, and hence $\wh{Z}_{n_k}$ is contained in $\ol{D}$ for sufficiently large integer $k$. As a consequence, $\{x_{n_k}, y_{n_k}\}\subset\ol{D}$ and $x_{n_k}, y_{n_k}\to z$ as $k\to \infty$. The claim is proved.

As $\gamma$ satisfies that joint condition at the point $z$, we obtain a contradiction by the claim. Statement (1) follows directly.

For statement (2), since the endpoint $z_1$ is eventually periodic in $P\subseteq \mc{V}$, the point $z_1$ is contained in the subgraph $T_n:=f^{-n}(G)\cap \wh{X}_n$ for each sufficiently large $n$. If there is no wicked sector of $\wh{X}_{n}$, then $\wh{X}_n=\wt{X}_n$. There is an arc in $T_n$ connecting $z_1$ and $y_n$. Otherwise, let $\wh{Z}_n$ be the only wicked sector of $\wh{X}_n$ in a Fatou domain $U$. By definition, $\wh{X}_n=\wt{X}_n\cup \wh{Z}_n$ and $\tu{int}(\wh{Z}_n)\cap f^{-n}(G)=\es$. Since $\partial \widehat{X}_n\subseteq T_n$ is a Jordan curve and $\wh{Z}_n\cap T_n$ is a subarc of $\partial\wh{X}_n$ in a Fatou domain, then $T_n\setminus\wh{Z}_n$ is still connected. Therefore, one can find an arc in $T_n\setminus \wh{Z}_n$ joining $z_1$ and $y_n$. The proof of the claim is complete.
\end{proof}
\end{proof}

\section{Proof of Theorem \ref{thm:invariant}}\label{sec_invariant}

This section has two parts: the first part is to show that an admissible graph can be modified in an arbitrarily small neighborhood to a new one having the ``local invariance'' property; see Proposition \ref{prop:new}; the second part is devoted to proving Theorem \ref{thm:invariant} along the outline given in the introduction.

\begin{lemma}\label{lem:one}
	Let  $U$ be a Fatou domain and $z\in\partial U$ with $\tu{Acc}(z,U)=\{R\}$. Then for any $\epsilon>0$, there is a number $\delta>0$ such that  for any $z'\in\partial U$ with $\tu{dist}(z,z')<\delta$ and any $R'\in\tu{Acc}(z',U)$, we have
	$R'\subseteq \mc{N}(R,\epsilon).$	
\end{lemma}

\begin{proof}
	Let $z_n$ be points in $\partial U$ converging to $z$, and $R_n$ be internal rays of angles $\theta_n$ in $U$ landing at $z_n$. Let $\theta$ be the limit of a convergent subsequence $\theta_{n_k}$. Then the radial rays $[0, 1]\,e^{2\pi \tb{i}\theta_{n_k}}$ converge to $[0, 1]\,e^{2\pi \tb{i}\theta}$ in $\ol{\mb{D}}$ in the sense of Hausdorff topology. By the uniform continuity of the inverse of the B\"{o}ttcher coordinate $\phi^{-1}_U$ on $\ol{\mb{D}}$, the closed rays $R_{n_k}\cup\{z_{n_k}\}$ converge to $\ol{R_U(\theta)}$ in the sense of Hausdorff topology. Since $z_{n_k}\to z$ as $k\to\infty$ and $\#\tu{Acc}(z, U)=1$, we have $R=R_U(\theta)$. It follows that $\theta_n\to \theta$ and $R_n\to R$ as $n\to \infty$ whenever $z_n\to z$.
	
	This argument shows that when $z'\in\partial U$ is close to $z$,  the internal angles (resp. internal rays), associated to $z'$ are also close to $\theta$ (resp. $R_U(\theta)$). Thus the proof is complete.
\end{proof}

\begin{proposition}\label{prop:new}
Let $f$ be a PCF and $P_0$ be the set given in Theorem \ref{thm:invariant}. Let $P$ be a finite and $f$-invariant set containing $P_0$. Let $G=(\mc{V},\mc{E})$ be an admissible graph with respect to $(f,P)$. Then for any $\epsilon>0$, there is an admissible graph $G_\epsilon=(\mc{V}_\epsilon,\mc{E}_\epsilon)$ satisfying that
\begin{itemize}
	\item $\mc{V}_\epsilon=\mc{V}$ and $G_\epsilon\subseteq \mc{N}(G,\epsilon)$;
	\item $G_\epsilon\cap (U_1\cup\cdots\cup U_m)$ is $f^n$-invariant for each sufficiently large integer $n$, where $U_1,\dotsc, U_m$  are the Fatou domains containing Fatou-type vertices in $\mc{V}$.
\end{itemize}
\end{proposition}

\begin{proof}
Let $\mc{E}_{0}$ be the collection of edges $e$ of $G$, such that
\begin{itemize}
	\item an endpoint of $e$ is a Fatou-type vertex of $G$, say $c(U)$, and
	\item the landing point of the internal ray $e\cap U$ is never iterated into $P$.
\end{itemize}
By deleting all $\tu{int}(e), e\in\mc{E}_{0}$ from $G$, we obtain a set $G_0$ (a subgraph if still connected).
Let $e$ be an edge in $\mc{E}_{0}$ with endpoints $z_i,i=1,2$. If $z_i$ is of Fatou-type, we denote by $U_i$ the Fatou domain containing $z_i$, $R_i:=e\cap U_i$, and $\xi_i$ the landing point of $R_i$.

\noindent\tb{Claim}~ \emph{For any $\epsilon>0$, there exists a regulated arc $\wt{e}$ connecting $z_1$ and $z_2$, such that \begin{enumerate}
\item $\wt{e}\subseteq \mc{N}(e,\epsilon)$ satisfies the joint condition;
\item $\wt{e}$ coincides with $e$ near $z_i$ whenever $z_i$ is of Julia-type;
\item  the landing point of $\wt{R}_i:=\wt{e}\cap U_i$ is iterated into $P$ whenever $z_i$ is of Fatou-type.
\end{enumerate}}

We first complete the proof of Proposition \ref{prop:new} under the claim. By letting $\epsilon$ be sufficiently small, the open arcs $\tu{int}(\wt{e}), e\in\mc{E}_{0},$ are mutually disjoint, and do not intersect $G_0$. Thus, the graph $$G_\epsilon:=G_0\cup~\{\wt{e}:e\in\mc{E}_{0}\}$$
has the vertex set $\mc{V}$, and is homeomorphic to $G$ under a homeomorphism $\phi:G\to G_\epsilon$ such that $\phi(e):=\wt{e}$ if $e\in\mc{E}_{0}$ and $\phi(e):=e$ otherwise. It follows immediately from the claim that $G_\epsilon$ is an admissible graph as required.

\begin{proof}[Proof of the claim]
Suppose that $z_1$ is a Fatou center. To keep the notation simple, we write $z_1,U_1,R_1$ and $\xi_1$ as $z,U,R$ and $\xi$.
We claim that $\xi\in\tu{int}(e)$ and $\#\tu{Acc}(\xi)=1\tu{ or }2$. For otherwise, either $\xi\in \mc{V}$ or $\#\tu{Acc}(\xi)\geq 3$. In both cases $\xi$ is eventually mapped into $P$, which contradicts the choice of $e$.

Let $\theta$ be the angle such that $R=R_U(\theta)$. We claim that for any $\epsilon$ there exist two angles $\theta_{\pm}$ with $\theta\in]\theta_-, \theta_+[$ such that $R_\pm:=R_U(\theta_\pm)\subseteq \mc{N}(R,\epsilon)$, and the landing points $w_{\pm}$ of $R_\pm$ are eventually iterated into $P_0$. To see this, we know that $U$ eventually falls into a Fatou domain cycle, and the graph $G$ contains the $f^p$-invariant internal ray $R_V(0)$ for each Fatou domain $V$ in the cycle, where $p$ is the period of the cycle. This follows by the choice of $P_0$ and condition (C.5) for admissible graphs. Let $d:=\tu{deg}(f^p|_V)$. Then the angles
$${i}/{d^k}, {k\geq 1, 1\leq i\leq d^k}$$
are dense in $\mb{R}/\mb{Z}$ and are eventually fixed under the iterations of $\eta\mapsto d\eta\tu{ mod }\mb{Z}$. Since the inverse of the B\"{o}ttcher coordinate $\phi_V^{-1}:\ol{\mb{D}}\to \ol{V}$ is uniformly continuous,  the rays $R_V(\frac{i}{d^k})$ can be arbitrarily close to any internal ray of $V$ in both directions. By pulling back these rays into $U$, the claim of this paragraph follows.

Without loss of generality, we may assume $w_-\neq w_+$. By Lemma \ref{lem:dusty_linking}, there is an $\Omega$-clean arc $\gamma$ relative to $U$ joining $w_-$ and $w_+$, such that $\gamma\subset\mc{N}(e,\epsilon)\cup U'$, where either $U'=\es$ or $U'\in\tu{Dom}(\xi)\sm\{U\}$ if $\#\tu{Dom}(\xi)=2$. Then $\gamma\cup R_+\cup R_-$ forms a Jordan curve. 
Let $W$ be the complementary component of this curve that contains $\tu{int}(R)$. We may assume that $z_2\notin W$ and even $z_2\notin \ol{W}$ if $z_2\in \mc{J}_f$.

Since $e$ joins $z_1$ to $z_2$, there exists the last-out place $w$ of $e$ meeting $\gamma$. Then $w$ is either a point in $\mc{J}_f$ or a Fatou center. If $w\in\tu{end}(\gamma)$, say $w=w_+$, we define $e_\epsilon:=R_+\cup[w, z_2]_e$; we are done.
Otherwise, we consider the following two cases.

Case 1: $\gamma\subseteq\mc{N}(R,\epsilon)$. Let us define
$$e_\epsilon:=R_+\cup [w_+, w]_\gamma\cup [w, z_2]_e.$$
Then $e_\epsilon$ differs from $e$ just in an $\epsilon$-neighborhood of $R$. We now verify that $e_\epsilon$ satisfies the joint condition at the point $w$ when $w\in\mc{J}_f$.

%Case 1: $\#\tu{Dom}(\xi)=1$. Then by Lemma \ref{lem:dusty_linking} and the choices of $R_{\pm}$, it holds that $\ol{W}\subseteq\mc{N}(R,\epsilon)$. Since $e$ joins $z_1$ to $z_2$, there exists the last-out place $w$ of $e$ meeting $\gamma$. Then $w$ is either a point in $\mc{J}_f$ or a Fatou center. If $w\in\{w_\pm\}$, we define $e_\epsilon:=R_w\cup [w, z_2]_e$, where $R_w\in\tu{Acc}(w, U)$ is contained in $\ol{W}$; we are done.  Otherwise, let us define
%\begin{equation}\label{eq:new-edge}
%e_\epsilon:=R_+\cup [w_+, w]_\gamma\cup [w, z_2]_e.
%\end{equation}
%It is clear that $e_\epsilon$ differs from $e$ just in an $\epsilon$-neighborhood of $R$. We now verify that $e_\epsilon$ satisfies the joint condition at the point $w$ when $w\in\mc{J}_f$.

If $w\notin\partial U$, then $w$ belongs to the interior of a clean arc. So $w$ is either buried or of type 1a or of type 2a. In the first case, the joint condition is satisfied obviously. In the last two cases, all the one or two rays in $\tu{Acc}(w)$ are contained in $\gamma$, but they are disjoint from $[w, z_2]_e$. By exchanging $w_+$ and $w_-$ if necessary, we may assume $[w_+, w]_\gamma$ approaches $w$ along a ray in $\tu{Acc}(w)$. Since the joint condition holds for $[w, z_2]_e$ at $w$, we see that $e_\epsilon$ satisfies the joint condition at $w$ as well.

If $w\in \partial U$, since $w\in \tu{int}(\gamma)$, from the definition of $\Omega$-clean arcs and the topological structure of $K_U$, there are two rays $R_w$ and $R_w'$ in $\tu{Acc}(w,U)$ such that the Jordan curve $\beta:=\ol{R_w\cup R_w'}$ separates $[w_+,w[_\gamma$ and $[w_-, w[_\gamma$. By exchanging $w_+$ and $w_-$ if necessary, we assume $\beta$ also separates $[w_+, w[_\gamma$ and $]w, z_2]_e$. Therefore, the joint condition holds for $e_\epsilon$ at the point $w$.

Case 2: $\gamma\nsubseteq\mc{N}(R, \epsilon)$. Then $\#\tu{Dom}(\xi)=2$ and $c(U')\in\gamma$ with $\tu{Dom}(\xi):=\{U, U'\}$. Assume $\tu{Acc}(\xi)=\{R,R'\}$ and $R'\subset U'\subset \Omega\in\tu{Comp}(\wh{\mb{C}}\setminus \ol{U})$. Since $e$ satisfies the joint condition at $\xi$, the arc $[\xi, z_2]_e$ approaches $\xi$ along the ray $R'$. 
By letting $w_+$ and $w_-$ be sufficiently close to $\xi$, we have 
$w=c(U')\tu{ and }\gamma\cap U'\subset \mc{N}(R',\epsilon)$. This is a consequence of Lemma \ref{lem:dusty_linking} and Lemma \ref{lem:one}. Then in this case we define 
$$e_\epsilon:=R_+\cup [w_+, c(U')]_\gamma\cup [c(U'), z_2]_e.$$ Clearly, $e_\epsilon\subseteq \mc{N}(e,\epsilon)$ and the joint condition holds for $e_\epsilon$.

Now, we have modified the edge $e$ near the endpoint $z_1$. If the other endpoint $z_2$ of $e$ is of Julia-type, then $\wt{e}:=e_\epsilon$ satisfies the required properties in the claim. Otherwise, we repeat the process above by replacing $z,U,e,R,\xi$ with $z_2,U_2,e_\epsilon,R_2,\xi_2$, and the resulting edge, denoted by $\wt{e}$, is as required.
\end{proof}
\end{proof}

\subsection{Pipes of admissible graphs}
	Let $f$ be a  PCF and $G=(\mc{V},\mc{E})$ be an admissible graph with respect to $(f, P)$. 
	
	Let $J_G$ be the union of Julia-type endpoints of $\ol{G\cap U}$ with $U$ taken over all Fatou domains whose centers are vertices of $\mc{V}$. Associated to $G$, we have a new graph $G_0=(\mc{V}_0, \mc{E}_0)$ with $G_0:=G$ and $\mc{V}_0:=\mc{V}\cup J_G$. Note that $G_0$ may not be admissible with respect to $(f, P)$, as it possibly fails condition (C.4). %Clearly $G_0$ may not be admissible for 

	We will construct a set $\mc{P}=\mc{P}(G)$, called a \emph{pipe} of $G$ (or $G_0$), which is almost a tubular neighborhood of $G$. It coincides with $G$ on the Fatou domains whose centers are in $\mc{V}$; see Figure \ref{fig:conv2}.

Let $z$ be a vertex of $G_0$. If $z$ is of Fatou-type, let $B_z=\{z\}$; otherwise, choose a closed topological disk $B_z$ around the vertex $z$ and assume $B_z$ is well-behaved: $B_z\cap e, e\in \mc{E}_0(z)$ is connected; $B_z\cap B_{z'}=\emptyset$ whenever $z\neq z'\in\mc{V}$. Here $\mc{E}_0(z)$ is  the collection of edges in $\mc{E}_0$ containing the vertex $z$.
	
Next, we assign to each edge $e$ a tubular-like neighborhood $P_e$. An edge of a regulated graph is called \emph{trivial} if it is a closed internal ray and \emph{non-trivial} otherwise. By the above assumption on $\mc{V}_0$, edges attached to Fatou-type vertices are trivial.  In this case, we define $H_e:=e\setminus\tu{int}(B_{z'}),$
where $z'$ is the Julia-type endpoint of $e$.
If $e$ is non-trivial, let $g$ be a homeomorphism, defined on a neighborhood of $e$, sending $e$ into $\mb{C}$ such that $$g(e)=[-2,2],~~ g(B_z)=\ol{\mathbb{D}(-2,1)}\textup{ and } g(B_{z'})=\ol{\mb{D}(2,1)},$$
	where $\tu{end}(e)=\{z,z'\}$.
   For a positive angle $\theta_\epsilon$ close to zero, we have four arcs
	$$e_{up}=g^{-1}\{\zeta\in\mb{C}:-2+\tu{cos}(\theta_\epsilon)\leq\tu{Re}(\zeta)\leq 2-\tu{cos}(\theta_\epsilon), \tu{Im}(\zeta)=\tu{sin}(\theta_\epsilon)\},$$ $$e_{down}=g^{-1}\{\zeta\in\mb{C}:-2+\tu{cos}(\theta_\epsilon)\leq\tu{Re}(\zeta)\leq 2-\tu{cos}(\theta_\epsilon), \tu{Im}(\zeta)=-\tu{sin}(\theta_\epsilon)\},$$
	$$e_{left}=g^{-1}\{-2+e^{\tb{i}t}:-\theta_\epsilon\leq t\leq \theta_\epsilon\},$$
	and
	$$e_{right}=g^{-1}\{2+e^{\tb{i}t}:\pi-\theta_\epsilon\leq t\leq \pi+\theta_\epsilon\}.$$
	We define $H_e$ as the closed quadrilateral surrounded by $e_{up}$, $e_{down}$, $e_{left}$ and $e_{right}$.
	One may assume that $H_e\cap H_{e'}=\es$ if $e\neq e'\in\mc{E}_0$ and $H_e\cap B_z=\emptyset$ if $e\in \mc{E}_0$ and $z\notin \tu{end}(e)$. For each $e\in\mc{E}_0$, the set $P_e:=B_z\cup H_e\cup B_{z'}$ is called a \emph{pipe} of $e$.
	Finally, the union
	$$\mc{P}=\mc{P}(G):=\bigcup_{e\in\mc{E}_0}\,P_e$$
	is called a \emph{pipe} of the graph $G$ (or $G_0$).

\begin{figure}[h]
\centering
\includegraphics[width=0.7\linewidth]{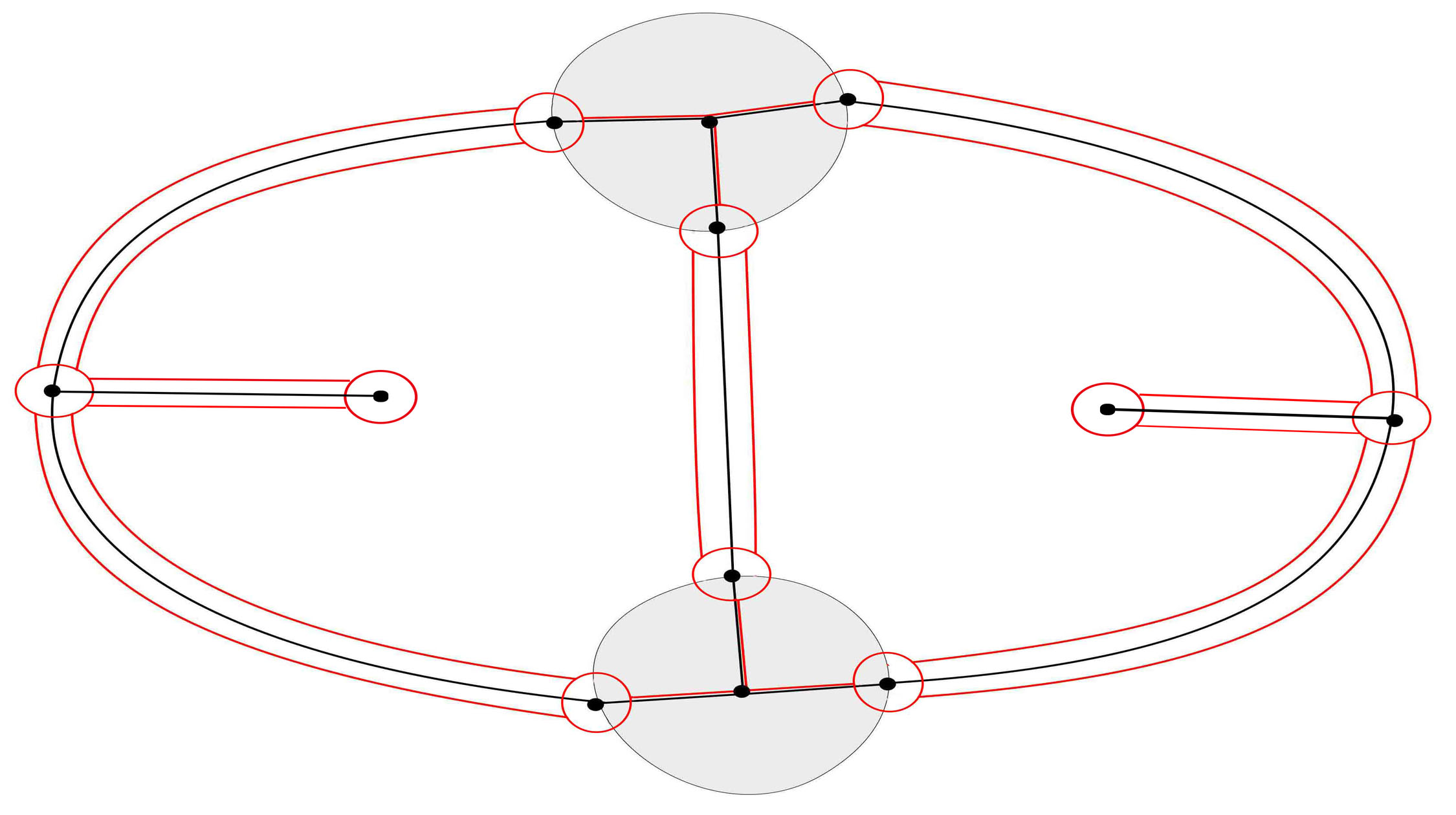}
\caption{A pipe of a graph $G$: the two shaded disks are Fatou domains; $G$ is composed of the black curves; while the area surrounded by the red curves is a pipe $P$ of $G$. It looks like a tubular neighborhood of $G$, though degenerated in the Fatou domains that contain vertices.}
\label{fig:conv2}
\end{figure}

\subsection{Proof of Theorem \ref{thm:invariant}}

\begin{proof}[Proof of Theorem \ref{thm:invariant}]
Recall that an edge of a regulated graph is called trivial if it is a closed internal ray and non-trivial otherwise.
Note that a non-trivial edge in $\mc{E}$ approaches an accessible Julia-type endpoint along either an internal ray or a special entrance by the joint condition. By adding finitely many Fatou centers into $\mc{V}$, we may assume that any non-trivial edge in $\mc{E}$ goes to an accessible Julia-type endpoint $v$ always through a special entrance. Clearly the graph $G=(\mc{V},\mc{E})$ is still admissible for $(f,P)$.

By Proposition \ref{prop:new}, we modify $G$ within its neighborhood to a new one, denoted by $G=(\mc{V},\mc{E})$ again, such that it is $f^n$-invariant for each sufficiently large integer $n$, when restricted to the union
of Fatou domains whose centers are contained in $\mc{V}$. 

Let $G_0=(\mc{V}_0,\mc{E}_0)$ be the graph induced by $G$ with $\mc{V}_0=\mc{V}\cup J_G$ and $\mc{P}_0$ be a pipe of $G_0$ such that $\mc{P}_0\subseteq \mc{N}(G_0, \epsilon)$, where $J_G$ is the union of landing points of all internal rays $R$ that $R\subset G$ and $R$ contains a Fatou-type vertex of $\mc{V}$. By the ``local invariance'' property for $G$ after the modification, we have $f^n(J_G)\subseteq P$ for each sufficiently large integer $n$. Moreover, $G_0$ has the properties that
\begin{itemize}
	\item[(i)]endpoints of non-trivial edges in $\mc{E}_0$ are of Julia-type;
	\item[(ii)] any non-trivial edge goes to an endpoint $v$ through either a special entrance or an internal ray, moreover, the latter happens only if $\tu{valence}(v, G_0)=2$ and the other edge attached to $v$ is trivial and $v\in J_G$. 
\end{itemize}
The proof is  broken up into three steps.
	
\noindent	\tb{Step 1.} Construct an isotopically $f^n$-invariant graph $G_1=(\mc{V}_1,\mc{E}_1)$ and its pipe $\mc{P}_1$. Let $\delta_0$ be a positive number such that, for every non-trivial edge $e$ in $\mc{E}_0$, it holds that
$\mc{N}(e, 4\delta_0)\subseteq \mc{P}_0$.

Now we construct an arc $l_e$ for each edge $e\in\mc{E}_0$. If $e$ is trivial, let $l_e:=e$. Otherwise, the two endpoints of $e$ are of Julia-type in $\mc{V}_0$. Note that $\mc{V}_0$ is $f^n$-invariant for each sufficiently large integer $n$.
Indeed, condition (C.4) for $G$ implies all Julia-type vertices in $\mc{V}\setminus P$ are multi-accessible and so are eventually iterated into $P_0$; and $f^{n}(J_G)\subseteq P$ as mentioned above.  Since $e$ satisfies the joint condition, by applying Proposition \ref{prop:epsilon_neighbor} to the graph $G=(\mc{V},\mc{E})$ and $e$, we obtain an arc $$l_e\subset f^{-n}(G)\cap \mc{N}(e,\delta_0)=f^{-n}(G_0)\cap \mc{N}(e,\delta_0)$$ joining $\tu{end}(e)$ for each sufficiently large integer $n$ in this case.

Consider a Julia-type vertex $v$ of $G_0$ which is not an endpoint. Then $v$ is accessible. Let $e$ be a non-trivial edge of $G_0$ with $v\in{\rm end}(e)$. We shrink $\delta_0$ such that no internal rays in $\tu{Acc}(v)$ are contained in $\mc{N}(e, \delta_0)$ except the possible one contained in $e$. If the edge $e$ approaches $v$ along a special entrance $E_e$, then $l_e$ also goes to $v$ along $E_e$. Indeed, for otherwise, $l_e$ has to intersect the interior of an internal ray bounding $E_e$, say $R$, by the choice of $\delta_0$. Since $l_e$ is regulated, we have $R\subset l_e\subset\mc{N}(e, \delta_0)$, which is impossible. If the edge $e$ approaches $v$ along an internal ray, by property (ii) for $G_0$ we have $\tu{valence}(v, G_0)=2$ and the other edge attached to $v$ is trivial, say $R'$. The choice of $\delta_0$ also gives that $R'\nsubseteq l_e$.

By the above arguments, if we take $\de_0$ sufficiently small, the interiors of these arcs $l_e,e\in\EEE_0$ are pairwise disjoint in neighborhoods of the vertices. By shrinking $\de_0$ furthermore if necessary, we get ${\rm int}(l_e),e\in\EEE_0$ are pairwise disjoint.  

%Consider a non-trivial edge $e$ with an accessible vertex $v$. By the assumption on $G$ at the beginning of the proof and the choice of $J_G$, the edge $e$ goes to $v$ along either a special entrance $E$ or an internal ray such that the latter happens only if $\tu{valence}(v)=2$ and the other edge, say $e'$, attached to $v$ is trivial. 
%If the number $\delta_0$ is so small that $R\nsubseteq\mc{N}(e,\delta_0)$ for all ray $R\nsubseteq e$ in $\tu{Acc}(v)$, then $l_e$ also goes to $z$ along $E$ in the former case and $e'\nsubseteq l_e$ in the latter case. This is because $l_e$ is regulated and $l_e\subset\mc{N}(e,\delta_0)$. Thus in a small neighborhood of $v$, the arcs $l_e, e\in\mc{E}_0$ attached to $v$ have no intersection except at the point $v$.
%It follows that globally we may assume that $\tu{int}(l_e), e\in\mc{E}_0$ are mutually disjoint by letting $\delta_0$ be small enough.

We set $G_1=\bigcup_{e\in\mc{E}_0}l_e.$ Then $G_1$ is indeed a graph homeomorphic to $G_0$. Letting $\delta_0$ be small enough, we have that $G_1$ is isotopic to $G_0$ rel. $\mc{V}_0$; refer to \cite[Proposition 11.7]{BM}.   The new vertex set $\mc{V}_1$ of $G_1$ is defined as $\mc{V}_1:=f^{-n}(\mc{V}_0)\cap G_1$. Note that each $l_e$ is made up of some edges of $G_1$, and an edge in $\mc{E}_1$ is trivial if and only if its image under $f^n$ is trivial in $\mc{E}_0$.
	
Let $\Phi_{0}:[0,1]\times G_0\to \mc{P}_0$ be an isotopy rel. $\mc{V}_0$ such that,
	for each $e\in \mc{E}_0$,
	$$\Phi_{0}(0,\cdot)=\tu{id} ,\ \ \Phi_{0}(1,e)=l_e\tu{ and } \Phi_{0}(t,a)=a~~ \forall~ t\in[0,1]\tu{ and } a\in\tu{end}(e).$$
	In particular, if $e$ is trivial, $\Phi_0(t,\cdot)|_e=\tu{id}\ \forall\, t\in[0,1]$. The isotopy $\Phi_0$ induces a homeomorphism $\phi_0:=\Phi_0(1,\cdot)$ from $G_0$ onto $G_1$ such that $\phi_0(v)=v$ for $v\in\mc{V}_0$ and $\phi_0(e)=l_e$ for $e\in\mc{E}_0$.

Besides, for any $\xi\in G_0$, the isotopy $\Phi_0$ also induces a curve $\gamma_\xi:[0,1]\ni t\mapsto\Phi_0(t,\xi).$
From the uniform continuity of $\Phi_{0}$, there is a number $\Lambda$, independent of $\xi\in G_0$, such that each curve $\gamma_\xi$ can be broken up into at most $\Lambda$ subcurves
	$\gamma_\xi^{(1)},\dotsc,\gamma_\xi^{(i)}$
	with their diameters satisfying Lemma \ref{lemma:shrinking_lemma_locally}. Let $M$ be the maximum among all diameters of these subcurves with respect to the orbifold metric. Clearly
	$\tu{dist}_{\mc{O}}(\xi, \phi_0(\xi))\leq \Lambda\cdot M.$
	
	 Now we construct a pipe of $G_1$ from $\mc{P}_0$. Note that each $\wt{e}\in\mc{E}_1$ is a lift of $e=f^n(\wt{e})$. We assume $\tu{end}(\wt{e})=\{w,w'\}$. Then $\tu{end}(e)=\{f^n(w),f^n(w')\}$. Let $(B_w,B_{w'},H_{\wt{e}})$ be the pullbacks of $(B_{f^n(w)},B_{f^n(w')},H_{e})$. We then define $P_{\wt{e}}=B_w\cup H_{\wt{e}}\cup B_{w'}$, and call it the \emph{pipe} of $\wt{e}$ (induced by $\mc{P}_0$).
	The pipe of $G_1$ (induced by $\mc{P}_0$) is defined as
	$\mc{P}_1=\bigcup_{\wt{e}\in \mc{E}_1}P_{\wt{e}}.$
	It is nested in $\mc{P}_0$ for each sufficiently large integer $n$ by Lemma \ref{lemma:shrinking_lemma}.
	
\noindent	\tb{Step 2.} Construct a sequence $G_{k+1}=(\mc{V}_{k+1},\mc{E}_{k+1}),\Phi_k,\phi_k\tu{ and } \mc{P}_{k+1}$ inductively.
	Indeed, for any $k\geq1$, let
\[\text{$\Phi_{k-1}:[0,1]\times G_{k-1}\to \mc{P}_{k-1}$ rel. $\mc{V}_{k-1}$}\]
 be an isotopy between $G_{k-1}$ and $G_k$. Since the critical values of $f^n$ are contained in $\tu{Post}(f)\subseteq \mc{V}_{k-1}$, the isotopy $\Phi_{k-1}$ can be lifted by $f^n$
to a unique isotopy
$\Phi_{k}:[0,1]\times G_k\to\mc{P}_k \text{ rel. } \mc{V}_k$
such that $$\Phi_k(0,\cdot)=\tu{id}\tu{ and }f^n\circ\Phi_k(t,\xi)=\Phi_{k-1}(t,f^n(\xi))~~\forall\, (t,\xi)\in[0,1]\times G_k.$$
We define the graph $G_{k+1}:=\Phi_k(1,G_k)$ and the homeomorphism $\phi_k:G_k\to G_{k+1}$ as $\phi_k(\cdot):=\Phi_k(1,\cdot)$. It follows immediately that $f^n(G_{k+1})\subseteq G_k$, and $f^n\circ \phi_k=\phi_{k-1}\circ f^n$ on $G_k$. The vertex set $\mc{V}_{k+1}$ of $G_{k+1}$ is chosen as $\mc{V}_{k+1}:=f^{-n}(\mc{V}_k)\cap G_{k+1}.$ Note that an edge of $G_{k+1}$ is trivial if and only if its image under $f^n$ is a trivial edge of $G_k$.
	
 For each $\xi\in G_k$, let $w:=f^{kn}(\xi)\in G_0$. Then the curves
	$\wt{\gamma}_\xi:[0,1]\ni t\mapsto \Phi_k(t,\xi)$ and $\wt{\gamma}^{(1)}_{\xi},\dotsc,\wt{\gamma}_{\xi}^{(i)}$ are the lifts of $\gamma_w$ and $\gamma_w^{(1)},\dotsc,\gamma_w^{(i)}$, respectively. We estimate the distance between $G_k$ and $G_{k+1}$. If $\xi$ is contained in a non-trivial edge of $G_k$, it follows by Lemma \ref{lemma:shrinking_lemma_locally} that
	\begin{equation}\notag
	\begin{split}
	\tu{dist}_{\mc{O}}(\xi,\phi_k(\xi))\leq &~\tu{diam}_{\mc{O}}\,\wt{\gamma}_\xi\leq  \tu{diam}_{\mc{O}}\,\wt{\gamma}_\xi^{(1)}+\cdots+ \tu{diam}_{\mc{O}}\,\wt{\gamma}_\xi^{(i)}\\
	\leq &~C\cdot\tu{diam}_{\mc{O}}\gamma_w^{(1)}/\lambda^{kn}+\cdots+C\cdot \tu{diam}_{\mc{O}}\gamma_w^{(i)}/\lambda^{kn}\\
	\leq &~C\Lambda M/\lambda^{kn}.
	\end{split}
	\end{equation}
 Otherwise, we have $\tu{dist}_{\mc{O}}(\xi,\phi_k(\xi))=0$, since $\Phi_k([0,1],\xi)=\xi$. Then the Hausdorff distance between $G_k$ and $G_{k+1}$ in the sense of orbifold metric is bounded by $C\Lambda M/\lambda^{kn}$.
	
	It remains to construct a pipe $\mc{P}_{k+1}$ of $G_{k+1}$. Each edge $\wt{e}\in\mc{E}_{k+1}$ is sent to $e:=f^{n}(\wt{e})\in\mc{E}_k$. We assume $\tu{end}(\wt{e}):=\{w,w'\}$. Then $(B_w,B_{w'},H_{\wt{e}})$ are the pullbacks of $(B_{f^n(w)},B_{f^n(w')},H_e)$ under $f^n$. Let $P_{\wt{e}}=B_{w}\cup H_{\wt{e}}\cup B_{w'}$ be a pipe of $\wt{e}$. Then we define
	$\mc{P}_{k+1}:=\bigcup_{\wt{e}\in\mc{E}_{k+1}}P_{\wt{e}}$.
    Note that $\mc{P}_{k+1}$ is nested in $\mc{P}_k$ by induction.
    	
	In summary, this step is devoted to producing a sequence $(G_{k+1}, \mc{P}_{k+1}, \phi_k)$ satisfying that
	\begin{itemize}
		\item[(1)] $f^n(G_{k+1})\subseteq G_k$;
		\item[(2)] $G_{k+1}\simeq G_k\tu{ rel. }\mc{V}_k$;
		\item[(3)] $\tu{dist}_{\mc{O}}(\xi,\phi_k(\xi))\leq C\Lambda M/\lambda^{kn}\ \ \forall\,\xi\in G_k$;
		\item[(4)] $G_{k+1}\subseteq\mc{P}_{k+1}\subseteq \mc{P}_k$; and
\item[(5)]  $\max\,\{\tu{diam}\,P_{\wt{e}}:\wt{e}\in \mc{E}_k \text{ non-trivial}\}\to0$ as $k\to\infty$ by Lemma \ref{lemma:shrinking_lemma}.
	\end{itemize}

\noindent	\tb{Step 3.} The convergence of the sequence $(G_k,\mc{P}_k)$. For each $k\geq0$, there is a homeomorphism $$h_k=\phi_{k}\circ\cdots\circ\phi_0: G_0\to G_k$$ keeping $\mc{V}_0$ fixed.
Actually, the sequence $h_k$ uniformly converges to a map $h_\infty$. To see this, for integers $m> l\geq 1$ and $\xi\in G_0$, we have	
	\begin{equation}\notag
	\begin{split}
	\tu{dist}_{\mc{O}}(h_l(\xi),h_m(\xi))
	\leq&~\tu{dist}_\mc{O}(h_l(\xi),h_{l+1}(\xi))+\cdots+\tu{dist}_{\mc{O}}(h_{m-1}(\xi),h_m(\xi))\\
=&~\tu{dist}_\mc{O}(h_l(\xi),\phi_{l+1}\circ h_l(\xi))+\cdots+\tu{dist}_{\mc{O}}(h_{m-1}(\xi),\phi_m\circ h_{m-1}(\xi))\\
	\leq &~ C\Lambda M\left(\frac{1}{\lambda^{(l+1)n}}+\cdots+\frac{1}{\lambda^{mn}}\right)(\text{by point (3) in Step 2})\\
	\leq &~ \frac{C\Lambda M}{\lambda^{ln}(\lambda^n-1)}.
	\end{split}
	\end{equation}
Let $G_\infty:=h_{\infty}(G_0)$. Then, by points (1) and (4) in Step 2, we have
$$\mc{V}_0\subseteq G_\infty \subseteq \mc{P}_\infty:=\bigcap_{k\geq 0}\mc{P}_k \text{ and }f^n(G_\infty)\subseteq G_\infty.$$
Next, we show that $G_\infty$ is indeed a graph. Consider an edge $e\in G_0$. The arc $h_k(e)$ is composed of some edges, say $e_1,\dotsc, e_i$ in $\mc{E}_{k}$. Let
$L_k(e):=P_{e_1}\cup\cdots\cup P_{e_i}.$ Clearly $\mc{P}_k=\bigcup_{e\in\mc{E}_0}L_k(e).$
On the other hand, from the iterated construction of $G_k$ in Step 2, it holds that
$$h_{k+1}(e)\subseteq L_{k+1}(e)\subseteq L_{k}(e).$$
Let $L_{\infty}(e):=\bigcap_{k\geq 1}L_k(e)$. It then follows that
$h_\infty(e)\subseteq L_\infty(e).$
	
We claim that $h_\infty(e)=L_\infty(e)$ and it is an arc joining the two endpoints of $e$.
	From the continuity of $h_\infty$, the image $h_\infty(e)$ is an arcwise-connected continuum containing $\tu{end}(e)$. Then one can derive an arc $\gamma_e$ in $h_\infty(e)$ joining the two endpoints of $e$. We claim that $L_\infty(e)=\gamma_e$. For otherwise, one can find a point $\zeta\in L_\infty(e)\setminus \gamma_e$. For all $k\geq 0$, there exists a non-trivial edge $e_k\in\mc{E}_k$ such that the pipe $P_{e_k}$ of $e_k$ contains $\zeta$. Note that $\g_e$ lies in $L_k(e)$ and $P_{e_k}$ disconnects $L_k(e)$. Thus $\gamma_e$ has to pass through the ``passage" $P_{e_k}$ to connect the two components of $L_k(e)\setminus P_{e_k}$. We thus have the estimate
	$$\tu{dist}\,(\zeta,\gamma_e)\leq\tu{dist}\,(\zeta,\gamma_e\cap P_{e_k})\leq \tu{diam}\,P_{e_k}.$$
	As $\tu{diam}\,P_{e_k}$ tends to zero by point (5) in Step 2, it follows that $\tu{dist}(\zeta,\gamma_e)=0$, a contradiction.% Since $h_\infty(e)$ is contained in the $\epsilon$-neighborhood of $e$, it is disjoint from the points in $\mc{V}_0$ except $z\tu{ and }z'$. Combining the fact that $h_k(e)\sim e$ rel. $\mc{V}_0$, we see that $h_\infty(e)$ is isotopic to $e$ rel. $\mc{V}_0$. The claim is then proved.

We are left to check that the interiors of arcs $h_\infty(e)$ and $h_\infty(e')$ are  disjoint for $e\neq e'\in\mc{E}_0$. To see this, observe that, for any $k$, the sets $L_k(e)$ and $L_k(e')$ defined above can only possibly overlap on $B_{z,k}$ and $B_{z',k}$ with $\tu{end}(e)=\{z,z'\}$; while the disks $B_{z,k}$ and $B_{z',k}$ are shrinking to the points $z$ and $z'$, respectively,  as $k\to \infty$. So $$h_\infty(e)\cap h_\infty(e')=L_\infty(e)\cap L_\infty(e')\subseteq \{z,z'\}.$$
 Hence $G_\infty$ is a graph homeomorphic to $G_0$. Moreover, we have $G_\infty\subseteq\mc{N}(G_0,\epsilon)$. Let $\mc{V}_\infty:=\mc{V}$. By the same arguments on \cite[Proposition 11.7]{BM}, we conclude that the two graphs $G_\infty$ and $G$ are isotopic rel. $\mc{V}$. The proof of Theorem \ref{thm:invariant} is complete. %Since every edge $e$ of $G_0$ is isotopic to the edge $h_\infty(e)$ of $G_\infty$ rel. $\mc{V}_0$,  the graphs $G$ and $G_\infty$ are isotopic rel. $\mc{V}_0$; see \cite[Proposition 11.7]{BM}.% \cite[A.5 Theorem, p. 411]{Bu}.
\end{proof}
%\begin{proof}[Proof of Theorem \ref{thm:main}] It follows immediately from Propositions \ref{prop:transfer} and \ref{prop:invariant}.	
%\end{proof}

%%%%%%%%%%%%%%%%%%%%%%%%%%%%%%%%%%%%%%%%%%%%%%%
%%%%%%%%%%%%%%%%%%%%%%%%%%%%%%%%%%%%%%%%%%%%%%%

\end{document}